\documentclass[12pt]{amsart}
\usepackage{graphicx}
\usepackage{mathrsfs}
\usepackage{epsfig}
\usepackage{amsmath}
\usepackage{amsfonts}
\usepackage{amssymb}
\usepackage{amssymb,amscd}
\usepackage{tikz}
\usepackage{verbatim}
\usepackage{booktabs}
\usepackage[all]{xy}
\usepackage{mathtools}
 \usepackage{makecell}

\usepackage{overpic}
\usepackage{hyperref}

\usepackage[scr=boondox,  cal=esstix]{mathalpha}
\usepackage{adjustbox}

\usepackage{geometry}
\usepackage{cite}
\usepackage{hyperref}

\usepackage{amsthm}
\usepackage{amssymb}
\usepackage{amsfonts}
\usepackage{amsmath}
\usepackage{mathtools}
\usepackage[shortlabels]{enumitem}
\usepackage{adjustbox}
\usepackage{url} % insert url via \url{}
\usepackage[T1]{fontenc}
\usepackage{float}

\usepackage[T1]{fontenc}
\usepackage[utf8]{inputenc}
\usepackage[justification=centering]{caption}

\DeclareFontFamily{U}{mathx}{}
\DeclareFontShape{U}{mathx}{m}{n}{<-> mathx10}{}
\DeclareSymbolFont{mathx}{U}{mathx}{m}{n}
\DeclareMathAccent{\widehat}{0}{mathx}{"70}
\DeclareMathAccent{\widecheck}{0}{mathx}{"71}

\newtheorem{thm}{Theorem}[section]

\newtheorem{lem}[thm]{Lemma}
\newtheorem{prop}[thm]{Proposition}

\theoremstyle{definition}
\newtheorem{defn}[thm]{Definition}
\theoremstyle{remark}
\newtheorem{rem}[thm]{Remark}
\numberwithin{equation}{section}
\newcommand{\HdC}{\mathbf H^{2}_{\mathbb C}}

\renewcommand{\Re}{\mbox{\rm Re}\,}
\renewcommand{\Im}{\mbox{\rm Im}\,}
\renewcommand{\leq}{\leqslant}
\renewcommand{\geq}{\geqslant}

\newcommand{\hc}{\mathbf{H}^2_{\mathbb{C}}}

\begin{document}

	\title[]{Spherical CR uniformizations of a sequence of hyperbolic 3-manifolds}
	\author[J. Ma]{Jiming Ma}
	\address{School of Mathematical Sciences, Fudan University, Shanghai, China}
	\email{majiming@fudan.edu.cn}
	
	\author[B. Xie]{Baohua Xie}
	\address{School of Mathematics, Hunan University, Changsha, China}
	\address{Greater Bay Area Institute for Innovation, Hunan University, Guangzhou, China}
	\email{xiexbh@hnu.edu.cn}

	
	
	\thanks{J. Ma was  supported by National Natural Science Foundation of China (No.12171092). B. Xie was supported  by Guangdong Basic and Applied Basic Research Foundation (No. 2025A1515011486) and National Natural Science Foundation of China (No.12271148).  }
	
    \begin{abstract}
    	   	   	
Let $s782$ be the 2-cusped hyperbolic 3-manifold in the SnapPy census. 
Its spherical CR uniformization was established in \cite{JWX2023} using the Ford domain of the complex hyperbolic triangle group $\Delta_{4,4,\infty;\infty}$. 
By comparing the combinatorial structures of the Ford domain of $\Delta_{4,4,\infty;\infty}$ and the Dirichlet domain of $\Delta_{4,4,n;\infty}$, 
we prove that for each $n \geq 5$, the Dehn filling of $s782$ along the slope $(n-1)\mathcal{m}_1 + \mathcal{l}_1$ on its second cusp admits a spherical CR uniformization, 
where $(\mathcal{m}_1, \mathcal{l}_1)$ denotes the meridian-longitude system of a cusp in SnapPy notation.

    \end{abstract}

\maketitle
\tableofcontents

\section{Introduction}\label{section:intro}   

Thurston's work on $3$-manifolds has underscored the fundamental role of geometry in the study of their topology. This geometric viewpoint also provides the framework for our discussion of spherical CR-structures below.

A \emph{spherical CR-structure} on a smooth $3$-manifold $M$ is a maximal collection
 of distinguished charts modeled on the boundary $\partial \mathbf{H}^2_{\mathbb{C}}$
  of complex hyperbolic space $\mathbf{H}^2_{\mathbb{C}}$, with transition maps given
   by restrictions of elements of $\mathbf{PU}(2,1)$. Equivalently, 
   a spherical CR-structure is a $(G,X)$-structure with $G=\mathbf{PU}(2,1)$ 
   and $X=\mathbb{S}^{3}$. 

In contrast to other geometric structures on $3$-manifolds, such as hyperbolic or
 Seifert structures, relatively few examples of spherical CR-structures are known.
  Determining whether a given $3$-manifold admits a spherical CR-structure is, 
  in general, a subtle and difficult problem. For instance, $3$-manifolds with
   $Nil^{3}$-geometry naturally admit such structures, whereas 
   Goldman~\cite{Goldman:1983} showed that no closed $3$-manifold with Euclidean or $Sol^{3}$-geometry admits a spherical CR-structure.

Among spherical CR-structures, we focus on the important class of uniformizable structures. 
A spherical CR-structure on a 3-manifold $M$ is called \emph{uniformizable} if it can be realized as 
$M = \Gamma \backslash \Omega_{\Gamma}$, where $\Gamma$ is a discrete subgroup of $\mathbf{PU}(2,1)$, 
$\Omega_{\Gamma} \subset \partial \mathbf{H}^2_{\mathbb{C}} = \mathbb{S}^3$ is the domain of discontinuity of the action of $\Gamma$ 
on $\partial \mathbf{H}^2_{\mathbb{C}}$, and $M$ is known as the \emph{3-manifold at infinity} of $\Gamma$. The construction of discrete subgroups of $\mathbf{PU}(2,1)$ thus provides a direct method for producing spherical CR-structures on 3-manifolds.

In general, establishing the discreteness of a subgroup $\Gamma \subset \mathbf{PU}(2,1)$ is a formidable task. Even when discreteness is known for arithmetic reasons, understanding the corresponding 3-manifold at infinity of $\Gamma$ remains challenging. There exist discrete and faithful representations of surface groups into $\mathbf{PU}(2,1)$, such as $\mathbb{R}$-Fuchsian and $\mathbb{C}$-Fuchsian surface groups. Consequently, constructing Seifert 3-manifolds with uniformizable spherical CR-structures is relatively straightforward. However, the situation for hyperbolic 3-manifolds is considerably more difficult. The first example of a cusped hyperbolic 3-manifold admitting a uniformizable spherical CR-structure 
was given by Schwartz \cite{Schwartz:2001acta}, who showed that the Whitehead link complement admits 
such a structure. A second explicit example was provided by Deraux and Falbel \cite{Deraux:2015}, who 
established this property for the figure-8 knot complement. Additional explicit, though ad hoc, 
examples appear in \cite{Deraux:2015scr},\cite{Deraux:2016}, \cite{JWX2023}, 
\cite{MaXie2020}, and \cite{MaXie2021}.

Schwartz showed that if a cusped $3$-manifold $M$ admits a uniformizable 
spherical CR-structure satisfying certain additional hypotheses, then infinitely 
many Dehn fillings of $M$ also admit uniformizable spherical 
CR-structures~\cite{Schwartz:2007}. However, Schwartz's argument relies
 on a compactness argument and is therefore not effective; in particular, 
 it does not identify which specific Dehn fillings of $M$ admit such structures. 
 To the authors' knowledge, the work of Acosta~\cite{Acosta2019} is the only
  result to date that explicitly constructs infinitely many Dehn fillings of a
   concrete cusped $3$-manifold. Namely, explicitly infinitely many Dehn fillings of the Whitehead link complement admit
    uniformizable spherical CR-structures. While this represents a significant
	 breakthrough, the broader question of whether other nontrivial cusped $3$-manifolds support infinitely many such Dehn fillings has remained open. The principal contribution of the present work is to demonstrate explicitly that infinitely many Dehn fillings of the cusped $3$-manifold $s782$ admit uniformizable spherical CR-structures. This not only enlarges the class of cusped $3$-manifolds known to admit infinitely many spherical CR-structured Dehn fillings,  but also confirms that Acosta's example is not an isolated phenomenon confined to the Whitehead link complement.

Let $T_{p,q,r}$ denote the abstract $(p,q,r)$ reflection triangle group with presentation
\[
\langle \sigma_1, \sigma_2, \sigma_3 \mid \sigma^2_1 = \sigma^2_2 = \sigma^2_3 = (\sigma_2\sigma_3)^p = (\sigma_3\sigma_1)^q = (\sigma_1\sigma_2)^r = \mathrm{id} \rangle,
\]
where $p, q, r$ are positive integers or $\infty$ satisfying $1/p + 1/q + 1/r < 1$. We assume $p \leq q \leq r$. 
If any of $p, q, r$ equals $\infty$, the corresponding relation is omitted.
A \emph{complex hyperbolic $(p,q,r)$ triangle group} is a representation $\rho: T_{p,q,r} \to \mathbf{PU}(2,1)$ 
in which the generators fix complex lines. We write $I_i = \rho(\sigma_i)$ for $i = 1, 2, 3$. 
It is well known that the deformation space of $(p,q,r)$ complex hyperbolic triangle groups has real dimension one when $3 \leq p \leq q \leq r$.
We denote by $\Delta_{p,q,r;n}$ the image of a representation of $T_{p,q,r}$ into $\mathbf{PU}(2,1)$ 
for which the element $I_1I_3I_2I_3$ (in the image group) has order $n$. 
Similarly, $\Delta_{p,q,r;\infty}$ denotes the image group when $I_1I_3I_2I_3$ is parabolic.

Let $\Sigma = \langle I_1I_2, I_2I_3 \rangle$ be the even subgroup of the complex hyperbolic triangle group $\Delta_{4,4,\infty;\infty}$. 
Jiang, Wang and Xie \cite{JWX2023} established that the 3-manifold $\mathscr{M}_{4,4,\infty;\infty}$ at infinity of $\Sigma$ corresponds to the 3-manifold $s782$ in the SnapPy census. The fundamental group of $s782$ has the presentation:
\[
\pi_1(s782) = \left\langle a, b, c \ \middle|\ 
\begin{array}{l}
	a^2 c b^4 c, \,
	a b c a^{-1} b^{-1} c^{-1}
\end{array} \right\rangle.
\]
The manifold $s782$ has two cusps. In SnapPy notation, the first cusp $C_0$ has meridian-longitude system
\[
(\mathcal{m}_0, \mathcal{l}_0) = (c^{-1}b^{-1}a^{-2},\ c^{-1}b^{-1}ca),
\]
and the second cusp $C_1$ has meridian-longitude system
\[
(\mathcal{m}_1, \mathcal{l}_1) = (bc,\ ba).
\]
We note that in SnapPy's convention, the $k$-th cusp of a cusped 3-manifold is labeled $C_{k-1}$ rather than $C_k$.
Let $\rho: T_{4,4,n} \to \mathbf{PU}(2,1)$ be a representation such that each $\rho(\sigma_i) = I_i$ is a complex reflection 
fixing a complex line, and $I_1 I_3 I_2 I_3$ is parabolic. We denote the image group by $\Delta_{4,4,n;\infty}$.
Our main result is the following:
\begin{thm} \label{thm:44np}
	For each $n \geq 5$, let $\rho: T_{4,4,n} \to \mathbf{PU}(2,1)$ be the representation with image group $\Delta_{4,4,n;\infty}$ as above. Then:
	\begin{enumerate}[(i)]
		\item \label{item:44np1} $\rho$ is a discrete embedding;
		\item \label{item:44np2} The 3-manifold at infinity of the even subgroup $\langle I_1I_2, I_2I_3\rangle$ of $\Delta_{4,4,n;\infty}$ 
		is the 1-cusped hyperbolic 3-manifold $\mathscr{M}_{4,4,n;\infty}$, which is obtained from the 2-cusped 3-manifold 
		$s782$ by Dehn filling on the second cusp along the slope $(n-1)\mathcal{m}_1 + \mathcal{l}_1$.
	\end{enumerate}
\end{thm}

We note that Theorem~\ref{thm:44np} \ref{item:44np1} confirms a special case of a conjecture 
of Schwartz \cite{Schwartz:2002icm}. 
Moreover, Theorem~\ref{thm:44np} \ref{item:44np2}  can be viewed as a complex hyperbolic analogue of Thurston's hyperbolic Dehn surgery theorem: 
it shows that certain explicit Dehn fillings of uniformizable spherical CR 3-manifolds yield new uniformizable spherical CR 3-manifolds 
\cite{Schwartz:2007, Thurston:1979}.

Our proof of Theorem~\ref{thm:44np} is inspired by the work of Schwartz~\cite{Schwartz:2007} 
and Acosta~\cite{Acosta2019}. The key insight is that for each $n \geq 5$, the Dirichlet domain
 $\mathcal{D}$ of $\Delta_{4,4,n;\infty}$ and the Ford domain $\mathcal{F}$ of $\Delta_{4,4,\infty;\infty}$ share the same local combinatorics and topology. Hence, $\mathscr{M}_{4,4,n; \infty}$ is obtained by Dehn filling one cusp of $\mathscr{M}_{4,4,\infty; \infty}$. However, identifying the specific Dehn filling slope in Theorem~\ref{thm:44np} requires additional analysis. To clarify the relationship between these groups, we compare in Table~\ref{table:notation} the notations and methods used to identify the 3-manifolds at infinity of $\Delta_{4,4,\infty;\infty}$ and $\Delta_{4,4,n;\infty}$. These notations will be defined in Sections~\ref{section:forddomain}, \ref{section:dirichlet} and ~\ref{section:proof}, respectively.

\begin{table}[!htbp]
	\caption{Sketch of the difference between the 3-manifolds at infinity of 
$\Delta_{4,4,\infty;\infty}$ and $\Delta_{4,4,n;\infty}$. 
Here, $W_\mathcal{F} \cong \mathbb{T}^2 \times [1,\infty)$ while 
$W_\mathcal{D} \cong \mathbb{D}^2 \times \mathbb{S}^1$, 
so the difference corresponds to a Dehn filling.}
	\centering
	\renewcommand{\arraystretch}{2}
	\begin{tabular}{|c|c|c|}
		\hline
		\textbf{$\Delta_{4,4,\infty;\infty}$} & 
		\textbf{$\Delta_{4,4,n;\infty}$} & \textbf{remark} \\
		\hline 
		Ford domain $\mathcal{F}$ & Dirichlet domain $\mathcal{D}$ & 4-dimension \\
		\hline 
		$\mathcal{F}_{\infty}=\mathcal{F} \cap \partial \hc$ & $\mathcal{D}_{\infty}=\mathcal{D} \cap \partial \hc$ & 3-dimension \\
		\hline
		{$\!\begin{aligned}
				T_\mathcal{F}&=\mathcal{F}_{\infty}/\langle I_1I_2\rangle\\ 
				&\cong \mathbb{T}^2\times[0,\infty)\\
				&=V_\mathcal{F}\cup W_\mathcal{F}
			\end{aligned}$} 
		& {$\!\begin{aligned}
				T_\mathcal{D}&=\mathcal{D}_{\infty}/\langle I_1I_2\rangle\\
				&\cong \mathbb{D}^2 \times \mathbb{S}^1\\
				&=V_\mathcal{D}\cup W_\mathcal{D}
			\end{aligned}$}
		& {$\!\begin{aligned}
				V_\mathcal{F} &\cong V_\mathcal{D} \cong \mathbb{T}^2 \times [0,1], \\
				W_\mathcal{F} &\cong \mathbb{T}^2 \times [1,\infty), \\
				W_\mathcal{D} &\cong \mathbb{D}^2 \times \mathbb{S}^1
			\end{aligned}$} \\
		\hline
		$\partial T_\mathcal{F}=\partial _1 V_\mathcal{F}$ & $\partial T_\mathcal{D}= \partial_1 V_\mathcal{D}$ & the same side-pairing pattern \\
		\hline
		$\mathscr{M}_{4,4,\infty; \infty}=T_\mathcal{F}/\sim$ & $\mathscr{M}_{4,4,n; \infty}=T_\mathcal{D}/\sim$ & the difference is a Dehn filling \\
		\hline
	\end{tabular}
	\label{table:notation}
\end{table}

%The intersection of balanced pairs of bisectors in Lemma \ref{lem1:bisector-balanced} (Section \ref{section:dirichlet}) is also one of the new features of this paper, which allows us to rigorously prove the local combinatorics of the Dirichlet domain of $\Delta_{4,4,n;\infty}$.

\textbf{Outline of the paper}: Section~\ref{sec:Preliminary} reviews the necessary background on the complex hyperbolic plane. Section~\ref{subsection:1-dim} presents a 1-dimensional moduli space of complex hyperbolic triangle groups, which contains both $\Delta_{4,4,\infty;\infty}$ and $\Delta_{4,4,n;\infty}$ for each $n \geq 5$.
In Section~\ref{section:forddomain}, we summarize the combinatorial structure of the Ford domain for $\Delta_{4,4,\infty;\infty}$ following \cite{JWX2023}. We also outline the proof of an isomorphism between the fundamental group of the 3-manifold at infinity of $\Delta_{4,4,\infty;\infty}$ and $\pi_1(s782)$. This isomorphism is re-derived explicitly, as we require a precise concordance between the Dirichlet domain of $\Delta_{4,4,n;\infty}$ and the Ford domain of $\Delta_{4,4,\infty;\infty}$.
Section~\ref{section:dirichlet} forms the core of the paper, where we show that the Ford domain of $\Delta_{4,4,\infty;\infty}$ and the Dirichlet domain of $\Delta_{4,4,n;\infty}$ share the same local combinatorics and topology.
The proof of Theorem~\ref{thm:44np} is completed in Section~\ref{section:proof}.
Finally, Section~\ref{section:group} provides a direct computation of the fundamental group of the 3-manifold at infinity of $\Delta_{4,4,5;\infty}$, serving as an independent verification of part of Theorem~\ref{thm:44np} \ref{item:44np2}.

\section{Preliminary}\label{sec:Preliminary}
In this section, we review the basic geometry of the complex hyperbolic plane. 
For further details, we refer the reader to \cite{Go:1999}, \cite{Parker:2003} and \cite{Mostow:1980}.

Let $\mathbb{C}^{2,1}$ denote the complex vector space $\mathbb{C}^{3}$ equipped with the canonical Hermitian form 
$\langle \cdot, \cdot \rangle$ of signature~$(2,1)$.  
Consider the following subsets of $\mathbb{C}^{3}$:
\begin{align*}
	V_{-} &= \{\, \mathbf{z} \in \mathbb{C}^{2,1} \setminus \{0\} \mid \langle \mathbf{z}, \mathbf{z} \rangle < 0 \,\}, \\
	V_{0} &= \{\, \mathbf{z} \in \mathbb{C}^{2,1} \setminus \{0\} \mid \langle \mathbf{z}, \mathbf{z} \rangle = 0 \,\}, \\
	V_{+} &= \{\, \mathbf{z} \in \mathbb{C}^{2,1} \setminus \{0\} \mid \langle \mathbf{z}, \mathbf{z} \rangle > 0 \,\}.
\end{align*}

Let $\mathbb{P} : \mathbb{C}^{2,1} \setminus \{0\} \to \mathbb{CP}^{2}$ be the canonical projectivization.  
The complex hyperbolic plane $\mathbf{H}^2_{\mathbb{C}}$ is defined as $\mathbb{P}(V_{-})$, and its ideal boundary is given by $\partial \mathbf{H}^2_{\mathbb{C}} = \mathbb{P}(V_{0})$.

Let $\mathbf{U}(2,1)$ denote the group of linear automorphisms of $\mathbb{C}^{2,1}$ preserving the Hermitian form 
$\langle \cdot, \cdot \rangle$, and let $\mathbf{I}$ denote the identity matrix in~$\mathbf{U}(2,1)$.  
The \emph{Bergman metric} $\rho$ on $\mathbf{H}^2_{\mathbb{C}}$ is defined by
\begin{equation*}
	\cosh^{2}\!\left( \frac{\rho(p,q)}{2} \right)
	= \frac{\langle \mathbf{p}, \mathbf{q} \rangle \, \langle \mathbf{q}, \mathbf{p} \rangle}
	{\langle \mathbf{p}, \mathbf{p} \rangle \, \langle \mathbf{q}, \mathbf{q} \rangle},
\end{equation*}
where $\mathbf{p}$ and $\mathbf{q}$ are lifts of $p,q \in \mathbf{H}^2_{\mathbb{C}}$, respectively.

The image $\mathbf{PU}(2,1)$ of $\mathbf{U}(2,1)$ in $\mathbf{PGL}(\mathbb{C}^{2,1})$ is the full group of biholomorphic isometries of $\mathbf{H}^2_{\mathbb{C}}$.  
We will also frequently use the subgroup $\mathbf{SU}(2,1) \subset \mathbf{U}(2,1)$ consisting of matrices of determinant~$1$.  
It is connected and is a triple cover of~$\mathbf{PU}(2,1)$:
\[
\mathbf{PU}(2,1) = \mathbf{SU}(2,1) \big/ \{\mathbf{I}, \,\omega \mathbf{I}, \,\omega^{2} \mathbf{I}\},
\]
where $\omega = \frac{-1 + i\sqrt{3}}{2}$ is a primitive cube root of unity.

\subsection{Two models of the complex hyperbolic plane}

Among the various models of \(\hc\), we describe the following two models, which will be most useful for our purposes.

\subsubsection{The Ball Model}
If the Hermitian form \(\langle \cdot, \cdot \rangle\) has matrix
\[
J_1 = \begin{pmatrix}
    1 & 0 & 0 \\
    0 & 1 & 0 \\
    0 & 0 & -1
\end{pmatrix},
\]
then we obtain the \emph{ball model} of the complex hyperbolic plane. In this model, \(\hc = \mathbb{P}(V_{-})\) is entirely contained in the affine chart \(z_3 = 1\) of \(\mathbb{CP}^2\), and can be identified with the unit ball in \(\mathbb{C}^2\):
\[
\hc = \{ (z_1, z_2) \in \mathbb{C}^2 \mid |z_1|^2 + |z_2|^2 < 1 \}.
\]
We observe that \(\hc\) corresponds to the ball \(\mathbb{B}^4\) in \(\mathbb{C}^2\) and \(\partial\hc\) to the sphere \(\mathbb{S}^3\). The ball model for the complex hyperbolic plane is the direct analogue of the Poincar\'e  disk model for the complex hyperbolic line.

\subsubsection{The Siegel Model}
If the Hermitian form \(\langle \cdot, \cdot \rangle\) has matrix
\[
J_2 = \begin{pmatrix}
    0 & 0 & 1 \\
    0 & 1 & 0 \\
    1 & 0 & 0
\end{pmatrix},
\]
then we obtain the \emph{Siegel model} of the complex hyperbolic plane. Note that the matrices \(J_1\) and \(J_2\) representing the Hermitian form are conjugate via the Cayley matrix
\[
\frac{1}{\sqrt{2}} \begin{pmatrix}
    1 & 0 & 1 \\
    0 & \sqrt{2} & 0 \\
    1 & 0 & -1
\end{pmatrix}.
\]

In the Siegel model, \(\hc = \mathbb{P}(V_{-})\) is entirely contained in the affine chart \(z_3 = 1\) of \(\mathbb{CP}^2\), and can be represented as the following domain in \(\mathbb{C}^2\):
\[
\hc = \left\{ (z_1, z_2) \in \mathbb{C}^2 \mid |z_1|^2 + 2\Re(z_2) < 0 \right\}.
\]

The boundary of \(\hc\) is given by
\[
\partial\hc = \left\{ \left(-\frac{1}{2}(|z|^2 + it), z, 1\right)^{\top} \mid (z,t) \in \mathbb{C} \times \mathbb{R} \right\}
\cup \{(1,0,0)^{\top}\}.
\]
This model serves as an analogue for the complex hyperbolic plane of the upper half-plane model for the complex 
hyperbolic line. We can thus identify \(\partial\hc\) with \(\mathbb{C} \times \mathbb{R} \cup \{\infty\}\), 
where the point at infinity corresponds to $(1,0,0)^{\top}$.
We denote by \(\mathbf{q}_{\infty}\) the lift of this point at infinity in \(\mathbb{C}^{2,1}\), 
i.e. \(\mathbf{q}_{\infty} = (1,0,0)^{\top}\).
Removing the point at infinity, we obtain the Heisenberg group, defined as \(\mathbb{C} \times \mathbb{R}\) 
equipped with the group law
\[
(z_1, t_1) \star (z_2, t_2) = (z_1 + z_2, t_1 + t_2 + 2\Im(z_1\bar{z}_2)).
\]

Let \(\mathbf{m}\) be a vector in \(V_{+}\) and \(\mathbf{m}^\perp\) its orthogonal complement with respect to the
 Hermitian form \(\langle \cdot, \cdot \rangle\). A complex reflection  with respect to \(\mathbf{m}^\perp\) is the linear automorphism of \(\mathbb{C}^{2,1}\) defined by:
\[
\mathbf{z} \mapsto -\mathbf{z} + 2 \frac{\langle\mathbf{z}, \mathbf{m} \rangle}{\langle \mathbf{m}, \mathbf{m} \rangle} \mathbf{m}.
\]
This defines an involutory element of \(\mathbf{SU}(2,1)\) that fixes \(\mathbf{m}^\perp\) pointwise.

\subsection{Isometries of complex hyperbolic plane}

The isometries of \( \hc \) are classified into three types. If \( h \) is an isometry of \( \hc \),
 we say that:
\begin{itemize}
	\item \( h \) is \emph{elliptic} if it has a fixed point in \( \hc \);
	\item \( h \) is \emph{loxodromic} if it has exactly two fixed points in \( \partial \hc \);
	\item \( h \) is \emph{parabolic} if it has exactly one fixed point in \( \partial \hc \).
\end{itemize}

Let \( h \) be an elliptic isometry of \(\hc \), and let \( \mathbf{h} \) be a lift of \( h \) in \( \mathbf{SU}(2, 1) \).
 We say that:
\begin{itemize}
	\item \( h \) is \emph{regular elliptic} if  $h$  has three distinct eigenvalues, each of modulus 1;
	\item \( h \) is \emph{special elliptic} if  $h$ has two equal eigenvalues.
\end{itemize}

\subsection{Totally geodesic 2-dimensional subspaces of complex hyperbolic plane}

In \(\mathbf{H}^2_\mathbb{C}\), there are exactly two conjugacy classes under the action of 
\(\mathbf{PU}(2,1)\) of totally geodesic real surfaces of dimension 2: the \(\mathbb{R}\)-planes 
and the \(\mathbb{C}\)-planes.

Let \( \mathbf{H}^2_\mathbb{R} = \mathbb{P}(\{z \in \mathbb{R}^{2,1} : \langle z, z \rangle < 0\}) \) and
 \( \mathbf{H}^1_\mathbb{C} = \mathbb{P}(\{z \in \mathbb{C}^{1,1} : \langle z, z \rangle < 0\}) \) be 
 the standard models of the real hyperbolic plane and complex hyperbolic line, respectively, embedded 
 in \(\mathbf{H}^2_\mathbb{C}\) in the natural way.
An \emph{\(\mathbb{R}\)-plane} in \(\mathbf{H}^2_\mathbb{C}\) is defined as the image under \(\mathbb{P}\) of 
a real 3-dimensional subspace \(W \subset \mathbb{C}^{2,1}\) that is Lagrangian with respect to the 
symplectic form \(\omega(u,v) = \Im\langle u, v \rangle\) associated to the Hermitian form 
\(\langle \cdot, \cdot \rangle\). Equivalently, all \(\mathbb{R}\)-planes arise as images of 
\(\mathbf{H}^2_\mathbb{R}\) under the action of \(\mathbf{PU}(2,1)\).
A \emph{\(\mathbb{C}\)-plane (or complex line)} in \(\mathbf{H}^2_\mathbb{C}\) is the image under 
\(\mathbb{P}\) of a complex 2-dimensional subspace of \(\mathbb{C}^{2,1}\). All such 
\(\mathbb{C}\)-planes are obtained as images of \(\mathbf{H}^1_\mathbb{C}\) under elements of 
\(\mathbf{PU}(2,1)\).
 Each of the \(\mathbb{R}\)-planes 
 and the \(\mathbb{C}\)-planes  is isometric to the real hyperbolic plane 
\(\mathbf{H}^2_\mathbb{R}\), and their boundaries are topologically embedded circles in 
\(\partial\mathbf{H}^2_\mathbb{C} \cong \mathbb{S}^3\).

%In \(\mathbf{H}^2_\mathbb{C}\), there are exactly two conjugacy classes under the action of 
%\(\mathbf{PU}(2,1)\) of totally geodesic real surfaces of dimension 2: the %\(\mathbb{R}\)-planes 
%and the \(\mathbb{C}\)-planes. Each of these surfaces is isometric to the real hyperbolic plane 
%\(\mathbf{H}^2_\mathbb{R}\), and their boundaries are topologically embedded circles in 
%\(\partial\mathbf{H}^2_\mathbb{C} \cong \mathbb{S}^3\).

\subsection{Isometric spheres, bisectors,  Ford domain and Dirichlet domain}

We now introduce a concept that will play an important role in the constructions to follow.

%\begin{defn}
%    Let \( z_1 \) and \( z_2 \) be two distinct points in \( \hc \). The \emph{spinal hypersurface} or \emph{bisector} of \( \{z_1, z_2\} \) is the subset \( \mathcal{B}(\{z_1, z_2\}) \) defined by:
%    \[
%    \mathcal{B}(\{z_1, z_2\}) = \{ z \in \hc \mid \rho(z_1, z) = \rho(z_2, z) \}.
%    \]
%\end{defn}

%Let \( \Sigma \subset \hc \) be the complex geodesic generated by \( z_1 \) and \( z_2 \). 
%We call \( \Sigma \) the \emph{complex spine} of \( \mathcal{B}(\{z_1, z_2\}) \). The \emph{spine} of 
%\( \mathcal{B}(\{z_1, z_2\}) \) is given by:
%\[
%\sigma(\{z_1, z_2\}) = \mathcal{B}(\{z_1, z_2\}) \cap \Sigma = \{ z \in \Sigma \mid \rho(z_1, z) = \rho(z_2, z) \}.
%\]

\begin{defn} \label{def:bisector}
	For $z_1,z_2\in\hc$, the \emph{bisector} between $z_1$ and $z_2$ is the set  
	$$\mathcal{B}(z_1,z_2)=\{ p \in\overline{\hc}: |\langle \bf{p}, \bf{z}_1 \rangle | = |\langle \bf{p}, \bf{z}_2 \rangle|\}.$$
\end{defn} 

Let \( \Sigma \subset \hc \) be the complex geodesic generated by \( z_1 \) and \( z_2 \). 
We call \( \Sigma \) the \emph{complex spine} of \( \mathcal{B}(z_1, z_2) \). The \emph{spine} of 
\( \mathcal{B}(z_1, z_2) \) is given by:
\[
\sigma(z_1, z_2) = \mathcal{B}(z_1, z_2) \cap \Sigma = \{  |\langle \bf{p}, \bf{z}_1 \rangle | = |\langle \bf{p}, \bf{z}_2 \rangle| \}.
\]

\begin{defn}
    Let \( \mathcal{B} \) be a bisector. The set \( S = \mathcal{B} \cap \partial \hc \) is called a \emph{spinal sphere}. If \( \sigma \) is the spine of \( \mathcal{B} \), the two points of \( \sigma \cap \partial \hc \) are called the \emph{vertices} of the bisector.
\end{defn}

Since the orthogonal projection of \( \hc \) onto \( \sigma \) is real analytic, we deduce the following:

\begin{prop}
    A bisector is a smooth hypersurface of \( \hc \) that is diffeomorphic to \( \mathbb{R}^3 \). 
    A spinal sphere is a smooth hypersurface of \( \partial \hc \) that is diffeomorphic to \( \mathbb{S}^2 \).
\end{prop}

\begin{defn} \label{def:Dirichlet} 
	The \emph{Dirichlet domain} $\mathcal{D}_{\Gamma}$ for a discrete group $\Gamma \subset \mathbf{PU}(2,1)$
	 centered on $q_0\in\hc$ is defined as
	$$\mathcal{D}_{\Gamma}=\{p\in \overline{\hc}: |\langle \mathbf{p},\mathbf{q}_0\rangle|\leq|\langle 
		\mathbf{p},g(\mathbf{q}_0)\rangle|
	\ \forall g\in \Gamma \ \mbox{with} \ g(\mathbf{q}_0)\neq \mathbf{q}_0\}.$$
\end{defn}

Let $g=(g_{ij})$ be an element of $\mathbf{PU}(2,1)$ not fixing $\mathbf{q}_{\infty}$. 
This implies that $g_{31}\neq 0$.
\begin{defn} \label{def:isosphere}
	The \emph{isometric sphere} of $g$, denoted by $\mathcal{I}(g)$, is the set
	\begin{equation}\label{eq:isom-sphere}
	\mathcal{I}(g)=\{ p \in\overline{\hc}: |\langle {\bf{p}}, {\bf{q}}_{\infty} \rangle | = |\langle {\bf{p}}, g({\bf{q}}_{\infty}) \rangle| \}.
	\end{equation}
\end{defn}

An isometric sphere is a real 3-dimensional hypersurface in $\overline{\hc}$, 
diffeomorphic to $\mathbb{R}^3$.
Equivalently, $\mathcal{I}(g)$ is the bisector $\mathcal{B}(\mathbf{q}_{\infty}, g(\mathbf{q}_{\infty}))$.
The \emph{spinal sphere} of $\mathcal{I}(g)$ or $g$ is defined as $\mathcal{I}(g)\cap\partial \hc$, denoted by $\partial_{\infty}\mathcal{I}(g)$, which is a 2-sphere.

\begin{defn}The \emph{Ford domain} $\mathcal{F}_{\Gamma}$ for a discrete group $\Gamma \subset \mathbf{PU}(2,1)$ centred at $q_{\infty}$ is defined as
	$$\mathcal{F}_{\Gamma}=\{p\in \overline{\hc}: |\langle \mathbf{p},\mathbf{q}_{\infty}\rangle|\leq|\langle \mathbf{p},g(\mathbf{q}_{\infty})\rangle|
	\ \forall g\in \Gamma \ \mbox{with} \ g(\mathbf{q}_{\infty})\neq \mathbf{q}_{\infty}\}.$$
\end{defn}

\textbf{Important remark on the definition of an isometric sphere:} 
For an element $g$, the conventional definition of the isometric sphere $\mathcal{I}(g)$ 
(as found in \cite{ParkerWill:2017, MaXie2021}) corresponds to $\mathcal{I}(g^{-1})$ in our notation. 
The two definitions of isometric spheres are essentially equivalent in the study of 
Ford domains. We choose the one in Definition~\ref{def:isosphere}, which is 
consistent with Definitions~\ref{def:bisector} and~\ref{def:Dirichlet}, 
as it is more convenient for the proofs in Section \ref{section:proof}.

\section{A 1-dimensional moduli space of complex hyperbolic  triangle groups }
\label{subsection:1-dim}

%\section{A 1-dimensional moduli space of complex hyperbolic  hyperbolic triangle group 
%	\texorpdfstring{$\Delta_{4,4,\infty}$}{}}
%\label{subsection:1-dim}

In this section, we give the one-dimensional moduli space of three complex reflections $I_1$, $I_2$ and $I_3$ such that $$(I_2I_3)^4=id, \quad (I_3I_1)^4=id \quad \text{and} \quad I_1I_3I_2I_3 \quad \text{is parabolic}.$$
Here $I_1I_2$ may be loxodromic, parabolic or elliptic. This space contains the group $\Delta_{4,4,\infty;\infty}$ and $\Delta_{4,4,n;\infty}$ for each $n \geq 5$. 
We use the Siegel model of the complex hyperbolic plane.  Our parameterization  is a conjugacy of the one in 	\cite{JWX2023}.

Let $I_i$ be the reflection along  complex line $C_i$ for $i=1,2,3$. We see that $C_{2}$ and $C_{1}$ both meet $C_{3}$ at the angle $\pi/4$.
Suppose that the polar vectors $n_1$, $n_2$ and $n_3$ of the complex lines $C_1$, $C_2$ and $C_3$ are given by 
\begin{equation*}
n_1=\left[
\begin{array}{c}
	z_1 \\
	1 \\
	0 \\
\end{array}
\right], \quad \quad \quad  n_2=\left[
\begin{array}{c}
0 \\
-1 \\
z_2 \\
\end{array}
\right],  \quad  {\rm and}  \quad n_3=\left[
\begin{array}{c}
1 \\
0 \\
1 \\
\end{array}
\right].
\end{equation*}
Since $(I_2I_3)^4=(I_3I_1)^4=id$, then $tr(I_2I_3)=1$, $tr(I_3I_1)=1$, we have $|z_1|=|z_2|=1$. 
With the condition  $tr(I_1I_3I_2I_3)=3$, the corresponding complex reflections $I_1$, $I_2$   and  $I_3$ are given by 
\begin{equation*} \label{mtrixs:I1I2}
I_1=\left[
\begin{array}{ccc}
	-1 & 2 e^{\theta i} & 2 \\
	0 & 1 & 2 e^{-\theta i} \\
	0 & 0 & -1 \\
\end{array}
\right], \quad  I_2=\left[
\begin{array}{ccc}
	-1 & 0 & 0 \\
	 2 e ^{\theta i}  & 1 & 0 \\
	2 &  2 e ^{-\theta i}  & -1 \\
\end{array}
\right], \quad {\rm and}  \quad I_3=\left[
\begin{array}{ccc}
	0 & 0  & 1 \\
	  0 & -1 & 0 \\
	1 &  0 & 0 \\
\end{array}
\right]
\end{equation*} 
for $\theta \in [0, \pi/2]$. 
It is straightforward to check that $(I_3I_1)^4=id$, $(I_3I_2)^4=id$ and $I_1I_3I_2I_3$ is parabolic.  
When $\theta =0$, we have a classical $\mathbb{R}$-Fuchsian group.

The trace of $I_1I_2$ is $8 \cos(2 \theta)+7$:
\begin{itemize}
\item  When $\theta\in [0, \pi/3)$, $I_1I_2$ is loxodromic.
\item  When $\theta= \pi/3$, $I_1I_2$ is parabolic. We have the group $\Delta_{4,4,\infty;\infty}$.
\item   When $\theta\in (\pi/3, \pi/2]$, $I_1I_2$ is elliptic. In particular,  when $\theta= \arccos(\frac{ \cos(\pi/n)}{2})$, $tr(I_1I_2)=1+2 \cos(2 \pi/n)$, and $I_1I_2$ is elliptic with order $n$. We have the group $\Delta_{4,4,n;\infty}$. 

\end{itemize}

%So in the space of groups %parameterizied by $[0, \pi/2]$,  at the point $\theta_{\infty}=\pi/3$, $I_1I_2$ is parabolic, at the point $\theta_{n}=\arccos(\frac{ \cos(\pi/n)}{2})$, $I_1I_2$ is elliptic. 

For each $n \geq 5$, consider the Dirichlet domain $\mathcal{D}$ of the even subgroup of the group $\Delta_{4,4,n;\infty}$, centered at the fixed point of $I_1I_2$ when $\theta=\arccos(\frac{\cos(\pi/n)}{2})$. 
We also consider the Ford domain $\mathcal{F}$ of the even subgroup of the group $\Delta_{4,4,\infty;\infty}$, centered at the fixed point of $I_1I_2$ when $\theta=\pi/3$. 
We show that $\mathcal{D}$ and $\mathcal{F}$
have the same combinatorics and local topology when $n \geq 5$ in Sections~\ref{section:forddomain} and~\ref{section:dirichlet}.
This is a key step to demonstrate that the 3-manifold at infinity of the even subgroup of $\Delta_{4,4,n;\infty}$ can be obtained from the 3-manifold at infinity of the even subgroup of $\Delta_{4,4,\infty;\infty}$ by Dehn filling on the second cusp.
Although the explicit parameterization will not be used in the detailed computations of Sections~\ref{section:forddomain} and~\ref{section:dirichlet}, it is essential for establishing the existence of the required representations and for fixing the correspondence between the parameter $\theta$ and the Dehn filling slope $n$.

%\textcolor{red}{This section is complete}

\section{Ford domain of the even subgroup of \texorpdfstring{$\Delta_{4,4,\infty;\infty}$}{}}
\label{section:forddomain}

For the even subgroup of the group $\Delta_{4,4,\infty;\infty}$ described in 
Section~\ref{subsection:1-dim}, we give its Ford domain $\mathcal{F}$ following Jiang, Wang and Xie\cite{JWX2023}.  
However, we use different definitions and notations in order to maintain consistency between $\mathcal{F}$ and the Dirichlet domain $\mathcal{D}$ of the even subgroup of $\Delta_{4,4,n;\infty}$.
%, the isometric spheres in this section are not the same as in \cite{JWX2023}. 
%Consequently, the labels of the foiswing isometric spheres and the points are different from those in \cite{JWX2023}. 
Although all the results in this section are essentially proved in~\cite{JWX2023}, we outline them in some detail to prepare for Sections \ref{section:dirichlet} and \ref{section:proof}. 
All the properties in this section for $\mathcal{F}$ have counterparts for $\mathcal{D}$ in Section \ref{section:dirichlet}.
            
Let   
$$A=I_1I_2, ~ S=I_2I_3, ~ R=(I_2I_3)^2, ~ Q=(AS)^2=(I_1I_3)^2$$
be elements in $\Delta_{4,4,\infty;\infty}$ .
Then $A$ is parabolic and $S$, $R$, $Q$ are elliptic elements of orders 4, 2, 2, respectively. 
%Moreover, the element $A=I_1I_2$ is parabolic in $\Delta_{4,4,\infty;\infty}$.
We may conjugate $\Delta_{4,4,\infty;\infty}$ such that the fixed point of $A$ is ${\bf{q}}_{\infty}=(1,0,0)^{\top}$ in the Siegel model. Moreover, let $\Sigma=\langle I_1I_2, I_2I_3 \rangle$ be the even subgroup of $\Delta_{4,4,\infty;\infty}$.
 
 \subsection{The combinatorics of the Ford domain}

%  \begin{defn}\label{def:isospheres}
% We consider some isometric spheres for the group $\In addition,% \begin{enumerate}[(1)]
% \item Let $\mathcal{I}^{+}_{k}$  be the isometric sphere of $A^{k}SA^{-k}$, $\partial_{\infty}\mathcal{I}^{+}_{k}$ be the spinal sphere of $A^{k}SA^{-k}$.
% \item  Let $\mathcal{I}^{-}_{k}$  be the isometric sphere of $A^{k}S^{-1}A^{-k}$, $\partial_{\infty} \mathcal{I}^{-}_{k}$ be the spinal sphere of $A^{k}S^{-1}A^{-k}$.
% \item  Let $\mathcal{I}^{\star}_{k}$  be the isometric  sphere of $A^{k}RA^{-k}$, $\partial_{\infty} \mathcal{I}^{\star}_{k}$ be the spinal sphere of $A^{k}RA^{-k}$.
% \item  Let $\mathcal{I}^{\diamond}_{k}$  be the isometric sphere of $A^{k}QA^{-k}$, $\partial_{\infty} \mathcal{I}^{\diamond}_{k}$ be the spinal sphere of $A^{k}QA^{-k}$.
% \end{enumerate}
% \end{defn}

\begin{defn}\label{def:isospheres} For each $k \in \mathbb{Z}$, we set 
    \begin{equation*}
       \mathcal{I}^{+}_{k}=\mathcal{I}(A^{k}SA^{-k}),~
       \mathcal{I}^{-}_{k}=\mathcal{I}(A^{k}S^{-1}A^{-k}),~
       \mathcal{I}^{\star}_{k}=\mathcal{I}(A^{k}RA^{-k}),~
       \mathcal{I}^{\diamond}_{k}=\mathcal{I}(A^{k}QA^{-k})
    \end{equation*} to be the isometric spheres in the complex hyperbolic plane, 
    and use $\partial_{\infty}\mathcal{I}^{+}_{k}$,
    $\partial_{\infty}\mathcal{I}^{-}_{k}$,
    $\partial_{\infty}\mathcal{I}^{\star}_{k}$,
    $\partial_{\infty}\mathcal{I}^{\diamond}_{k}$ to denote their boundaries.
\end{defn}

\begin{defn} \label{def:Folddomain44ii} 
The infinite polyhedron $\mathcal{F}$ is the intersection of the closures of the exteriors
 of all isometric spheres in $\{\mathcal{I}_{k}^{+}, \mathcal{I}_{k}^{-},\mathcal{I}^{\star}_{k},
  \mathcal{I}^{\diamond}_{k}: k\in \mathbb{Z}\}$. That is,
\[
\mathcal{F}=\Bigl\{
  p \in \overline{\hc} :
  \, 
  |\langle \mathbf{p}, \mathbf{q}_{\infty} \rangle|
  \le 
  |\langle \mathbf{p}, A^{k} g(\mathbf{q}_{\infty}) \rangle|
  \ \text{for all } g \in \{ S, S^{-1}, R, Q \},\ k \in \mathbb{Z}
\Bigr\}.
\]

%We also define $\mathcal{F}_{\infty}=\mathcal{F} \cap \partial \hc$.
We use $\mathcal{F}_{\infty}$ to denote $\mathcal{F} \cap \partial \hc$ and
 $\partial\mathcal{F}_{\infty}$ to denote the boundary of $\mathcal{F}_{\infty}$.
\end{defn}
Jiang, Wang and Xie proved that $\mathcal{F}$ is a Ford domain of $\Sigma$ in \cite{JWX2023}.
\begin{defn}
    For $k\in \mathbb{Z}$, let $\mathfrak{s}_{k}^{+}$, $\mathfrak{s}_{k}^{-}$, $\mathfrak{s}^{\star}_{k}$ and $\mathfrak{s}^{\diamond}_{k}$ denote the 3-sides of $\mathcal{F}$ contained in the isometric spheres $\mathcal{I}_{k}^{+}$, $\mathcal{I}_{k}^{-}$, $\mathcal{I}^{\star}_{k}$ and $\mathcal{I}^{\diamond}_{k}$, respectively. 
    Besides, let $\partial _{\infty}\mathfrak{s}_{k}^{+}$, $\partial _{\infty} \mathfrak{s}_{k}^{-}$, $\partial _{\infty} \mathfrak{s}^{\star}_{k}$ and $\partial _{\infty} \mathfrak{s}^{\diamond}_{k}$ denote   $\mathfrak{s}_{k}^{+} \cap \partial \hc$, $\mathfrak{s}_{k}^{-} \cap \partial \hc$, $\mathfrak{s}^{\star}_{k} \cap \partial \hc$ and $\mathfrak{s}^{\diamond}_{k} \cap \partial \hc$, respectively. 
\end{defn}

\begin{defn}
    A \emph{ridge} is a 2-dimensional connected intersection of two sides.
\end{defn}

Figure \ref{figure:44iiford} illustrates a combinatorial picture of $\partial\mathcal{F}_{\infty}$, which should be compared with Figure 9 of \cite{JWX2023}.  
Here, consider a set of points  
$$\{p_i, q_i, r_i, s_i, t_i, u_i,v_i, w_i, x_i, y_i\}_{ i \in \mathbb{Z}},$$ 
in $\mathcal{F}_{\infty}$, 
such that $$A(p_i)=p_{i+1}, \quad A(q_i)=q_{i+1}, \quad  \cdots, \quad A(y_i)=y_{i+1}$$
for each $i \in \mathbb{Z}$. 
Figure \ref{figure:44iiford} is an infinite annulus with a $\mathbb{Z}$-action in $\mathbb{R}^3$, 
consisting of triangles, quadrangles, and octagons, each of which forms part  of 
 $\partial_{\infty} \mathfrak{s}_{k}^{j}$ for $k \in \mathbb{Z}$ and
  $j\in\{ +,-,\star,\diamond\}$. 
The infinite annulus is isotopic to the boundary of a neighborhood of an infinite line in the Heisenberg space  $\mathbb{C} \times \mathbb{R}$. 
The region outside the annulus (viewed   in  $\mathbb{C} \times \mathbb{R}$) in Figure \ref{figure:44iiford} is $\mathcal{F}_{\infty}$.  
Thus, the annulus in Figure \ref{figure:44iiford} is a combinatorial picture of the boundary 
 at infinite  of $\mathcal{F}$. 
The red pinched annulus is $\partial_{\infty}\mathfrak{s}_{0}^{+}$, which is the union of 
a quadrangle and an octagon. Similarly, the blue colored pinched annulus represents $\partial_{\infty} \mathfrak{s}_{0}^{-}$.

\begin{figure}[htbp]
	\begin{center}
		\begin{tikzpicture}
		\node at (0,0) {\includegraphics[width=14cm,height=6.0cm]{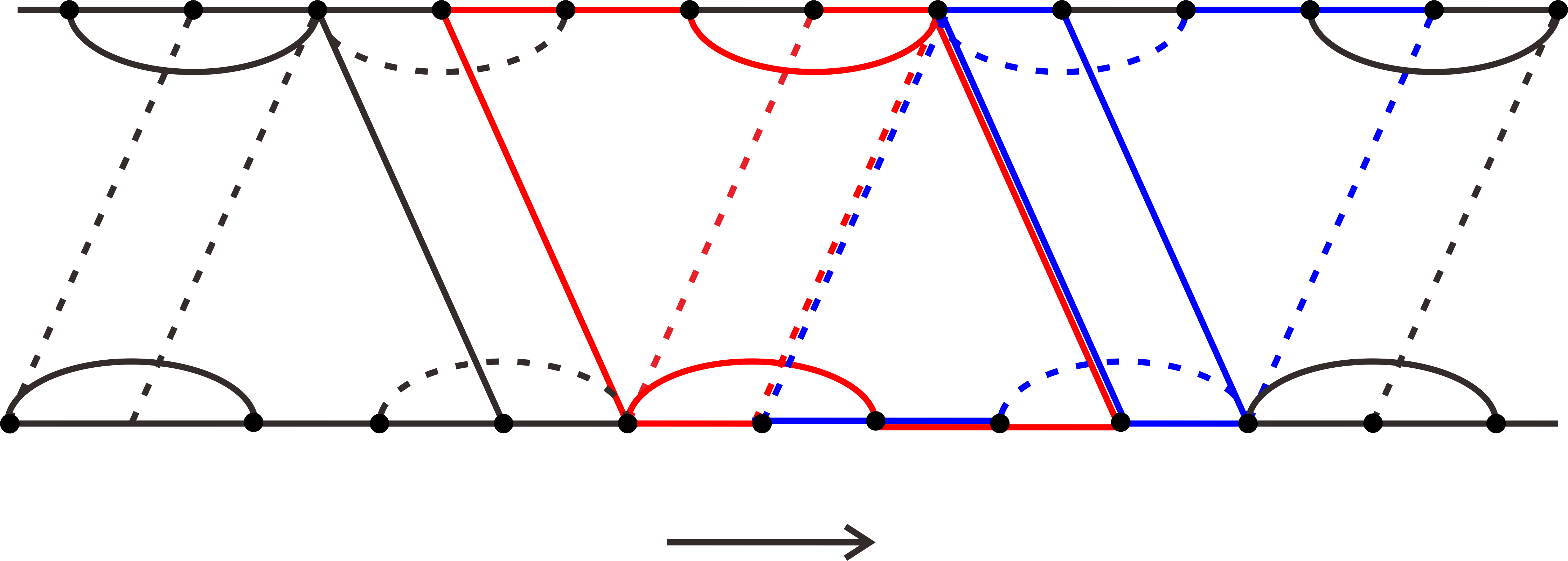}};

\node at (-6.3,3.3){\large $r_0$};
\node at (-5.1,3.3){\large $s_0$};
\node at (-4,3.3){\large $t_0$};
\node at (-2.8,3.3){\large $p_1$};
\node at (-1.9,3.3){\large $q_1$};
\node at (-0.9,3.3){\large $r_1$};
\node at (0.5,3.3){\large $s_1$};
\node at (1.5,3.3){\large $t_1$};
  		
\node at (2.3,3.3){\large $p_2$};
\node at (3.5,3.3){\large $q_2$};
\node at (4.5,3.3){\large $r_2$};
\node at (5.6,3.3){\large $s_2$};
\node at (6.7,3.3){\large $t_2$};

\node at (-6.8,-2.1){\large $u_0$};
\node at (-2.4,-2.1){\large $y_0$};
\node at (-1.3,-2.1){\large $u_1$};
\node at (-0.3,-2.1){\large $v_1$};
\node at (0.8,-2.1){\large $w_1$};
\node at (1.8,-2.1){\large $x_1$};
\node at (2.9,-2.1){\large $y_1$};
\node at (4,-2.1){\large $u_2$};
\node at (5.1,-2.1){\large $v_2$};
\node at (6.2,-2.1){\large $w_2$};

\draw[->] (0.2,-1.3)--(3.5,-2.5);
\draw[->] (3.2,-1.3)--(3.7,-2.5);
\node at (3.8,-2.75){$\partial_{\infty}\mathfrak{s}_{0}^{\star}$};

\draw[->] (-5.5,-1.3)--(-5.6,-2.5);
\draw[->] (-2.8,-1.3)--(-5.3,-2.5);
\node at (-5.6,-2.75){$\partial_{\infty}\mathfrak{s}_{-1}^{\star}$};

%\node at (-1.32,4.15){\tiny $A^{-1}QA$};
\node at (-1.32,4.0){ $\partial_{\infty}\mathfrak{s}_{-1}^{\diamond}$};
\draw[->] (-2.3,2.7)--(-1.5,3.7);
\draw[->] (0.0,2.7)--(-1.2,3.7);
  
%\node at (3.9,4.02){\tiny $Q$};
\node at (3.85,4.0){ $\partial_{\infty}\mathfrak{s}_0^{\diamond}$};
\draw[->] (2.8,2.7)--(3.7,3.7);
\draw[->] (5.1,2.7)--(4.0,3.7);

\node at (-0.2,-3.25){\small $A$-\text{action}};
\end{tikzpicture}
\end{center}
\caption{An abstract picture of $\mathcal{F}_{\infty}$. %the Ford domain of the even subgroup of $\Delta_{4,4,\infty;\infty}$.  
For example, the red region $\partial_{\infty}\mathfrak{s}_{0}^{+}$ and the blue region $\partial_{\infty}\mathfrak{s}_{0}^{-}$ are parts of the spinal spheres $\partial_{\infty}\mathcal{I}_0^{+}$ and $\partial_{\infty}\mathcal{I}_0^{-}$, respectively. 
}
\label{figure:44iiford}
\end{figure}
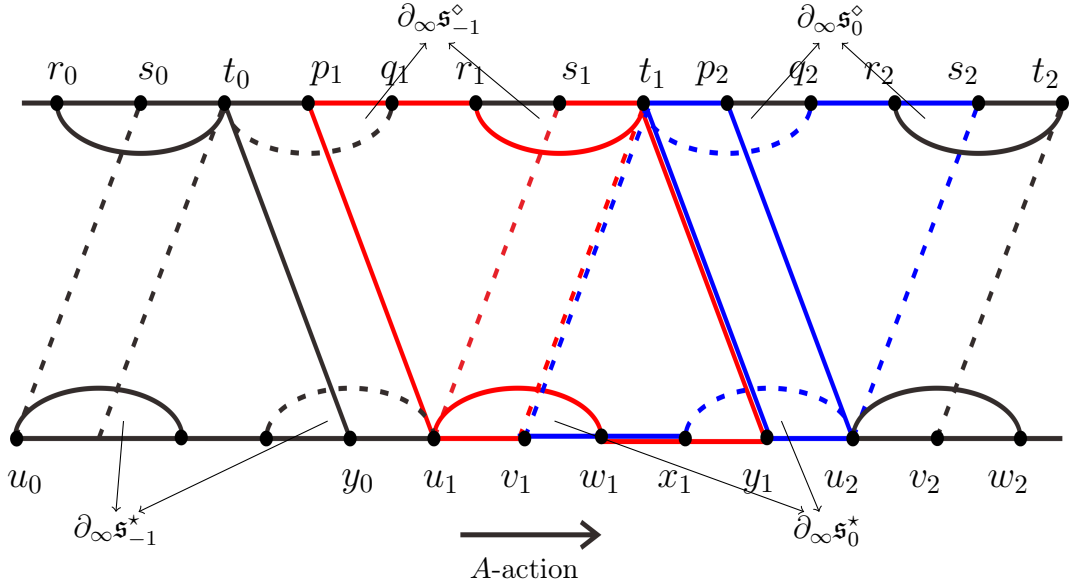

%In \cite{JWX2023}, Jiang-Xie-Wang proved  Proposition \ref{prop:pair-disjoint44pp}. 
%Readers may compare it with Figure \ref{figure:44iiford}.

We have the following result.

\begin{prop}[Corollary~4.11 in \cite{JWX2023}]
\label{prop:pair-disjoint44pp}
For any $k \in \mathbb{Z}$, the following hold:
\begin{enumerate}[(i)]

\item \label{item:Ikplus0k} 
$\mathcal{I}_{0}^{+} \cap \mathcal{I}_{k}^{+} = \emptyset$ when $k \neq 0$;

\item \label{item:Ikminus}  
$\mathcal{I}_{0}^{-} \cap \mathcal{I}_{k}^{-} = \emptyset$ when $k \neq 0$;

\item \label{item:IKplusminus}  
$\mathcal{I}_{0}^{+} \cap \mathcal{I}_{k}^{-} = \emptyset$ when $k \neq 0, -1$;

\item \label{item:IKplusstar}  
$\mathcal{I}_{0}^{+} \cap \mathcal{I}_{k}^{\star} = \emptyset$ when $k \neq 0, -1$.  
Moreover, $\mathcal{I}_{0}^{+}$ is tangent to $\mathcal{I}_{-1}^{\star}$ at the point $u_{1}$;

\item \label{item:IKplusdiamond} 
$\mathcal{I}_{0}^{+} \cap \mathcal{I}_{k}^{\diamond} = \emptyset$ when $k \neq 0, -1$.  
Moreover, $\mathcal{I}_{0}^{+}$ is tangent to $\mathcal{I}_{0}^{\diamond}$ at the point $t_{1}$;

\item 
$\mathcal{I}_{0}^{-} \cap \mathcal{I}_{k}^{\star} = \emptyset$ when $k \neq 0, 1$.  
Moreover, $\mathcal{I}_{0}^{-}$ is tangent to $\mathcal{I}_{1}^{\star}$ at the point $u_{2}$;

\item 
$\mathcal{I}_{0}^{-} \cap \mathcal{I}_{k}^{\diamond} = \emptyset$ when $k \neq 0, -1$.  
Moreover, $\mathcal{I}_{0}^{-}$ is tangent to $\mathcal{I}_{-1}^{\diamond}$ at the point $t_{1}$;

\item 
$\mathcal{I}_{0}^{\star} \cap \mathcal{I}_{k}^{\star} = \emptyset$ when $k \neq 1, -1$.  
Moreover, $\mathcal{I}_{0}^{\star}$ is tangent to both $\mathcal{I}_{1}^{\star}$ and $\mathcal{I}_{-1}^{\star}$ at the points $u_{2}$ and $u_{1}$, respectively, which are the fixed points of $AR$ and $A^{-1}R$;

\item \label{item:Ikstardiamond} 
$\mathcal{I}_{0}^{\star} \cap \mathcal{I}_{k}^{\diamond} = \emptyset$ for all $k$;

\item 
$\mathcal{I}_{0}^{\diamond} \cap \mathcal{I}_{k}^{\diamond} = \emptyset$ when $k \neq 1, -1$.  
Moreover, $\mathcal{I}_{0}^{\diamond}$ is tangent to both $\mathcal{I}_{1}^{\diamond}$ and $\mathcal{I}_{-1}^{\diamond}$ at the points $t_{2}$ and $t_{1}$, respectively, which are the fixed points of $AQ$ and $A^{-1}Q$.

\end{enumerate}
\end{prop}

\begin{prop}[Proposition 4.10 in \cite{JWX2023}]
\label{prop:44ppGirauddisk}
For the isometric spheres $\mathcal{I}_{0}^{+}$, $\mathcal{I}_{0}^{-}$, $\mathcal{I}_{1}^{+}$, $\mathcal{I}^{\star}_{0}$ and $\mathcal{I}^{\diamond}_{0}$, the following hold:
\begin{enumerate}[(i)]
\item Each of the intersections 
\[
\mathcal{I}_{0}^{+} \cap \mathcal{I}^{-}_{0},\qquad  
\mathcal{I}_{0}^{+} \cap \mathcal{I}^{\star}_{0},\qquad 
\mathcal{I}_{0}^{-} \cap \mathcal{I}^{\diamond}_{0}
\]
is a Giraud disk;

\item The triple intersection 
\[
\mathcal{I}_{0}^{+} \cap \mathcal{I}^{-}_{0} \cap \mathcal{I}^{\star}_{0}
\]
consists of two geodesics intersecting at the fixed point of $S$, with endpoints lying on $\partial\hc$;

\item The triple intersection 
\[
\mathcal{I}_{1}^{+} \cap \mathcal{I}^{-}_{0} \cap \mathcal{I}^{\diamond}_{0}
\]
consists of two geodesics intersecting at the fixed point of $Q$, with endpoints lying on $\partial\hc$.
\end{enumerate}
\end{prop}

Readers may refer to Figures \ref{figure:s0plusminus} and \ref{figure:s0stardiamond} before 
reading Propositions \ref{prop:44pp3face0plus}, \ref{prop:44pp3face0minus}, \ref{prop:44pp3face0star}, and \ref{prop:44pp3face0diamond}.

\begin{prop}[Propositions 4.15 and 5.3 in \cite{JWX2023}]
\label{prop:44pp3face0plus}
For the 3-side $\mathfrak{s}_{0}^{+}$ of $\mathcal{F}$, the following hold:
\begin{enumerate}[(i)]

\item 
Each of 
\[
\mathfrak{s}_{0}^{+} \cap \mathfrak{s}_{0}^{-}, \qquad
\mathfrak{s}_{0}^{+} \cap \mathfrak{s}_{-1}^{-}, \qquad
\mathfrak{s}_{0}^{+} \cap \mathfrak{s}^{\star}_{0}, \qquad
\mathfrak{s}_{0}^{+} \cap \mathfrak{s}^{\diamond}_{-1}
\]
is topologically the union of two sectors;

\item
The ideal boundary $\partial_{\infty} \mathfrak{s}_{0}^{+}$ is a union of a quadrangle and an octagon, intersecting at exactly two points.  
The quadrangle has vertices cyclically
\[
\{t_1, v_1, u_1, s_1\},
\]
and the octagon has vertices cyclically
\[
\{t_1, y_1, x_1, w_1, u_1, p_1, q_1, r_1\};
\]

\item
$\mathfrak{s}_{0}^{+}$ is a 3-ball in ${\overline{\hc}}$.  
Its boundary $\partial \mathfrak{s}_{0}^{+}$ is the union of 
\[
\partial_{\infty} \mathfrak{s}_{0}^{+},\quad
\mathfrak{s}_{0}^{+} \cap \mathfrak{s}_{0}^{-},\quad
\mathfrak{s}_{0}^{+} \cap \mathfrak{s}_{-1}^{-},\quad
\mathfrak{s}_{0}^{+} \cap \mathfrak{s}^{\star}_{0},\quad
\mathfrak{s}_{0}^{+} \cap \mathfrak{s}^{\diamond}_{-1}.
\]

\end{enumerate}
\end{prop}

\begin{prop}[Propositions 4.15 and 5.4 in \cite{JWX2023}]
\label{prop:44pp3face0minus}
For the 3-side $\mathfrak{s}_{0}^{-}$ of $\mathcal{F}$, the following hold:
\begin{enumerate}[(i)]

\item 
Each of 
\[
\mathfrak{s}_{0}^{-} \cap \mathfrak{s}_{0}^{+}, \qquad
\mathfrak{s}_{0}^{-} \cap \mathfrak{s}_{1}^{+}, \qquad
\mathfrak{s}_{0}^{-} \cap \mathfrak{s}^{\star}_{0}, \qquad
\mathfrak{s}_{0}^{-} \cap \mathfrak{s}^{\diamond}_{0}
\]
is topologically the union of two sectors;

\item 
The ideal boundary $\partial_{\infty} \mathfrak{s}_{0}^{-}$ is a union of a quadrangle and an octagon, intersecting in exactly two points.  
The quadrangle has vertices cyclically
\[
\{t_1, p_2, u_2, y_1\},
\]
and the octagon has vertices cyclically
\[
\{t_1, q_2, r_2, s_2, u_2, x_1, w_1, v_1\};
\]

\item 
$\mathfrak{s}_{0}^{-}$ is a 3-ball in $\overline{\hc}$.  
Its boundary $\partial \mathfrak{s}_{0}^{-}$ is the union of 
\[
\partial_{\infty} \mathfrak{s}_{0}^{-},\quad
\mathfrak{s}_{0}^{+} \cap \mathfrak{s}_{0}^{-},\quad
\mathfrak{s}_{0}^{-} \cap \mathfrak{s}_{1}^{+},\quad
\mathfrak{s}_{0}^{-} \cap \mathfrak{s}^{\diamond}_{0}.
\]

\end{enumerate}
\end{prop}

\begin{prop}[Propositions 4.16 and 5.6 in \cite{JWX2023}]
\label{prop:44pp3face0star}
For the 3-side $\mathfrak{s}_{0}^{\star}$ of $\mathcal{F}$, the following hold:
\begin{enumerate}[(i)]

\item 
$\mathfrak{s}_{0}^{\star}$ is topologically a solid light cone in $\overline{\hc}$;

\item 
$\mathfrak{s}_{0}^{\star}$ has the ridges 
\[
\mathfrak{s}_{0}^{\star} \cap \mathfrak{s}_{0}^{-}
\quad\text{and}\quad
\mathfrak{s}_{0}^{\star} \cap \mathfrak{s}_{0}^{+},
\]
each of which is topologically the union of two sectors;

\item 
The ideal boundary $\partial_{\infty} \mathfrak{s}_{0}^{\star}$ is the union of two triangles, with vertices
\[
\{u_1, v_1, w_1\}
\quad\text{and}\quad
\{x_1, y_1, u_2\},
\]
respectively.

\end{enumerate}
\end{prop}

\begin{prop}[Propositions 4.16 and 5.7 in \cite{JWX2023}]
\label{prop:44pp3face0diamond}
For the 3-side $\mathfrak{s}_{0}^{\diamond}$ of $\mathcal{F}$, the following hold:
\begin{enumerate}[(i)]

\item 
$\mathfrak{s}_{0}^{\diamond}$ is topologically a solid light cone in $\overline{\hc}$;

\item 
$\mathfrak{s}_{0}^{\diamond}$ has the ridges
\[
\mathfrak{s}_{0}^{\diamond} \cap \mathfrak{s}_{0}^{-}
\quad\text{and}\quad
\mathfrak{s}_{0}^{\diamond} \cap \mathfrak{s}_{1}^{+},
\]
each of which is topologically the union of two sectors;

\item 
The ideal boundary $\partial_{\infty} \mathfrak{s}_{0}^{\diamond}$ is the union of two triangles, with vertices
\[
\{t_1, p_2, q_2\}
\quad\text{and}\quad
\{r_2, s_2, t_2\},
\]
respectively.

\end{enumerate}
\end{prop}

Let $\kappa$, $\varkappa$, and $\tilde{\varkappa} = A^{-1}(\varkappa)$ denote the fixed points of $S$, $S^{-1}A^{-1}$, and $A^{-1}S^{-1}$, respectively.  
Tables~\ref{s0+} and~\ref{s0-} describe the vertices and the combinatorial shapes that form part of the intersection of $\mathfrak{s}^{+}_0$ and $\mathfrak{s}^{-}_0$ with $\hc$, as illustrated in Figure~\ref{figure:s0plusminus}.

\begin{table}[ht]
\begin{minipage}[t]{0.495\textwidth}
\makeatletter\def\@captype{table}
\centering
\begin{tabular}{c|c|c}
    \toprule
   ridges   & quadrangle & triangle \\
    \hline 
  $\mathfrak{s}^{-}_{-1} \cap \mathfrak{s}^{+}_0$   & $\tilde{\varkappa}$, $p_1$, $u_1$, $s_1$ & $q_1$, $r_1$, $\tilde{\varkappa}$ 
    \\
    \hline
 $\mathfrak{s}^{\diamond}_{-1} \cap \mathfrak{s}^{+}_0$      & $\tilde{\varkappa}$, $r_1$, $t_1$, $s_1$ & $q_1$, $p_1$, $\tilde{\varkappa}$    \\
    \hline
  $\mathfrak{s}^{\star}_{0} \cap \mathfrak{s}^{+}_0$    & $\kappa$, $v_1$, $u_1$, $w_1$ & $x_1$,$y_1$, $\kappa$   \\
    \hline
  $\mathfrak{s}^{-}_{0}\cap \mathfrak{s}^{+}_0$    & $\kappa$, $y_1$, $t_1$, $v_1$ & $x_1$, $w_1$, $\kappa$ \\
    \bottomrule
\end{tabular}
\caption{The ridges of $\mathfrak{s}^{+}_0$. }%\cap\hc
\label{s0+}
\end{minipage}
\begin{minipage}[t]{0.495\textwidth}
\makeatletter\def\@captype{table}
\centering
\begin{tabular}{c|c|c}
    \toprule
    ridges    & quadrangle & triangle\\
    \hline 
    $\mathfrak{s}^{\star}_0 \cap \mathfrak{s}^{-}_0$   & $\kappa$, $y_1$, $u_2$, $x_1$ & $w_1$, $v_1$, $\kappa$ \\
    \hline
    $\mathfrak{s}^{+}_0 \cap \mathfrak{s}^{-}_0$  & $\kappa$, $y_1$, $t_1$, $v_1$  &  $w_1$, $x_1$, $\kappa$     \\
    \hline
    $\mathfrak{s}^{+}_1 \cap \mathfrak{s}^{-}_0$ & $\varkappa$, $p_2$, $u_2$, $s_2$  & $r_2$, $q_2$, $\varkappa$       \\
    \hline
    $\mathfrak{s}^{\diamond}_1 \cap \mathfrak{s}^{-}_0$  & $\varkappa$, $q_2$, $t_1$, $p_2$ & $r_2$, $s_2$, $\varkappa$   \\
    \bottomrule
\end{tabular}
\caption{The ridges of $\mathfrak{s}^{-}_0$. }
\label{s0-}
\end{minipage}
\end{table}
%\end{minipage}

%Let $\kappa$, $\varkappa$ and $\tilde{\varkappa}=A^{-1}(\varkappa)$ denote the fixed points of $S$, $S^{-1}A^{-1}$ and $A^{-1}S^{-1}$, respectively.
%$\varkappa$ denote the fixed point of $S^{-1}A^{-1}$ and $\tilde{\varkappa}=A^{-1}(\varkappa)$ denote the fixed point of $A^{-1}S^{-1}$.
%The tables \ref{s0+} and \ref{s0-} describe the vertices and shapes that are part of the intersection of $\mathfrak{s}^{+}_0$ and $\mathfrak{s}^{-}_0$ with $\hc$ drawn %in Figure \ref{figure:s0plusminus}.

%$\\$
%\begin{minipage}{\textwidth}

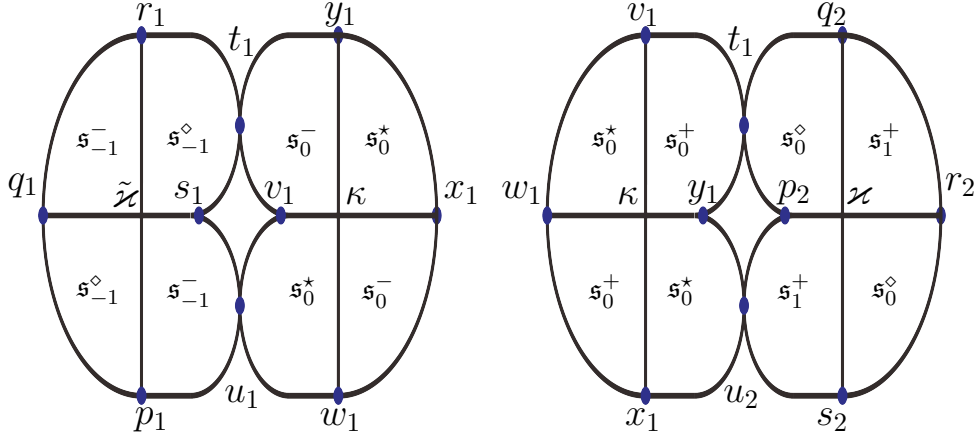
\begin{figure}[ht]
	\begin{center}
		\begin{tikzpicture}
		\node at (0,0) {\includegraphics[width=12cm,height=5.0cm]{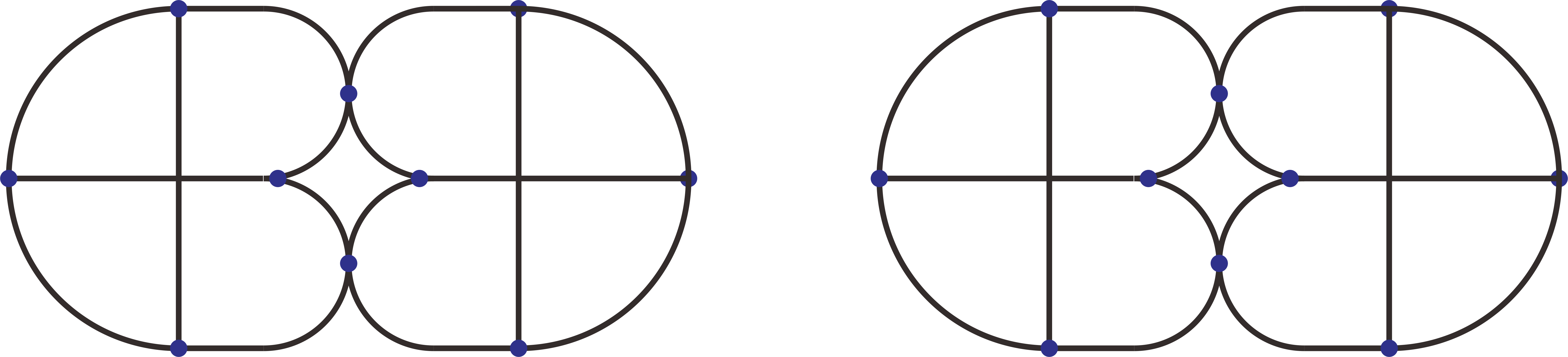}};
 \node at (-4.5,2.65){\large $r_1$};
 \node at (-3.3,2.3){\large $t_1$};
 \node at (-2,2.65){\large $y_1$};
 \node at (-0.4,0.3){\large $x_1$};
 \node at (-2,-2.7){\large $w_1$};
 \node at (-3.3,-2.4){\large $u_1$};
 \node at (-4.5,-2.7){\large $p_1$};
 \node at (-6.2,0.4){\large $q_1$};
 \node at (-4,0.3){\large $s_1$};
 \node at (-2.8,0.3){\large $v_1$};
%%%%%%%%%%%%%%%%%%%%%%%
\node at (-4.85,0.28){\large $\tilde{\varkappa}$};
 \node at (-1.8,0.26){\large $\kappa$};
 \node at (1.8,0.26){\large $\kappa$};
 \node at (4.85,0.26){\large $\varkappa$};
 %%%%%%%%%%%%%%%%%%%%%%%%%%%%%%%%%%%%
 \node at (4.5,2.65){\large $q_2$};
 \node at (3.3,2.3){\large $t_1$};
 \node at (2,2.65){\large $v_1$};
 \node at (0.4,0.3){\large $w_1$};
 \node at (2,-2.7){\large $x_1$};
 \node at (3.3,-2.4){\large $u_2$};
 \node at (4.5,-2.7){\large $s_2$};
 \node at (6.2,0.4){\large $r_2$};
 \node at (4,0.3){\large $p_2$};
 \node at (2.8,0.3){\large $y_1$};

 \node at (-1.5,1){\small $\mathfrak{s}^{\star}_0$ };
 \node at (-2.5,-1){\small $\mathfrak{s}^{\star}_0$};
 \node at (-2.5,1){\small $\mathfrak{s}^{-}_0$ };
 \node at (-1.5,-1){\small $\mathfrak{s}^{-}_0$};
 \node at (-4,1){\small $\mathfrak{s}^{\diamond}_{-1}$ };
 \node at (-5.2,-1){\small $\mathfrak{s}^{\diamond}_{-1}$};
 \node at (-5.2,1){\small $\mathfrak{s}^{-}_{-1}$ };
 \node at (-4,-1){\small $\mathfrak{s}^{-}_{-1}$};

 \node at (1.5,1){\small $\mathfrak{s}^{\star}_0$ };
 \node at (2.5,-1){\small $\mathfrak{s}^{\star}_0$};
 \node at (2.5,1){\small $\mathfrak{s}^{+}_0$ };
 \node at (1.5,-1){\small $\mathfrak{s}^{+}_0$};
 \node at (4,1){\small $\mathfrak{s}^{\diamond}_0$ };
 \node at (5.2,-1){\small $\mathfrak{s}^{\diamond}_0$};
 \node at (5.2,1){\small $\mathfrak{s}^{+}_1$ };
 \node at (4,-1){\small $\mathfrak{s}^{+}_1$};
 
\end{tikzpicture}
\end{center}
\caption{Schematic views of the 2-cell decomposition of $\partial\mathfrak{s}^{+}_0$ (left) and  $\partial\mathfrak{s}^{-}_0$ (right). 
On the left, we use $ \mathfrak{s}^{*}_0$ to represents $\mathfrak{s}^{+}_0 \cap \mathfrak{s}^{*}_0$. 
The other labels obey the same rule.}
\label{figure:s0plusminus}
\end{figure}

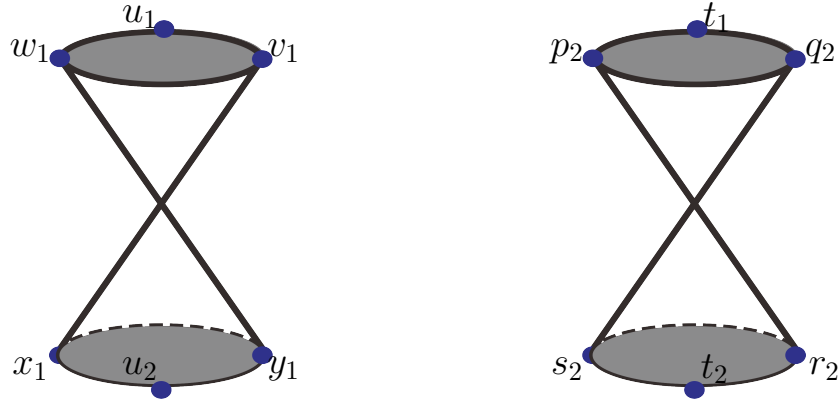
\begin{figure}[htbp]
	\begin{center}
		\begin{tikzpicture}
		\node at (0,0) {\includegraphics[width=10cm,height=5.0cm]{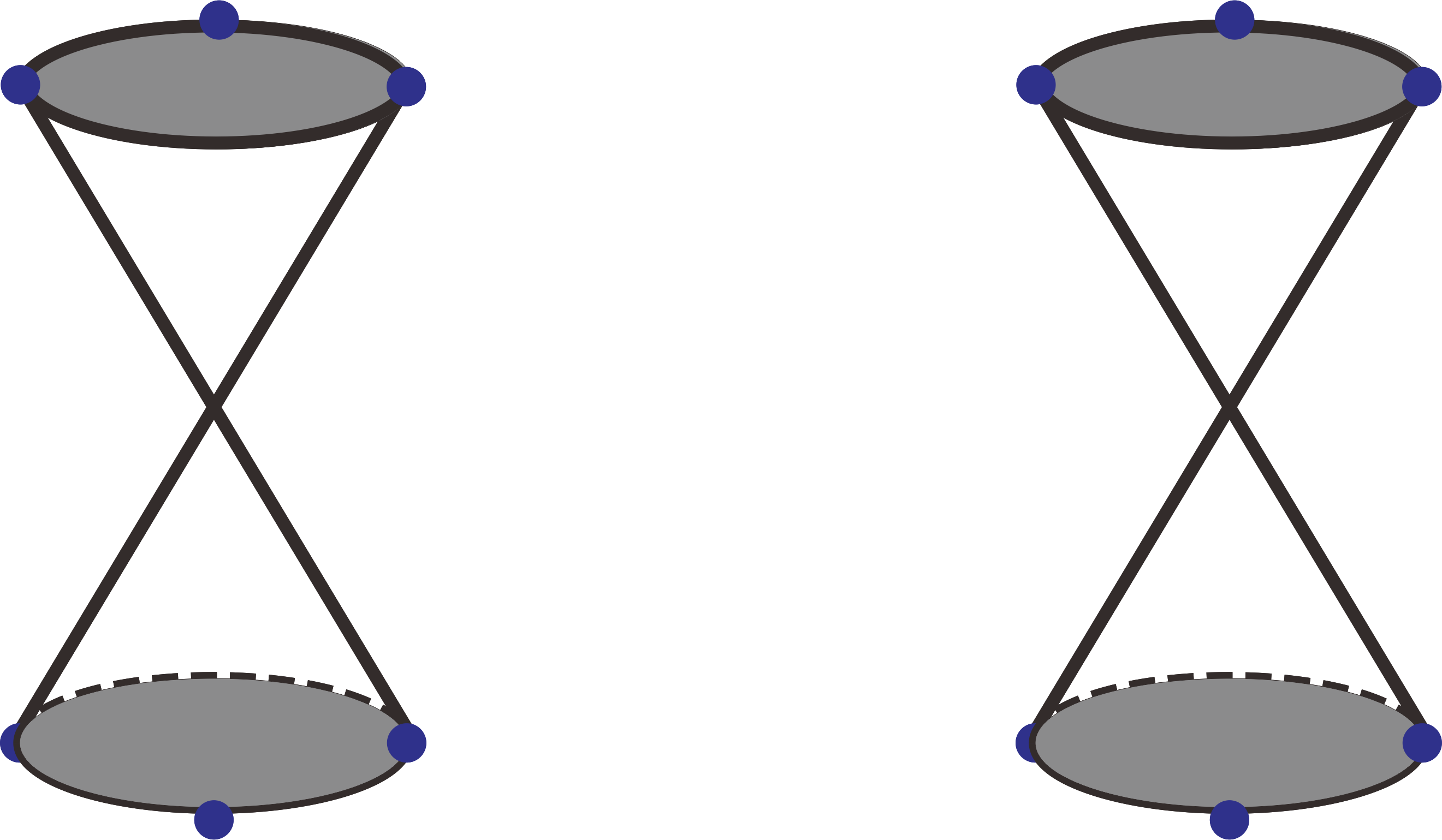}};
    \node at (-3.8,-2.1){\large $u_2$};
    \node at (-1.9,-2.1){\large $y_1$};
    \node at (-1.9,2.1){\large $v_1$};
    \node at (-5.25,-2.1){\large $x_1$};
    \node at (-5.25,2.1){\large $w_1$};
    \node at (-3.8,2.55){\large $u_1$};

    \node at (3.8,-2.1){\large $t_2$};
    \node at (1.85,-2.1){\large $s_2$};
    \node at (1.85,2.1){\large $p_2$};
    \node at (5.25,-2.1){\large $r_2$};
    \node at (5.2,2.1){\large $q_2$};
    \node at (3.83,2.55){\large $t_1$};
\end{tikzpicture}
\end{center}
\caption{A schematic view of the 3-sides $\mathfrak{s}^{\star}_0$ (left) and $\mathfrak{s}^{\diamond}_0$ (right). 
The shadowed regions are $\partial_{\infty}\mathfrak{s}^{\star}_0$ and $\partial_{\infty}\mathfrak{s}^{\diamond}_0$, respectively.}
\label{figure:s0stardiamond}
\end{figure}

Furthermore, the left subfigure of Figure~\ref{figure:s0plusminus} illustrates that
the union of the inner quadrangle with vertices $s_1, t_1, v_1, u_1$ and the outer octagon with vertices $r_1, t_1, y_1, x_1, w_1, u_1, p_1, q_1$ constitutes $\partial_{\infty}\mathfrak{s}^{+}_0$.
Similarly, the right subfigure shows that the union of the inner quadrangle with vertices $y_1, t_1, p_2, u_2$ and the outer octagon with vertices $v_1, t_1, q_2, r_2, s_2, u_2, x_1, w_1$ constitutes $\partial_{\infty}\mathfrak{s}^{-}_0$.

\begin{prop}[Proposition 5.2 in \cite{JWX2023}]
    The points $p_1$, $q_1$, $r_1$, $s_1$, $t_1$, $u_1$, $v_1$, $w_1$, $x_1$, $y_1$ satisfy:
    \begin{enumerate}[(i)]
    \item $t_1$  is the intersection  $\mathcal{I}_{0}^{+}\cap\mathcal{I}_{0}^{-}\cap\mathcal{I}_{0}^{\diamond}\cap\mathcal{I}_{-1}^{\diamond}$;
    \item $u_1$  is the intersection  $\mathcal{I}_{0}^{+}\cap\mathcal{I}_{-1}^{-}\cap\mathcal{I}_{0}^{\star}\cap\mathcal{I}_{-1}^{\star}$;
    \item $\{v_1, w_1, x_1, y_1\}$ is exactly the set $\mathcal{I}_{0}^{+}\cap\mathcal{I}_{0}^{-}\cap\mathcal{I}_{0}^{\star} \cap \partial \hc$; 
    \item $\{p_1,q_1,r_1,s_1\}$ is exactly the set $\mathcal{I}_{0}^{+}\cap\mathcal{I}_{-1}^{-}\cap\mathcal{I}_{-1}^{\diamond} \cap  \partial \hc$.
    \end{enumerate}
\end{prop}

To use the Poincar\'e polyhedron theorem for the Ford domain $\mathcal{F}$, we consider the following:

\textbf{The side-pairing maps:} 
\[
A^k S A^{-k} : \mathfrak{s}_k^{-} \mapsto \mathfrak{s}_k^{+}, \quad
A^k S^2 A^{-k} : \mathfrak{s}_k^{\star} \mapsto \mathfrak{s}_k^{\star}, \quad
A^k Q A^{-k} : \mathfrak{s}_k^{\diamond} \mapsto \mathfrak{s}_k^{\diamond}
\]

\textbf{The ridge circles:}
\begin{equation*}
\begin{aligned}
(\mathfrak{s}_k^{-} \cap \mathfrak{s}_k^{+}, \;& \mathfrak{s}_k^{-}, \;\mathfrak{s}_k^{+})
  \xrightarrow{A^k S A^{-k}}
  (\mathfrak{s}_k^{\star} \cap \mathfrak{s}_k^{+}, \;\mathfrak{s}_k^{\star}, \;\mathfrak{s}_k^{+})
  \xrightarrow{A^k S^2 A^{-k}} \\
(\mathfrak{s}_k^{\star} \cap \mathfrak{s}_k^{-}, \;& \mathfrak{s}_k^{\star}, \;\mathfrak{s}_k^{-})
  \xrightarrow{A^k S A^{-k}}
  (\mathfrak{s}_k^{-} \cap \mathfrak{s}_k^{+}, \;\mathfrak{s}_k^{-}, \;\mathfrak{s}_k^{+})
\end{aligned}
\end{equation*}

and

\begin{equation*}
\begin{aligned}
(\mathfrak{s}_k^{-} \cap \mathfrak{s}_{k+1}^{+}, \;& \mathfrak{s}_k^{-}, \;\mathfrak{s}_{k+1}^{+})
  \xrightarrow{A^k (AS)^{-1} A^{-k}}
  (\mathfrak{s}_k^{\diamond} \cap \mathfrak{s}_k^{-}, \;\mathfrak{s}_k^{\diamond}, \;\mathfrak{s}_k^{-})
  \xrightarrow{A^k (AS)^{-2} A^{-k}} \\
(\mathfrak{s}_k^{\diamond} \cap \mathfrak{s}_{k+1}^{+}, \;& \mathfrak{s}_k^{\diamond}, \;\mathfrak{s}_{k+1}^{+})
  \xrightarrow{A^k (AS)^{-1} A^{-k}}
  (\mathfrak{s}_k^{-} \cap \mathfrak{s}_{k+1}^{+}, \;\mathfrak{s}_k^{-}, \;\mathfrak{s}_{k+1}^{+})
\end{aligned}
\end{equation*}

From the side-pairing maps,  Jiang, Wang and Xie proved

\begin{thm}[Theorem 4.17 in \cite{JWX2023}] \label{thm:44pFord}  
Let $\Sigma=\langle I_1I_2, I_2I_3 \rangle$ be the even subgroup of  complex hyperbolic  triangle group $\Delta_{4,4,\infty;\infty}$. Then $\mathcal{F}$ is the  Ford domain of $\Sigma$, and  $\mathcal{F}_{\infty}$ is homeomorphic to $\mathbb{R}^3-int(\mathbb{D}^2\times \mathbb{R} )$. Moreover,  $\Sigma$ is discrete and has the presentation
$$\langle S, A | S^4=id,~ (AS)^4=id\rangle$$
\end{thm}

\subsection{The isomorphism between \texorpdfstring{$\pi_1(\mathscr{M}_{4,4,\infty; \infty}) ~\text{and}~ \pi_1(s782)$}{}}

We denote by $\mathscr{M}_{4,4,\infty; \infty}$ the 3-manifold at infinity of the even subgroup of  complex hyperbolic  triangle group $\Delta_{4,4,\infty;\infty}$.  In this subsection, we outline the proof that there exists an isomorphism between the fundamental group of the 3-manifold at infinity of $\Delta_{4,4,\infty;\infty}$ and $\pi_1(\text{s782})$. 
Since we require a concordance between the Dirichlet domain of $\Delta_{4,4,n;\infty}$ and the Ford domain of the group $\Delta_{4,4,\infty;\infty}$, the isometric spheres considered in this section are not the same as those in \cite{JWX2023}. Consequently, the labels of the side-pairings are different from those in \cite{JWX2023}. For this reason, we cannot directly use the isomorphism from \cite{JWX2023} for the purposes in Section \ref{section:proof}.

Take a proper disk $E$ in $\mathcal{F}_{\infty}$, punctured at $q_{\infty}$, whose boundary is the union of the arcs
\[
[p_1,q_1],\ [q_1,r_1],\ [r_1,s_1],\ [s_1,u_1],\ [u_1,p_1].
\]
Then $A(E)$ is also a proper disk in $\mathcal{F}_{\infty}$, punctured at $q_{\infty}$, with boundary the union of the arcs
\[
[p_2,q_2],\ [q_2,r_2],\ [r_2,s_2],\ [s_2,u_2],\ [u_2,p_2].
\]
Since $\partial E \cap \partial A(E)=\emptyset$, by standard arguments in 3-dimensional topology we may assume that
\[
E \cap A(E)=\emptyset.
\]

The region in $\mathcal{F}_{\infty}$ bounded by $E$ and $A(E)$ is a fundamental domain $L$ for the $\mathbb{Z}=\langle A\rangle$-action on $\mathcal{F}_{\infty}$.  
We then take a proper disk $D$ in $L$ whose boundary is the union of 
\[
[q_{\infty},u_1],\ [u_1,v_1],\ [v_1,w_1],\ [w_1,x_1],\ [x_1,y_1],\ [y_1,u_2],\ [u_2,q_{\infty}].
\]
Again by standard arguments in 3-dimensional topology, we may assume that
\[
D \cap E = [u_1,q_{\infty}] \qquad \text{and} \qquad 
D \cap A(E) = [u_2,q_{\infty}].
\]
See Figure~\ref{figure:unit44iiford}.  
Here the left (orange) and right (green) circles represent $\partial E$ and $\partial A(E)$, respectively.  
The reader may compare Figure~\ref{figure:unit44iiford} with Figure~\ref{figure:44iiford}.

Cutting $L$ along $D$, we obtain a 3-ball $N$.  
See Figure~\ref{figure:edgecircle4iii} for the 2–cell decomposition of the 2-sphere boundary of $N$.

	The boundary $\partial N$ consists of:
\begin{itemize}
    \item two copies of $D$, denoted by $D_+$ and $D_-$;
    \item four disks labelled $S_8$, $S_8^{-1}$, $S_4$ and $S_4^{-1}$, which are associated to the group element $S$;
    \item two disks labelled $R_+$ and $R_-$, which are associated to the group element $R$;
    \item two disks labelled $Q_+$ and $A^{-1}Q_{-}A$,  which are associated to the group element $Q$.
\end{itemize}

\begin{figure}[htbp]
    \begin{center}
    \begin{tikzpicture}
    \node at (0,0) 
    {\includegraphics[width=10cm,height=7.0cm]{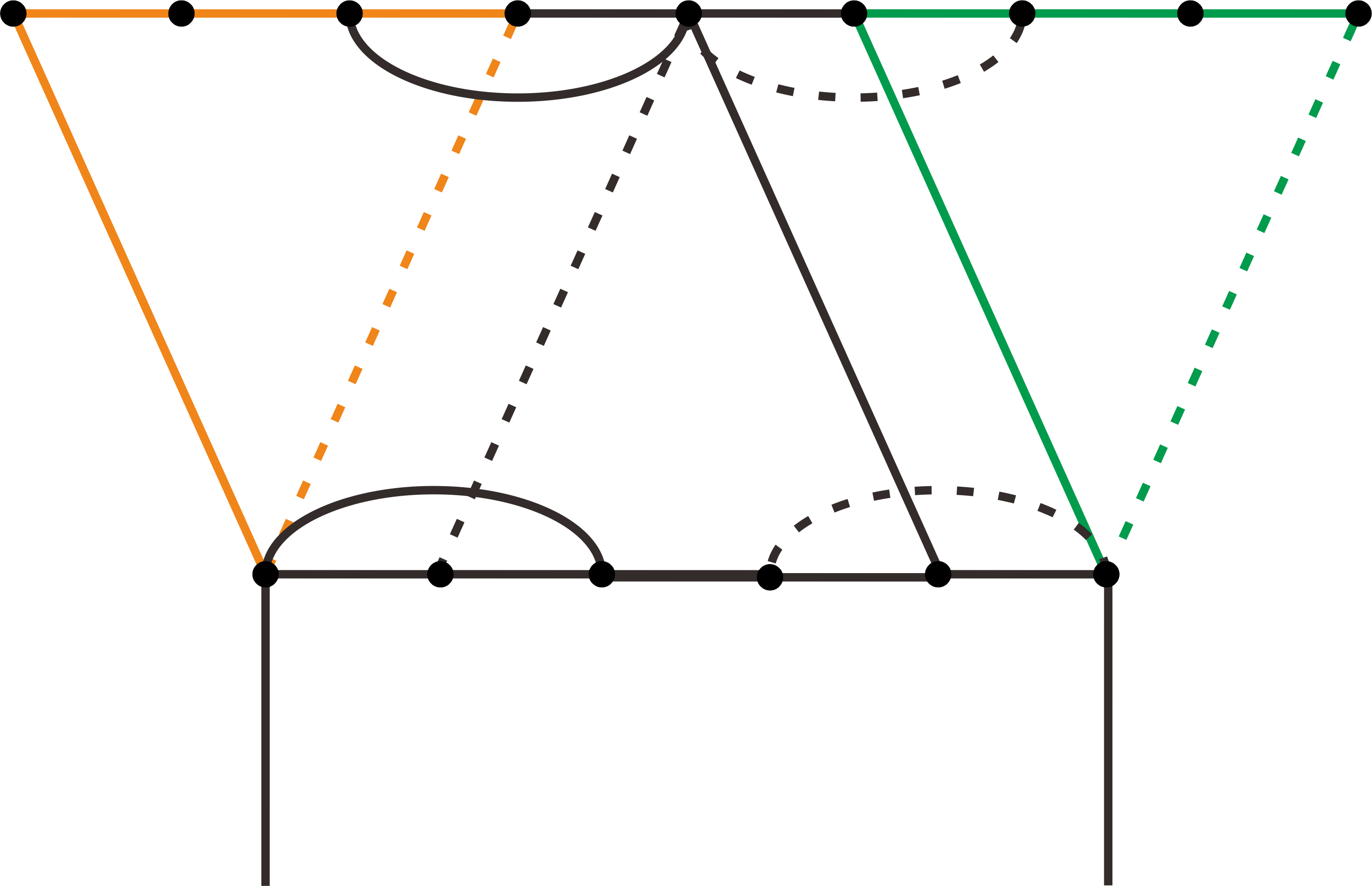}};
    \node at (-4.8,3.7){\large $p_1$};
    \node at (-3.65,3.7){\large $q_1$};
    \node at (-2.5,3.7){\large $r_1$};
    \node at (-1.25,3.7){\large $s_1$};
    \node at (0,3.7){\large $t_1$};
    \node at (1.25,3.7){\large $p_2$};
    \node at (2.5,3.7){\large $q_2$};
    \node at (3.75,3.7){\large $r_2$};
    \node at (4.8,3.7){\large $s_2$};
    
    \node at (-3.4,-1.2){\large $u_1$};
    \node at (-1.8,-1.4){\large $v_1$};
    \node at (-0.6,-1.4){\large $w_1$};
    \node at (0.6,-1.4){\large $x_1$};
    \node at (1.8,-1.4){\large $y_1$};
    \node at (3.4,-1.2){\large $u_2$};

    \node at (0,-2.7){\large $D$};
    \end{tikzpicture}
    \end{center}
    \caption{A fundamental domain of the $\mathbb{Z}=\langle A \rangle $-action on $\mathcal{F}_{\infty}$ with a cutting disk $D$.}
    \label{figure:unit44iiford}
\end{figure}
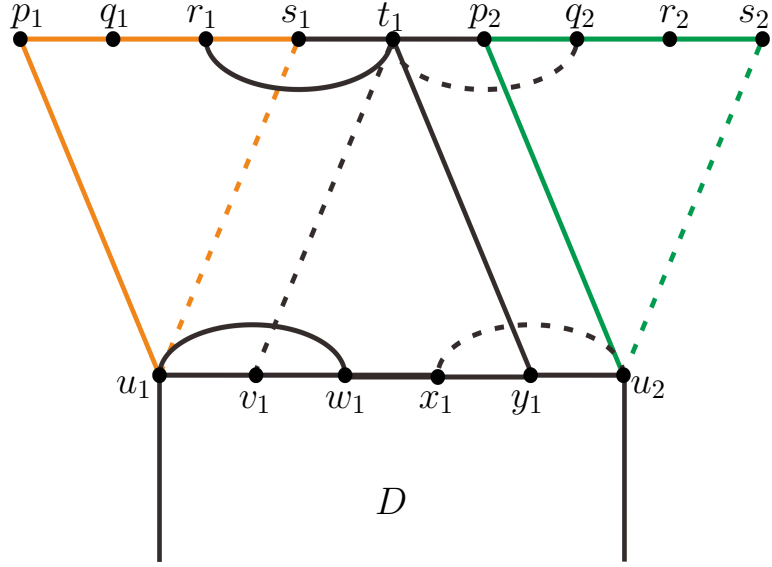

\begin{figure}[htbp]
\begin{center}
\begin{tikzpicture}
\node at (0,0) {\includegraphics[width=10cm,height=10.0cm]{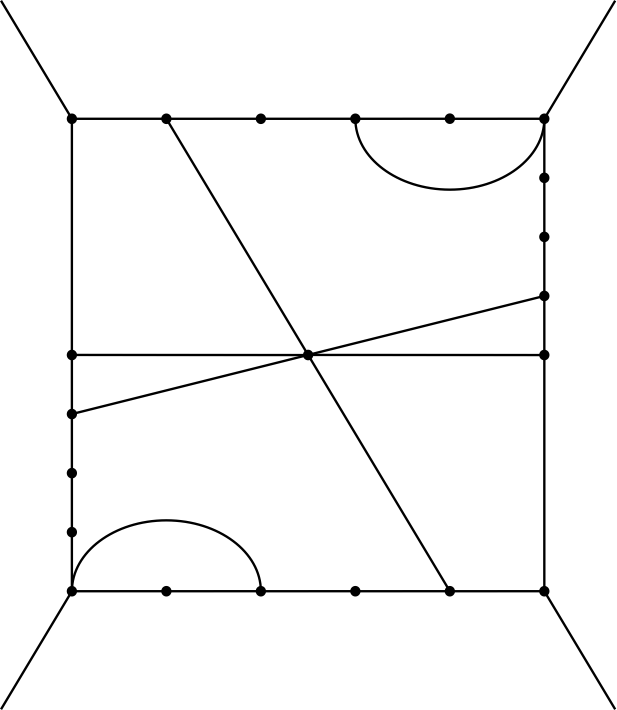}};

\node at (-4.1,3.3){\large $u_1$};
\node at (-2.3,3.55){\large $v_1$};
\node at (-0.85,3.55){\large $w_1$};
\node at (0.85,3.55){\large $x_1$};
\node at (2.3,3.55){\large $y_1$};
\node at (4.1,3.3){\large $u_2$};

\node at (-4.15,-3.3){\large $u_1$};
\node at (-2.3,-3.6){\large $v_1$};
\node at (-0.85,-3.6){\large $w_1$};
\node at (0.85,-3.6){\large $x_1$};
\node at (2.3,-3.6){\large $y_1$};
\node at (4.15,-3.3){\large $u_2$};

\node at (-4.15,0){\large $s_1$};    
\node at (-4.15,-0.825){\large $r_1$}; 
\node at (-4.15,-1.7){\large $q_1$}; 
\node at (-4.15,-2.525){\large $p_1$}; 

\node at (4.15,0){\large $p_2$};    
\node at (4.15,0.825){\large $q_2$}; 
\node at (4.15,1.7){\large $r_2$}; 
\node at (4.15,2.525){\large $s_2$}; 

\node at (-0.06,-0.25){\large $t_1$};

\node at (-2.2,1.25){\large $S_4$};
\node at (-1.46,-1.35){\large $S_8$};

\node at (1.06,1.25){\large $S^{-1}_8$};
\node at (2.2,-1.35){\large $S^{-1}_4$};

\node at (2.56,2.75){\large $R_{-}$};
\node at (-2.56,-2.85){\large $R_{+}$};

\node at (0.56,-4.65){\large $D_{+}$};

\node at (0.56,4.3){\large $D_{-}$};

\node at (-5,1.25){\large $E$};
\node at (5,-1.25){\large $A(E)$};

\node at (3.1,0.25){\small $Q_+$};
\node at (-3,-0.3){\small $A^{-1}Q_{-}A$};
\end{tikzpicture}
\end{center}
\caption{The 2-cell decomposition of the 2-sphere boundary of $N$ to get the 3-manifold $\mathscr{M}_{4,4,\infty; \infty}$. }
\label{figure:edgecircle4iii}
\end{figure}
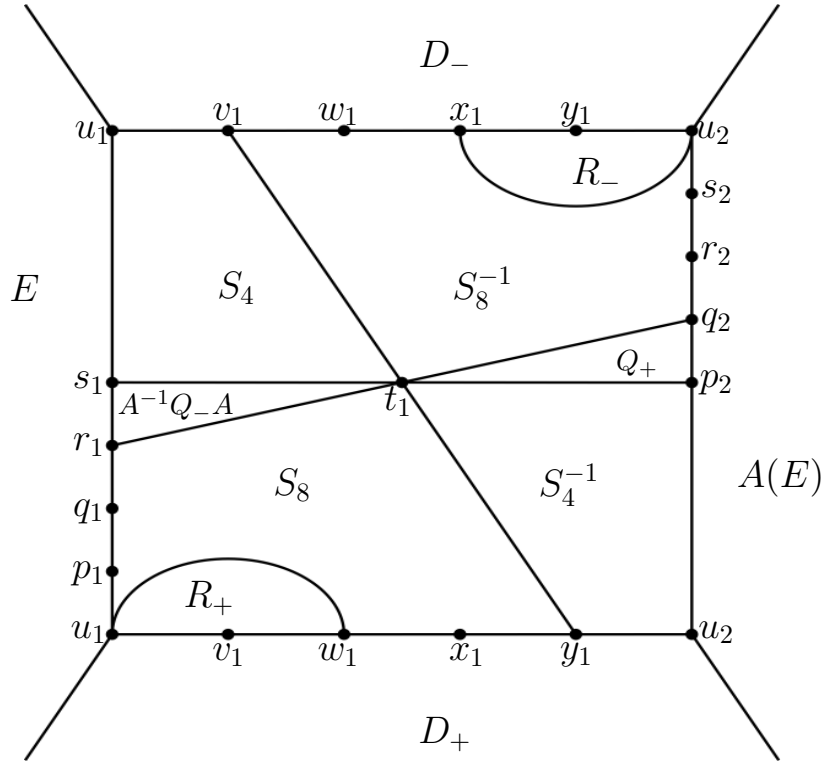

%According to the above analysis, the side-pairings on $N$ are
The 3-manifold $\mathscr{M}_{4,4,\infty; \infty}$ is the quotient space of $N$.  The side-pairings on $N$ are
\begin{equation*}
    \begin{aligned}
    %\begin{align*}
        f_1&: D_{-}\longrightarrow D_{+},\\
        f_2=A&: E\longrightarrow A(E),\\
        f_3=R&: R_+\longrightarrow R_-,\\
        f_4=A^{-1}Q&: Q_{+}\longrightarrow A^{-1}Q_{-}A,\\
	f_5=S^{-1}&: S_8\longrightarrow S_8^{-1},\\
        f_6=S^{-1}&: S_4\longrightarrow S_4^{-1}.\\
    \end{aligned}
\end{equation*}

The correspondences of vertices of the polygons are shown in Table \ref{vertices_corresponding44ii}.
\begin{table}[htbp]
\centering
    \caption{Vertices of the 2-cells involved in the side-pairing maps on $\partial N$.}
    \begin{tabular}{|c|c|c|}
    \hline 
    maps&vertices&vertices
    \\
    \hline  
    $f_1$
    &$u_1, v_1, w_1, x_1, y_1, u_2$
    &$u_1, v_1, w_1, x_1, y_1, u_2$\\
    
    \hline
    $f_2$
    &$u_1, s_1, r_1, q_1, p_1$
    &$u_2, s_2, r_2, q_2, p_2$\\

    \hline
    $f_3$&$u_1, v_1, w_1$
    &$u_2, x_1, y_1$\\

    \hline
    $f_4$&$t_1, p_2, q_2$
    &$t_1, r_1, s_1$\\

    \hline
    $f_5$&$w_1, x_1, y_1, t_1, r_1, q_1, p_1, u_1$
    &$v_1, w_1, x_1, u_2, s_2, r_2, q_2, t_1$\\

    \hline
    $f_6$&$t_1, v_1, u_1, s_1$
    &$ u_2, y_1, t_1, p_2$\\
    
    \hline
    \end{tabular}
    \label{vertices_corresponding44ii}
\end{table}

Then we have Table \ref{table:edgecircle4pp}, it gives the ridge circles of the 3-manifold at infinity of $\Delta_{4,4,\infty;\infty}$, and the fundamental group of it. So we have a presentation of $\pi_1(\mathscr{M}_{4,4,\infty; \infty})$ with six generators and eight relations. 

\begin{table}[!htbp]
	\caption{Edge cycles and relations for the 3-manifold $\mathscr{M}_{4,4,\infty; \infty}$.}
	\centering
	\begin{tabular}{c|c|c}
		\toprule
		\textbf{edge} & \textbf{edge cycle}  & \textbf{cycle relation}\\
		\midrule
		$e_{1}$ & $e_1\xrightarrow{f_1} e_{24}  
            \xrightarrow{f_2} e_{18} 
            \xrightarrow{f_1^{-1}}e_7
            \xrightarrow{f_2^{-1}}e_1$
            
            & $f_2^{-1}f_1^{-1}f_2f_1$ \\  [1 ex]
		
    $e_{2}$ & $e_2\xrightarrow{f_1} e_{23} 
    \xrightarrow{f_3} e_{8} 
    \xrightarrow{f_5^{-1}} e_{16} 
    \xrightarrow{f_{6}^{-1}}e_{2}$ 
    & $f_{6}^{-1}f_5^{-1}f_3f_1$ \\  [1 ex]

    $e_{3}$ & $e_3\xrightarrow{f_1} e_{22} 
    \xrightarrow{f_3} e_{5} 
    \xrightarrow{f_1} e_{19} 
    \xrightarrow{f_5}e_{4}
    \xrightarrow{f_1} e_{20} 
    \xrightarrow{f_5}e_{3}$ 
    & $f_5f_1f_5f_1f_3f_1$ \\  [1 ex]

    $e_{6}$ & $e_6\xrightarrow{f_1} e_{17} \xrightarrow{f_6^{-1}} e_{32} \xrightarrow{f_5^{-1}}e_{21}
    \xrightarrow{f_3}e_{6}$  
    & $ f_3f_5^{-1}f_6^{-1}f_1$ \\  [1 ex]

    $e_{9}$ & $e_9\xrightarrow{f_2^{-1}} e_{29} \xrightarrow{f_6} e_{15} \xrightarrow{f_4} e_{30} \xrightarrow{f_5}e_{9}$  
    & $f_5f_4f_6f_2^{-1}$ \\  [1 ex]

    $e_{10}$ & $e_{10}
    \xrightarrow{f_2^{-1}} e_{28}
    \xrightarrow{f_{4}^{-1}} e_{12}  
    \xrightarrow{f_2^{-1}}e_{26}
    \xrightarrow{f_5}e_{11}
    \xrightarrow{f_2^{-1}}e_{27}
    \xrightarrow{f_5}e_{10}$
    & $f_5f_2^{-1}f_5f_2^{-1}f_4^{-1}f_{2}^{-1}$ \\  [1 ex]

    $e_{13}$ & $e_{13}\xrightarrow{f_2^{-1}} e_{25} 
    \xrightarrow{f_5} e_{14}
    \xrightarrow{f_4}e_{31}
    \xrightarrow{f_6}e_{13}$
    & $ f_6f_4f_5f_2^{-1}$ \\  
    \bottomrule
	\end{tabular}
	\label{table:edgecircle4pp}
\end{table}

By using Magma, we obtain the following simplified presentation:
\[
\pi_1(\mathscr{M}_{4,4,\infty;\,\infty})
=\left\langle x_1, x_2, x_3 \;\middle|\;
x_2^{-1}x_1^{-1}x_2 x_1,\;
(x_3^{-1}x_1^{-1})^{2} x_3^{-1} x_2^{-1} (x_3 x_2^{-1})^{3}
\right\rangle,
\]
where $x_1=f_1$, $x_2=f_2$, and $x_3=f_5$.

There is an isomorphism
\[
\phi:\pi_1(\mathscr{M}_{4,4,\infty;\,\infty}) \longrightarrow \pi_1(s782)
\]
defined by
\[
\phi(x_1)=c^{-1}a, \qquad
\phi(x_2)=bc,  \qquad
\phi(x_3)=c.
\]

The inverse isomorphism
\[
\phi^{-1}: \pi_1(s782)\longrightarrow \pi_1(\mathscr{M}_{4,4,\infty;\,\infty})
\]
is given by
\[
\phi^{-1}(a)=x_3 x_1, \qquad
\phi^{-1}(b)=x_2 x_3^{-1}, \qquad
\phi^{-1}(c)=x_3.
\]

With the isomorphism $\phi$ at hand, Jiang, Wang and Xie proved the following.

\begin{thm}[Theorem 5.22 in \cite{JWX2023}] \label{thm:44p}
Let $\Sigma=\langle I_1 I_2,\, I_2 I_3 \rangle$ be the even subgroup of the complex hyperbolic
 triangle group $\Delta_{4,4,\infty;\infty}$.  
Then the 3-manifold $\mathscr{M}_{4,4,\infty;\,\infty}$ at infinity of the even subgroup
 $\Sigma$ is the 3-manifold $s782$ in the SnapPy census.
\end{thm}

Note that the Ford domain considered in \cite{JWX2023} is centered at the fixed point of $I_1 I_3 I_2 I_3$, not at the fixed point of $I_1 I_2$ as in our present setting.  
By \cite{Thompson:2010}, the group $\Delta_{4,4,\infty;\infty}$ admits a $\mathbb{Z}_2$--symmetry that exchanges the two conjugacy classes of parabolic elements.  
Therefore, the Ford domains of $\Sigma$ centered at the fixed point of $I_1 I_3 I_2 I_3$ and at the fixed point of $I_1 I_2$ have the same combinatorial and topological structure.

In Section~\ref{section:dirichlet}, we will prove that the even subgroup of the complex
 hyperbolic triangle group $\Delta_{4,4,n;\infty}$ has a Dirichlet domain whose local
  topology and combinatorial structure coincide with those of the Ford domain of the even
   subgroup of $\Delta_{4,4,\infty;\infty}$.  Consequently, the corresponding 3-manifold 
   at infinity is obtained from $s782$ by an appropriate Dehn filling.

%\textcolor{red}{this section is complete. }

\section{Dirichlet domain of the even subgroup of \texorpdfstring{$\Delta_{4,4,n;\infty}$}{}} 
\label{section:dirichlet}

In this section, we analyze the Dirichlet domain of the even subgroup of 
$\Delta_{4,4,n;\infty}$.  
The construction involves a substantial amount of explicit and rather intricate analytic computations, as well as geometric estimates concerning the intersections of various bisectors.  
Readers whose primary interest lies in the spherical CR Dehn filling results may take the conclusions of this section for granted and proceed directly to Sections~6 and~7.

We will prove Theorem \ref{thm:44nDirichlet} for  $n\geq 7$, whence  $\Delta_{4,4,n;\infty}$ is a discrete group  for $n\geq 7$.  In the proof presented here, we require the condition $n \geq 7$ in Propostions \ref{prop:kminus} and   \ref{prop:kstar}. However, the conclusions of   Theorem \ref{thm:44nDirichlet} also holds for the cases $n=5$ and $n=6$. Their  proofs   are slightly simpler, so we omit their proof herein. 
More importantly, we demonstrate that the Ford domain of $\Delta_{4,4,\infty;\infty}$ and the Dirichlet domain of $\Delta_{4,4,n;\infty}$ share the same local combinatoric and topology.

In this section, let $n > 6$ be an integer. Define
\begin{equation}\label{definition:csd}
	c = \cos\left(\frac{\pi}{n}\right), \quad 
	s = \sin\left(\frac{\pi}{n}\right), \quad 
	d = \frac{1}{1 - c^2},
\end{equation}
and let $a$ and $b$ be nonnegative real numbers such that
\begin{equation}\label{definition:ab}
	a^2 = \frac{1}{2}(d - 1)\left(1 + d + \frac{d - 1}{c}\right) \quad \text{and} \quad 
	b^2 = \frac{1}{2}(d - 1)\left(1 + d - \frac{d - 1}{c}\right).
\end{equation}

\subsection{A  matrix representation of \texorpdfstring{$\Delta_{4,4,n;\infty}$}{}}
In this section, we construct a Dirichlet domain for the complex hyperbolic triangle
 group $\Delta_{4,4,n;\infty}$ in the ball model. Although a matrix representation of
  this group was already introduced, we present here an alternative realization that
   facilitates our construction. Let $I_1$, $I_2$, and $I_3$ be the complex reflection
    generators. Following a parametrization analogous to that in \cite{ParkerWX:2016},
	 we establish the following result.

\begin{prop}
	Let $n \geq 5$ be an integer, and let $\Gamma = \langle I_1, I_2, I_3 \rangle$ be the complex hyperbolic triangle group $\Delta_{4,4,n;\infty}$, with parameters $c, s, d, a, b$ defined as in (\ref{definition:csd}) and (\ref{definition:ab}). Then, up to conjugacy, the generators of $\Gamma$ are given by
	\begin{equation*}\label{I1-I2_I3-eq}
		I_1 = \begin{bmatrix}
			-c & s & 0 \\
			s & c & 0 \\
			0 & 0 & -1
		\end{bmatrix}, \quad
		I_2 = \begin{bmatrix}
			-c & -s & 0 \\
			-s & c & 0 \\
			0 & 0 & -1
		\end{bmatrix}, \quad
		I_3 = \begin{bmatrix}
			\dfrac{a^2}{d-1} - 1 & -\dfrac{abi}{d-1} & -a \\[8pt]
			\dfrac{abi}{d-1} & \dfrac{b^2}{d-1} - 1 & -bi \\[8pt]
			a & -bi & -d
		\end{bmatrix}.
	\end{equation*}
\end{prop}
%\textcolor{red}{assign value of d in the equations of a and b, so, simlify the equations of a and b. }

\begin{proof}
	By a straightforward verification, $I_1$, $I_2$ and $I_3$ are complex reflections in
	 $\mathbf{PU}(2,1)$ satisfying the required group relations: $(I_3I_1)^4 = (I_2I_3)^4 = \mathrm{id}$,
	  $(I_1I_2)^n = \mathrm{id}$, and $I_1I_3I_2I_3$ is parabolic. The parameter conditions for
	   $a$, $b$, and $d$ are derived identically to \cite{ParkerWX:2016}, and are thus omitted.
\end{proof}

\subsection{The combinatorics of the Dirichlet domain}
Let
\begin{equation}\label{A-eq}
	A = I_1I_2 = \begin{bmatrix}
		c^2 - s^2 & 2cs & 0 \\
		-2cs & c^2 - s^2 & 0 \\
		0 & 0 & 1
	\end{bmatrix},
\end{equation}
which is an element of order $n$ and conjugate to
\begin{equation}\label{I1I2-conjugate}
	\begin{bmatrix}
		e^{2\pi i/n} & 0 & 0 \\
		0 & e^{-2\pi i/n} & 0 \\
		0 & 0 & 1
	\end{bmatrix}.
\end{equation}
Similarly to the group $\Delta_{4,4,\infty;\infty}$, we define
\[
S = I_2I_3, \quad R = S^2, \quad Q = (I_1I_3)^2,
\]
in  the group  $\Delta_{4,4,n;\infty}$, which have orders 4, 2, and 2, respectively. Note that the point $o = (0,0,1)^T \in \hc$ is the fixed point of $A$.

For the even subgroup $\Sigma = \langle I_1I_2, I_2I_3 \rangle$, we define a polyhedron $\mathcal{D}$ centered at $o$, bounded by sides lying on $4n$ bisectors. The only difference between the group $\Sigma$ here and the group $\Sigma$ in Section~\ref{section:forddomain} is that here $A$ has finite order.

The fixed point of $I_1I_3I_2I_3 = AS^2$ is given by
\begin{equation*}
	p_{AS^2} = \begin{bmatrix}
		-as(1+c) + ib(1+c)^2 \\
		as^2 - bs(1+c)i \\
		\dfrac{ - (d-1)^{2}(1 + c) s - i a b s^{2} + iab(1 + c)^{2} }{ d-1 }
	\end{bmatrix}.
\end{equation*}

\begin{defn}\label{def:involutions}
	For $k \in \mathbb{Z}$, define the following $4n$ elements of $\Sigma$:
	\[
	A_k^{+} = A^k S A^{-k}, \quad 
	A_k^{-} = A^k S^{-1} A^{-k}, \quad 
	A_k^{\star} = A^k R A^{-k}, \quad 
	A_k^{\diamond} = A^k Q A^{-k}.
	\]
\end{defn}

\begin{defn}\label{def:bisectors}
	Define the bisectors
	\begin{align*}
		\mathcal{B}^{+}_{k} &= \mathcal{B}(o, A^{+}_{k}(o)), &
		\mathcal{B}^{-}_{k} &= \mathcal{B}(o, A^{-}_{k}(o)), \\
		\mathcal{B}^{\star}_{k} &= \mathcal{B}(o, A^{\star}_{k}(o)), &
		\mathcal{B}^{\diamond}_{k} &= \mathcal{B}(o, A^{\diamond}_{k}(o)),
	\end{align*}
	and denote their boundaries by $\partial_{\infty}\mathcal{B}^{+}_{k}$, $\partial_{\infty}\mathcal{B}^{-}_{k}$, $\partial_{\infty}\mathcal{B}^{\star}_{k}$, $\partial_{\infty}\mathcal{B}^{\diamond}_{k}$, respectively.
\end{defn}

Since $A=I_1I_2$ has order $n$, we have the periodicity relations
\[
\mathcal{B}_{k+n}^{+} = \mathcal{B}_k^{+}, \quad 
\mathcal{B}_{k+n}^{-} = \mathcal{B}_k^{-}, \quad 
\mathcal{B}_{k+n}^{\star} = \mathcal{B}_k^{\star}, \quad 
\mathcal{B}_{k+n}^{\diamond} = \mathcal{B}_k^{\diamond}.
\]
Thus, there are exactly \(4n\) distinct bisectors, and the index \(k\) may be taken modulo \(n\).
See Figure~\ref{figure:dirichlet44n} for an illustration of these bisectors.
Throughout, we fix \(k = 0, 1, \dots, n-1\) and \(m = 1, 2, \dots, n-1\).

\begin{defn}\label{def:Dirichletdomain44ni}
	Define the polyhedron $\mathcal{D}$ in $\hc$ by
	\[
	\mathcal{D} = \left\{ p \in \overline{\hc} \,\middle|\, 
	\begin{aligned}
		& |\langle \mathbf{p}, o \rangle| \leq |\langle \mathbf{p}, g(o) \rangle|, \\
		& \forall g \in \{ A_k^{+}, A_k^{-}, A_k^{\star}, A_k^{\diamond} \mid k = 0,1,\dots, n-1 \}
	\end{aligned}
	\right\}.
	\]
	Denote $\mathcal{D}_{\infty} = \mathcal{D} \cap \partial \hc$ and $\partial\mathcal{D}_{\infty}$ its boundary.
\end{defn}

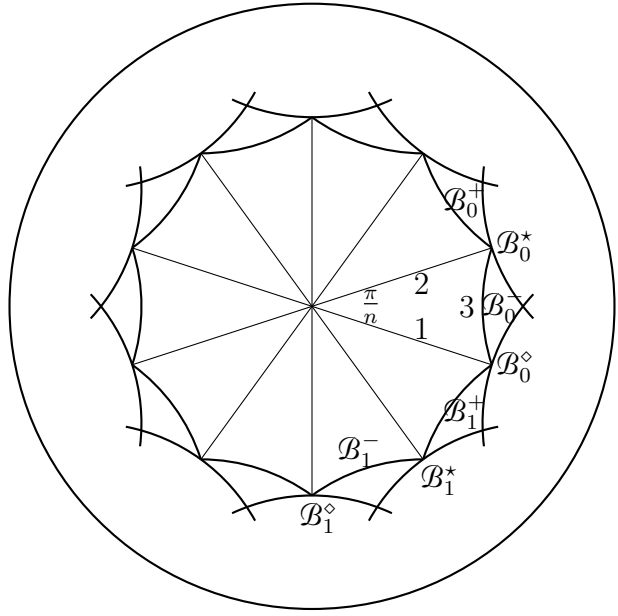
\begin{figure}
	\begin{center}
		\begin{tikzpicture}
			
			\draw[thick] (0,0) circle (4 cm);
			
			\draw (0,0) -- (18:2.5);
			\draw (0,0) -- (54:2.5);
			\draw (0,0) -- (90:2.5);
			\draw (0,0) -- (126:2.5);
			\draw (0,0) -- (162:2.5);
			\draw (0,0) -- (198:2.5);
			\draw (0,0) -- (234:2.5);
			\draw (0,0) -- (270:2.5);
			\draw (0,0) -- (306:2.5);
			\draw (0,0) -- (342:2.5);

			\draw [thick] (18:2.5) arc (162:198:2.5);
			\draw [thick] (54:2.5) arc (198:234:2.5);
			\draw [thick] (90:2.5) arc (234:270:2.5);
			\draw [thick] (126:2.5) arc (270:306:2.5);
			\draw [thick] (162:2.5) arc (306:342:2.5);
			\draw [thick] (162:2.5) arc (18:-18:2.5);
			\draw [thick] (198:2.5) arc (54:18:2.5);
			\draw [thick] (234:2.5) arc (90:54:2.5);
			\draw [thick] (270:2.5) arc (126:90:2.5);
			\draw [thick] (306:2.5) arc (162:126:2.5);

			\draw [thick] (18:2.5) arc (198:223:2.5);		
			\draw [thick] (18:2.5) arc (198:173:2.5);
			\draw [thick] (54:2.5) arc (234:259:2.5);
			\draw [thick] (54:2.5) arc (234:209:2.5);
			\draw [thick] (90:2.5) arc (270:295:2.5);
			\draw [thick] (90:2.5) arc (270:245:2.5);
			\draw [thick] (126:2.5) arc (306:331:2.5);	
			\draw [thick] (126:2.5) arc (306:281:2.5);
			\draw [thick] (162:2.5) arc (-18:7:2.5);
			\draw [thick] (162:2.5) arc (-18:-43:2.5);
			\draw [thick] (198:2.5) arc (18:43:2.5);	
			\draw [thick] (198:2.5) arc (18:-7:2.5);
			\draw [thick] (234:2.5) arc (54:79:2.5);	
			\draw [thick] (234:2.5) arc (54:29:2.5);
			\draw [thick] (270:2.5) arc (90:65:2.5);		
			\draw [thick] (270:2.5) arc (90:115:2.5);
			\draw [thick] (306:2.5) arc (126:151:2.5);
			\draw [thick] (306:2.5) arc (126:101:2.5);
			\draw [thick] (-18:2.5) arc (162:187:2.5);
			\draw [thick] (-18:2.5) arc (162:137:2.5);
			
			\coordinate [label=right:$\frac{\pi}{n}$] (n) at (0.5,0);
			\coordinate [label=right:$1$] (1) at (1.2,-0.3);
			\coordinate [label=right:$2$] (2) at (1.2,0.3);
			\coordinate [label=right:$3$] (3) at (1.8,0);
			\coordinate [label=right:${\mathcal B}_0^{-}$] (B_0n) at (2.1,0);
			\coordinate [label=right:${\mathcal B}_0^{+}$] (B_0p) at (1.6,1.4);
			\coordinate [label=right:${\mathcal B}_0^{\star}$] (B_0p) at (2.3,0.8);
			\coordinate [label=right:${\mathcal B}_0^{\diamond}$] (B_0p) at (2.3,-0.8);
			\coordinate [label=right:${\mathcal B}_1^{+}$] (B_0p) at (1.6,-1.4);
			\coordinate [label=right:${\mathcal B}_1^{\star}$] (B_0p) at (1.3,-2.3);
			\coordinate [label=right:${\mathcal B}_1^{-}$] (B_0p) at (0.2,-1.9);
			\coordinate [label=right:${\mathcal B}_1^{\diamond}$] (B_0p) at (-0.3,-2.8);
		\end{tikzpicture}
		
	\end{center}
	\caption{A schematic view of the Dirichlet domain of complex  triangle group $\Delta(4,4,n)$.}
	\label{figure:dirichlet44n}
\end{figure}

The main result of this section is the following.

\begin{thm} \label{thm:44nDirichlet}
	For $n \geq 5$, let $\Sigma = \langle I_1I_2, I_2I_3 \rangle$ be the even subgroup of the complex hyperbolic triangle group $\Delta_{4,4,n;\infty}$. 
	Then $\mathcal{D}$ is the Dirichlet domain of $\Sigma$ centered at $o \in \hc$. 
	Moreover, $\Sigma$ is discrete and has the presentation
	\[
	\langle S, A \mid A^{n} = S^4 = (AS)^4 = \mathrm{id} \rangle.
	\]
\end{thm}

Both Definition~\ref{def:Dirichletdomain44ni} and Theorem~\ref{thm:44nDirichlet} for $\Delta_{4,4,n;\infty}$ should be compared with their counterparts, Definition~\ref{def:Folddomain44ii} and Theorem~\ref{thm:44pFord}, for the group $\Delta_{4,4,\infty;\infty}$.

\begin{defn}
	For $k \in \{0, 1, \dots, n-1\}$, we define the following 3-sides of $\mathcal{D}$:
	\begin{itemize}
		\item $\mathfrak{s}_{k}^{+}=\mathcal{B}_{k}^{+} \cap \mathcal{D}$;
		\item $\mathfrak{s}_{k}^{-}=\mathcal{B}_{k}^{-}  \cap \mathcal{D}$;
		\item $\mathfrak{s}^{\star}_{k}= \mathcal{B}^{\star}_{k} \cap \mathcal{D}$;
		\item $\mathfrak{s}^{\diamond}_{k}=\mathcal{B}^{\diamond}_{k}  \cap \mathcal{D}$.
	\end{itemize}

%\begin{itemize}
%	\item $\mathfrak{s}_{k}^{+}$: contained in the isometric sphere $\mathcal{B}_{k}^{+}$,
%	\item $\mathfrak{s}_{k}^{-}$: contained in $\mathcal{B}_{k}^{-}$,
%	\item $\mathfrak{s}^{\star}_{k}$: contained in $\mathcal{B}^{\star}_{k}$,
%	\item $\mathfrak{s}^{\diamond}_{k}$: contained in $\mathcal{B}^{\diamond}_{k}$.
%\end{itemize}
	Their ideal boundaries are denoted by
	\[
	\partial_{\infty}\mathfrak{s}_{k}^{+} = \mathfrak{s}_{k}^{+} \cap \partial \hc, \quad
	\partial_{\infty}\mathfrak{s}_{k}^{-} = \mathfrak{s}_{k}^{-} \cap \partial \hc, \quad
	\partial_{\infty}\mathfrak{s}^{\star}_{k} = \mathfrak{s}^{\star}_{k} \cap \partial \hc, \quad
	\partial_{\infty}\mathfrak{s}^{\diamond}_{k} = \mathfrak{s}^{\diamond}_{k} \cap \partial \hc.
	\]
\end{defn}

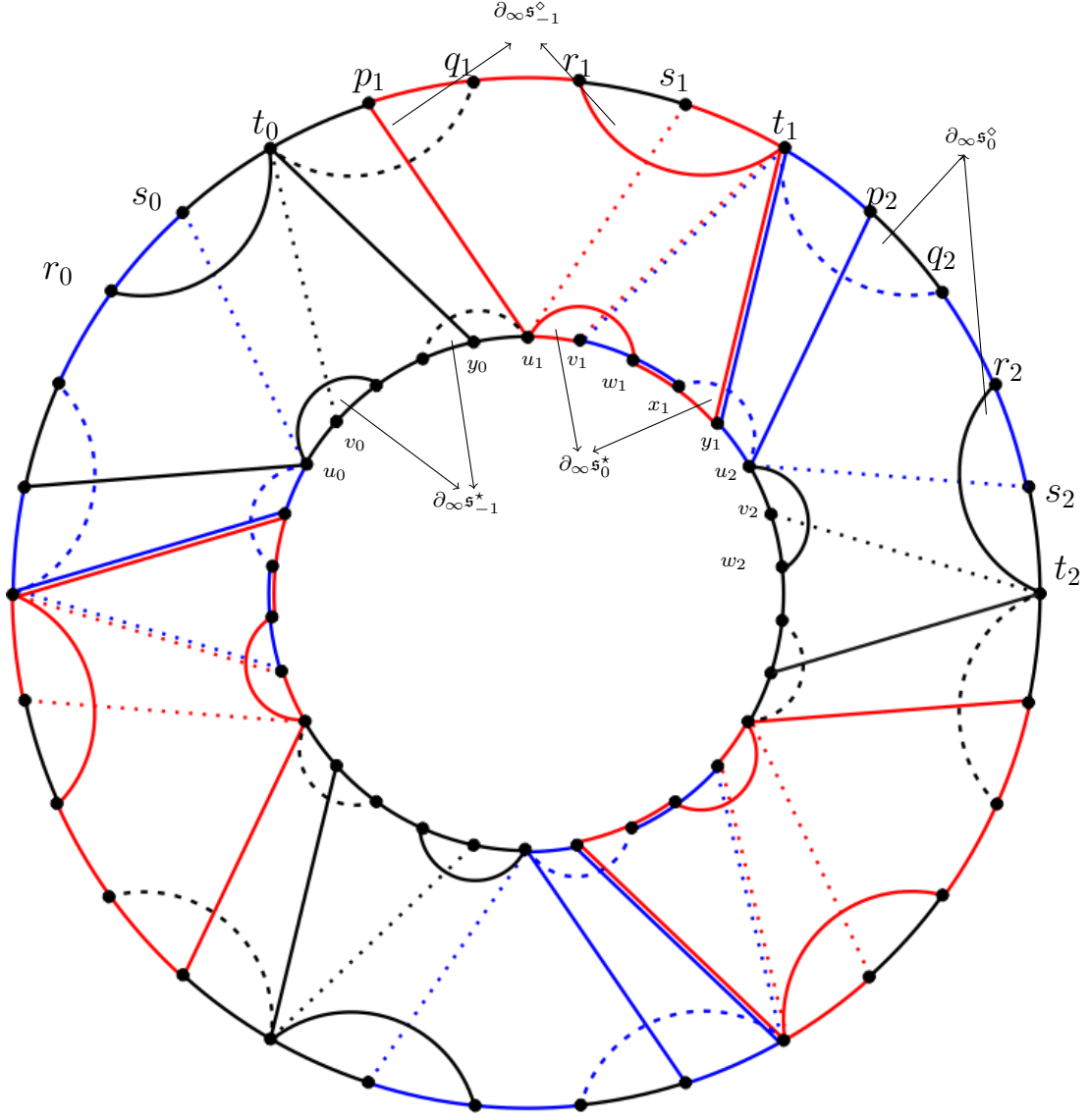
\begin{figure}[htbp]
	\begin{center}
		\begin{tikzpicture}
		\node at (0,0) {\includegraphics[width=14cm,height=14.0cm]{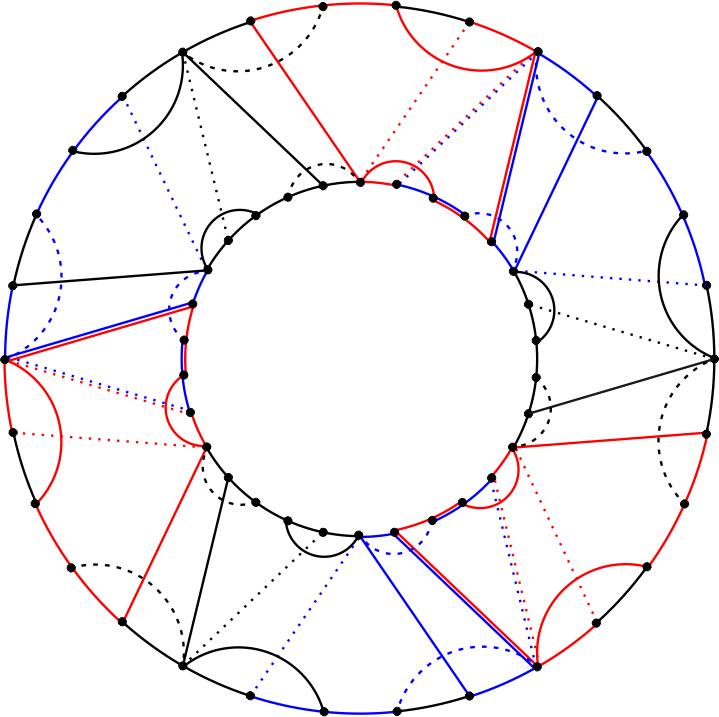}};

\node at (-6.3,4.3){\large $r_0$};
\node at (-5.1,5.3){\large $s_0$};
\node at (-3.5,6.3){\large $t_0$};
\node at (-2.1,6.9){\large $p_1$};
\node at (-0.9,7.1){\large $q_1$};
\node at (0.7,7.1){\large $r_1$};
\node at (2.0,6.9){\large $s_1$};
\node at (3.5,6.3){\large $t_1$};
  		
\node at (4.8,5.3){\large $p_2$};
\node at (5.6,4.5){\large $q_2$};
\node at (6.5,3.0){\large $r_2$};
\node at (7.2,1.3){\large $s_2$};
\node at (7.3,0.3){\large $t_2$};

\node at (0.1,3.1){\tiny $u_1$};
\node at (0.7,3.1){\tiny $v_1$};
\node at (1.2,2.8){\tiny $w_1$};
\node at (1.8,2.5){\tiny $x_1$};
\node at (2.5,2.0){\tiny $y_1$};
\node at (2.7,1.6){\tiny $u_2$};
\node at (3.0,1.1){\tiny $v_2$};
\node at (2.8,0.4){\tiny $w_2$};

\draw[->] (0.4,3.55)--(0.7,1.9);
\draw[->] (2.5,2.6)--(0.9,1.9);
\node at (0.8,1.7){\tiny $\partial_{\infty}\mathfrak{s}_{0}^{\star}$};

\draw[->] (-2.5,2.6)--(-0.9,1.4);
\draw[->] (-1.0,3.4)--(-0.7,1.4);
\node at (-0.8,1.2){\tiny $\partial_{\infty}\mathfrak{s}_{-1}^{\star}$};

\node at (0.02,7.8){\tiny $\partial_{\infty}\mathfrak{s}_{-1}^{\diamond}$};
\draw[->] (-1.8,6.3)--(-0.2,7.4);
\draw[->] (1.2,6.3)--(0.2,7.4);
  
\node at (6.0,6.12){\tiny $\partial_{\infty}\mathfrak{s}_{0}^{\diamond}$};
\draw[->] (4.8,4.7)--(5.9,5.9);
\draw[->] (6.2,2.4)--(5.9,5.9);

\node at (-2.57,1.6){\tiny $u_0$};
\node at (-2.3,2.0){\tiny $v_0$};
\node at (-0.65,3.05){\tiny $y_0$};
\end{tikzpicture}
\end{center}
\caption{An abstract picture of the Dirichlet domain of the even subgroup of $\Delta_{4,4,6;\infty}$. Although the proof is presented for $n \geq 7$, the combinatorial picture for $n=6$ is qualitatively similar and is shown here for illustration. The red regions are $\partial_{\infty}\mathfrak{s}_{0}^{+}$, $\partial_{\infty}\mathfrak{s}_{2}^{+}$ and $\partial_{\infty}\mathfrak{s}_{4}^{+}$, respectively. 
The blue regions are $\partial_{\infty}\mathfrak{s}_{0}^{-}$, $\partial_{\infty}\mathfrak{s}_{2}^{-}$  and $\partial_{\infty}\mathfrak{s}_{4}^{-}$, respectively.}
\label{figure:446dirichlet}
\end{figure}

\begin{figure}[htbp]
	\begin{center}
		\begin{tikzpicture}
		\node at (0,0) {\includegraphics[width=8cm,height=8.0cm]{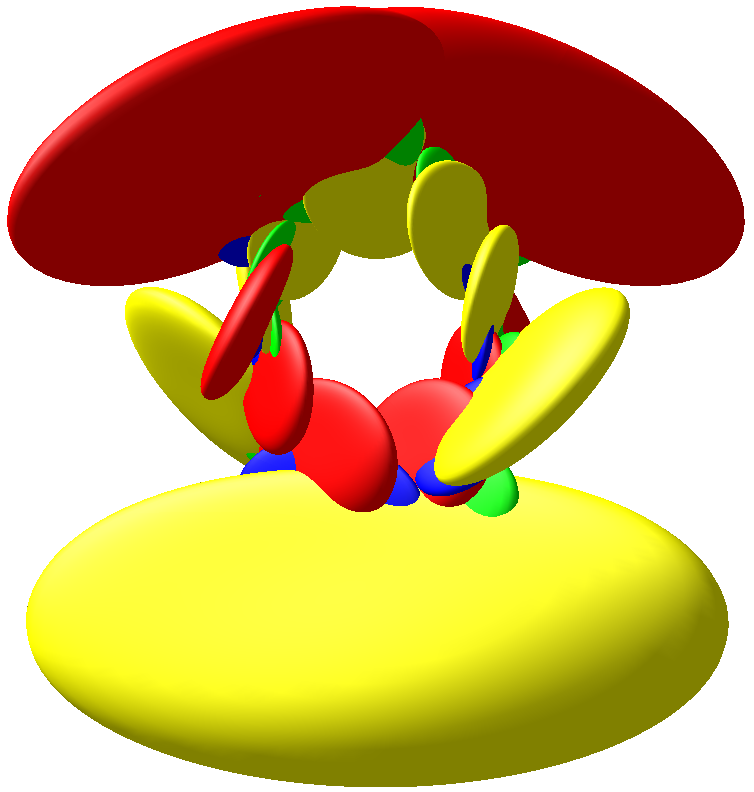}};
\end{tikzpicture}
\end{center}
\caption{A  realistic view of the ideal boundary $\mathcal{D}_{\infty}$ of the Dirichlet domain $\mathcal{D}$ of
$\Delta_{4,4,8;\infty}$. It is the solid torus outside all of the spheres. }
\label{figure:448realistic}
\end{figure}

Before proceeding, the reader may refer to Figures~\ref{figure:448realistic} and~\ref{figure:446dirichlet}. 
Figure~\ref{figure:448realistic} shows a realistic view of the ideal boundary of the Dirichlet domain $\mathcal{D}$ for $\Delta_{4,4,8;\infty}$, where the sides $\partial_{\infty}\mathfrak{s}_{k}^{+}$, $\partial_{\infty}\mathfrak{s}_{k}^{-}$, $\partial_{\infty}\mathfrak{s}^{\star}_{k}$, and $\partial_{\infty}\mathfrak{s}^{\diamond}_{k}$ are colored in red, yellow, green, and blue, respectively. 
Figure~\ref{figure:446dirichlet} provides an abstract schematic of the Dirichlet domain for the even subgroup of $\Delta_{4,4,6;\infty}$, which should be compared with Figure~\ref{figure:44iiford} in Section~\ref{section:forddomain}.

\subsection{Symmetry for  \texorpdfstring{$\mathcal{D}$}{}}
The combinatorial structure of the polyhedron $\mathcal{D}$ is preserved by  certain symmetries. Consider the anti-holomorphic isometry defined by
\[
\iota:\left[\begin{matrix} z_1 \\ z_2 \\ z_3 \end{matrix}\right]
\longmapsto \left[\begin{matrix} \bar{z}_1 \\
	-\bar{z}_2 \\ \bar{z}_3 \end{matrix}\right].
\]
Clearly, $\iota$ is an involution, i.e., $\iota^2 = \mathrm{id}$.

\begin{lem} \label{lem:iota-action}
	Let $\mathbf{n}_1, \mathbf{n}_2, \mathbf{n}_3$ be the polar vectors of $I_1, I_2, I_3$, respectively:
	\[
	\mathbf{n}_1 = \begin{bmatrix} 1 - c \\ s \\ 0 \end{bmatrix}, \quad 
	\mathbf{n}_2 = \begin{bmatrix} c - 1 \\ s \\ 0 \end{bmatrix}, \quad 
	\mathbf{n}_3 = \begin{bmatrix} a \\ bi \\ d - 1 \end{bmatrix}.
	\]
	Then $\iota(\mathbf{n}_1) = \mathbf{n}_2$ and $\iota(\mathbf{n}_3) = \mathbf{n}_3$.
\end{lem}

\begin{proof}
	This is verified by direct computation using the definition of $\iota$.
\end{proof}

As an immediate corollary, we obtain the conjugation relations:
\[
\iota I_1 \iota = I_2, \quad \iota I_2 \iota = I_1, \quad \iota I_3 \iota = I_3.
\]

\begin{prop} \label{prop:symmetry}
	Let $\mathcal{B}_k^{+}$, $\mathcal{B}_k^{-}$, $\mathcal{B}_k^{\star}$, and $\mathcal{B}_k^{\diamond}$ be defined as in Definition~\ref{def:bisectors}. Then for each $k$ modulo $n$, the following symmetries hold:
	\begin{itemize}
		\item $I_2(\mathcal{B}_k^{+})=\mathcal{B}_{-k}^{-}$;
		\item $I_2(\mathcal{B}_k^{\star})=\mathcal{B}_{-k}^{\star}$;
		\item $I_2(\mathcal{B}_k^{\diamond})=\mathcal{B}_{-k-1}^{\diamond}$;
		\item $\iota(\mathcal{B}_k^{+})=\mathcal{B}_{1-k}^{+}$;
		\item $\iota(\mathcal{B}_k^{-})=\mathcal{B}_{-k}^{-}$;
		\item $\iota(\mathcal{B}_k^{\star})=\mathcal{B}_{-k}^{\diamond}$.
	\end{itemize}
\end{prop}

\begin{proof}
	We prove two representative cases; the remaining identities follow by similar arguments.
	
	For the relation $I_2(\mathcal{B}_k^{+}) = \mathcal{B}_{-k}^{-}$:
	\begin{align*}
		I_2(\mathcal{B}_k^{+}) 
		&= \mathcal{B}\left(I_2(o), I_2 A_k^{+}(o)\right) 
		= \mathcal{B}\left(I_2(o), I_2 A^{k} S(o)\right) \\
		&= \mathcal{B}\left(o, I_2 (I_1 I_2)^{k} I_2 I_3(o)\right) 
		= \mathcal{B}\left(o, (I_2 I_1)^{k} I_3(o)\right) \\
		&= \mathcal{B}\left(o, (I_2 I_1)^{k} I_3 I_2(o)\right) 
		= \mathcal{B}_{-k}^{-}.
	\end{align*}
	
	For $\iota(\mathcal{B}_k^{+}) = \mathcal{B}_{1-k}^{+}$:
	\begin{align*}
		\iota(\mathcal{B}_k^{+}) 
		&= \mathcal{B}\left(\iota(o), \iota A^{k} S(o)\right) 
		= \mathcal{B}\left(o, \iota (I_1 I_2)^{k} I_2 I_3(o)\right) \\
		&= \mathcal{B}\left(o, (\iota I_1 \iota \cdot \iota I_2 \iota)^{k-1} \iota I_1 \iota \cdot \iota I_3 \iota(o)\right) \\
		&= \mathcal{B}\left(o, (I_2 I_1)^{k-1} I_2 I_3(o)\right) 
		= \mathcal{B}\left(o, A^{-(k-1)} S(o)\right) 
		= \mathcal{B}_{1-k}^{+}.
	\end{align*}
	
	Here the last equality uses the relation $I_2 I_3 = S$ and the identification $A^{-k} S A^{k} = A_{-(1-k)}^{+} = A_{1-k}^{+}$ modulo $n$.
	
	The verification of the other identities proceeds similarly, employing the fundamental relations $\iota I_1 \iota = I_2$, $\iota I_2 \iota = I_1$, $\iota I_3 \iota = I_3$, and the group structure of $\Sigma$.
	
\end{proof}

\subsection{The intersections of the bisectors}
Owing to the symmetry of $\mathcal{D}$ and its invariance under the action of $A$, it suffices to verify most statements only for the bisectors $\mathcal{B}_0^{+}$ and $\mathcal{B}_0^{\star}$ alone; the results then extend to the entire family by symmetry. The following propositions will be instrumental in analyzing intersections among these bisectors.

\begin{prop}[Phillips \cite{Phillips}] \label{prop:Phillips}
	Let $G \in \mathbf{SU}(2,1)$ satisfy $\operatorname{Re}(\operatorname{tr}(G)) \geq 3$. Then for any $u \in \hc$, the bisectors $\mathcal{B}(u, G(u))$ and $\mathcal{B}(u, G^{-1}(u))$ are disjoint. Moreover, if $G$ is parabolic, these bisectors are asymptotic to the parabolic fixed point of $G$. Consequently, the Dirichlet domain for $\langle G \rangle$ centered at $u$ has exactly two sides.
\end{prop}

\begin{prop}[Lemma~2.3 in \cite{DPP:2016}] \label{prop:complex-bisector}
	Let $\mathcal{B}$ be a bisector and $C$ a complex line orthogonal to a complex slice of $\mathcal{B}$. Then $C \cap \mathcal{B}$ is a real geodesic contained in a real slice of $\mathcal{B}$.
\end{prop}

Let $\mathcal{B}_k^{+}$, $\mathcal{B}_k^{-}$, $\mathcal{B}_k^{\star}$, and $\mathcal{B}_k^{\diamond}$ be the bisectors defined in Definition~\ref{def:bisectors}. We now describe their intersection patterns in the following propositions.

First, we establish global combinatorial conditions for the bisectors. Define
\[
c_m = \cos\left(\frac{2\pi m}{n}\right), \quad s_m = \sin\left(\frac{2\pi m}{n}\right).
\]
Then
\begin{equation}\label{Am-eq}
	A^m = \begin{bmatrix}
		c_m & s_m & 0 \\
		-s_m & c_m & 0 \\
		0 & 0 & 1
	\end{bmatrix}.
\end{equation}

We focus primarily on the case $n \geq 7$, as the cases $4 \leq n \leq 6$ have
 been treated separately.

\begin{prop} \label{prop:kminus}
For each non-zero residue class \( m \) modulo \( n \), the following statements hold:
\begin{enumerate}[(i)]
    \item \label{item:kminus:plus-plus} 
        \(\mathcal{B}_0^{+} \cap \mathcal{B}_{m}^{+} = \emptyset\);
    \item \label{item:kminus:plus-minus} 
        \(\mathcal{B}_0^{+} \cap \mathcal{B}_{m}^{-} = \emptyset\) except when \( m = 0 \) or \( m = n-1 \);
    \item \label{item:kminus:plus-star} 
        \(\mathcal{B}_0^{+} \cap \mathcal{B}_{m}^{\star} = \emptyset\) except when \( m = 0 \) or \( m = n-1 \);
    \item \label{item:kminus:plus-diamond} 
        \(\mathcal{B}_0^{+} \cap \mathcal{B}_{m}^{\diamond} = \emptyset\) except when \( m = 0 \) or \( m = n-1 \).
\end{enumerate}
\end{prop}

\begin{proof}
We begin by observing that Item~\ref{item:kminus:plus-diamond} follows immediately from Item~\ref{item:kminus:plus-star} via the anti-holomorphic isometry \(\iota\). 

\textbf{Proof of Item~\ref{item:kminus:plus-plus}.} 
Using the relation 
\[
\mathcal{B}_{0}^{+} \cap \mathcal{B}_{-m}^{+} = A^{-m}(\mathcal{B}_{0}^{+} \cap \mathcal{B}_{m}^{+}),
\]
we may restrict to the case \( 1 \leq m \leq \lfloor n/2 \rfloor \), which ensures \( 0 \leq s_m < 1 \). We divide the argument into two subcases.

\textbf{Case 1: \( m = 1 \).} 
The involution \(\iota\) interchanges the bisectors \(\mathcal{B}_{0}^{+}\) and \(\mathcal{B}_{1}^{+}\). If their intersection is non-empty, it must lie in the real plane fixed by \(\iota\), whose points have the form \( q_{u,v} = (u, v i, 1)^t \) with \( u, v \in \mathbb{R} \). Assume for contradiction that \( q_{u,v} \in \mathcal{B}_{0}^{+} \cap \mathcal{B}_{1}^{+} \). Then the bisector conditions yield:
\[
|\langle q_{u,v}, S(o)\rangle|^2 = |\langle q_{u,v}, o\rangle|^2, \quad
|\langle q_{u,v}, AS(o)\rangle|^2 = |\langle q_{u,v}, o\rangle|^2.
\]
A direct computation shows these equations are equivalent. Define the function:
\[
f(u,v) = |\langle q_{u,v}, S(o)\rangle|^2 - |\langle q_{u,v}, o\rangle|^2.
\]
Explicitly:
\[
f(u,v) = a^2 (c^2 u^2 + s^2 v^2) + d (2 b c v - 2 a c u) - 2 a b u v + b^2 (c^2 v^2 + s^2 u^2) + d^2 - 1.
\]
We analyze \( f(u,v) \) on the closed unit disk \( \{ (u,v) \mid u^2 + v^2 \leq 1 \} \). The Hessian matrix of \( f \) is positive definite, implying strict convexity and the existence of a unique global minimum. This minimum equals \( -1 \) and is attained at:
\[
(u_0, v_0) = \left( \frac{a d}{(d-1)^2}, \frac{b d}{(d-1)^2} \right).
\]
However, we compute:
\[
u_0^2 + v_0^2 = \frac{2 - c^2}{c^6} > 1,
\]
so the minimum over the unit disk occurs on the boundary \( u^2 + v^2 = 1 \). Parametrize the unit circle by \( e^{it} = (\cos t, \sin t) \), obtaining:
\[
f(e^{it}) = \sin^2 t (a^2 s^2 + b^2 c^2) - 2 a \cos t (b \sin t + c d) + \cos^2 t (b^2 + c (d-1)^2) + 2 b c d \sin t + d^2 - 1.
\]
An elementary calculus argument confirms that \( f(e^{it}) > 0 \) for all \( t \in \mathbb{R} \), which yields the desired contradiction.

\textbf{Case 2: \( 2 \leq m \leq \lfloor n/2 \rfloor \).} 
Suppose \( p = [z, w, 1]^t \in \mathcal{B}_{0}^{+} \cap \mathcal{B}_{m}^{+} \). Then the bisector equations are:
\begin{align*}
1 &= | z(ac - s b i) + w(a s + c b i) - d |, \\
1 &= | z(a c_m + b s_m i) + w(-a s_m + b c_m i) - d |.
\end{align*}
Solving this system yields:
\begin{align*}
z &= \frac{1}{s_m(a^2 - b^2)} \big[ a c s_m(d + e^{iu}) - a c_m s(d + e^{iu}) + a s(d + e^{iv}) \\
&\quad - i b c c_m(d + e^{iu}) + i b c(d + e^{iv}) - i b s s_m(d + e^{iu}) \big], \\
w &= \frac{1}{s_m(a^2 - b^2)} \big[ a c c_m(d + e^{iu}) - a c(d + e^{iv}) + a s s_m(d + e^{iu}) \\
&\quad + i b c s_m(d + e^{iu}) - i b c_m s(d + e^{iu}) + i b s(d + e^{iv}) \big].
\end{align*}
Define the quantity:
\[
E = |s_m(a^2 - b^2)|^2 (|z|^2 + |w|^2 - 1).
\]
Expanding this expression gives:
\[
\begin{aligned}
E = & -2 + 2 c_m d^2 + 2 d^4 - 2 c_m d^4 - s_m^2 - 4 b^2 s_m^2 - 4 b^4 s_m^2 + 2 d^2 s_m^2 \\
& - 2 (-1 + c_m) d (-1 + d^2) \cos u - 2 c_m (-1 + d^2) \cos(u - v) \\
& - 2 d \cos v + 2 c_m d \cos v + 2 d^3 \cos v - 2 c_m d^3 \cos v + 4 b^2 d^2 s_m^2 \\
& + 4 a b s_m \sin(u - v) + 4 a b d s_m (\sin u - \sin v) - d^4 s_m^2.
\end{aligned}
\]
Let \( E_0 \) denote the constant terms. Minimizing the trigonometric part over \( u, v \) yields the estimate:
\[
E \ge E_0 - 2\rho_{1,1} - \rho_{1,2},
\]
where we define:
\begin{align*}
\rho_{1,1} &= 2\sqrt{(1 - c_m)^2 d^2 (d^2 - 1)^2 + 4 a^2 b^2 d^2 s_m^2}, \\
\rho_{1,2} &= 2\sqrt{c_m^2 (d^2 - 1)^2 + 4 a^2 b^2 s_m^2}.
\end{align*}
After algebraic simplification, we obtain the function:
\[
F(c, c_m) = \delta_{1,1} + c^6 c_m^2 + (2 c^4 - 4 c^2) c_m - 2 c^2 (1 - c^2)^2 \delta_{1,2} + 4 c^2 (c^2 - 1) \delta_{1,3},
\]
with the auxiliary expressions:
\begin{align*}
\delta_{1,1} &= -2 c^8 + 7 c^6 - 12 c^4 + 8 c^2, \\
\delta_{1,2} &= \sqrt{4 + c^4 + c^2(c_m^2 - 5)}, \\
\delta_{1,3} &= -8 - 2 c^4 + c^2(9 + c_m), \\
\delta_{1,4} &= \sqrt{(c_m - 1)\delta_{1,3}}.
\end{align*}
We aim to prove that \( F(c, c_m) > 0 \) for \( c \in (\frac12, 1) \) and \( -1 < c_m < 8 c^4 - 8 c + 1 \). Observe the bounds:
\[
\delta_{1,2} \leq \sqrt{4 + c^4 - 4 c^2} = 2 - c^2, \quad \text{and} \quad \delta_{1,4} < \sqrt{\frac{3}{2}}(1 - c_m).
\]
Using these estimates, we derive:
\[
\begin{aligned}
F(c, c_m) &\geq c^6 c_m^2 - c^6 + 2 c^4 c_m - 2 c^4 - 4 c^2 c_m + 4 c^2 + 4 c^2 (c^2 - 1) \sqrt{\frac{3}{2}}(1 - c_m) \\
&= c^2 (1 - c_m) \left[ -c^4 (c_m + 1) + (2\sqrt{6} - 2) c^2 - (2\sqrt{6} - 4) \right].
\end{aligned}
\]
Furthermore, we verify:
\[
\begin{aligned}
& -c^4 (c_m + 1) + (2\sqrt{6} - 2) c^2 - (2\sqrt{6} - 4) \\
&\quad \geq 8 c^5 - 8 c^8 + (2\sqrt{6} - 2) c^2 + 4 - 2\sqrt{6} > 0,
\end{aligned}
\]
which completes the proof of Item~\ref{item:kminus:plus-plus}.

\textbf{Proof of Item~\ref{item:kminus:plus-minus}.} 
Note the identities:
\[
\mathcal{B}_0^+ = \mathcal{B}(o, S(o)) = \mathcal{B}(o, (A^{m} S^{-1})^{-1}(o)), \quad
\mathcal{B}_m^- = \mathcal{B}(o, A^m S^{-1}(o)).
\]
We demonstrate that \( A^m S^{-1} \) is hyperbolic with trace at least \( 3 \) for \( m \neq 0, n-1 \). Computing the trace:
\[
\begin{aligned}
\operatorname{tr}(A^m S^{-1}) &= -c c_{2m} \left(-1 + \frac{a^2}{d-1}\right) - c s_{2m} \frac{a b i}{d-1} + s c_{2m} \frac{a b i}{d-1} \\
&\quad - s s_{2m} \left(-1 + \frac{b^2}{d-1}\right) + s s_{2m} \left(-1 + \frac{a^2}{d-1}\right) - s c_{2m} \frac{a b i}{d-1} \\
&\quad + c s_{2m} \frac{a b i}{d-1} - c c_{2m} \left(-1 + \frac{b^2}{d-1}\right) + d \\
&= d + \frac{b^2 - a^2}{d-1} c_{2m+1} = \frac{1 - c c_{2m+1}}{s^2}.
\end{aligned}
\]
For \( m = 0 \) or \( n-1 \), we have \( \operatorname{tr}(A^m S^{-1}) = 1 \), corresponding to order 4 elements. For other values of \( m \):
\[
\operatorname{tr}(A^m S^{-1}) \ge \frac{1 - c c_3}{s^2} = \frac{1 - c^2(1 - 2 s^2) + 2 s^2 c^2}{s^2} = 1 + 4 c^2 > 3.
\]
By Proposition~\ref{prop:Phillips}, the Dirichlet domain for the cyclic group \( \langle A^m S^{-1} \rangle \) has faces \( \mathcal{B}_0^{+} \) and \( \mathcal{B}_{m}^{-} \) that do not intersect.

\textbf{Proof of Item~\ref{item:kminus:plus-star}.} 
One can check that \( \mathcal{B}_0^{+} \) and \( \mathcal{B}_{0}^{\star} \), as well as \( \mathcal{B}_0^{+} \) and \( \mathcal{B}_{-1}^{\star} \), form coequidistant pairs. Define the vectors:
\begin{align*}
v_1'&= (o - S(o)) \boxtimes (o - R(o)), \\
v_2'&= (o - S(o)) \boxtimes (o - A^{-1} R(o)).
\end{align*}
Direct computation shows:
\[
\langle v_1', v_1' \rangle = -4 (d-1)^2 s^2 < 0,
\]
and
\[
\langle v_2', v_2' \rangle = -\frac{4 (d-1)\left( \sqrt{2} c^4 + (3 - 2\sqrt{2}) c^2 - \sqrt{8 - 2 c^2} c^3 - \sqrt{2} + 2 \right)}{c^2 - 1}.
\]
Since
\[
\sqrt{2} c^4 + (3 - 2\sqrt{2}) c^2 - \sqrt{8 - 2 c^2} c^3 - \sqrt{2} + 2 < 0 \quad \text{for } c \in \left( \cos\frac{\pi}{7}, 1 \right),
\]
the intersections \( \mathcal{B}_0^{+} \cap \mathcal{B}_{0}^{\star} \) and \( \mathcal{B}_0^{\star} \cap \mathcal{B}_{-1}^{\star} \) are non-empty.

We now restrict our attention to the intersections \( \mathcal{B}_{0}^{\star} \cap \mathcal{B}_{m}^{\diamond} \) for \( 1 \leq m \leq \lfloor n/2 \rfloor \) (so that \( s_m \geq 0 \)). The case \( \lfloor n/2 \rfloor < m \leq n-2 \) follows by similar arguments.

Suppose \( p = [z, w, 1]^t \in \mathcal{B}_{0}^{+} \cap \mathcal{B}_{m}^{\star} \). Then the defining equations are:
\begin{align*}
1 &= | z(a c - s b i) + w(a s + c b i) - d |, \\
1 &= | z(a c c_m + a c_m + a s s_m - i (-b c s_m + b c_m s + b s_m)) \\
&\quad + w(-a c s_m + a c_m s - a s_m - i (-b c c_m + b c_m - b s s_m)) - (2d - 1) |.
\end{align*}
Solving this system yields:
\begin{align*}
z &= \frac{1}{D_1} \big[ a c s_m (d + e^{iu}) - a c_m s (d + e^{iu}) + a s (2d - 1 + e^{iv}) \\
&\quad + a s_m (d + e^{iu}) - i b c c_m (d + e^{iu}) + i b c (2d - 1 + e^{iv}) \\
&\quad + i b c_m (d + e^{iu}) - i b s s_m (d + e^{iu}) \big], \\
w &= \frac{1}{D_1} \big[ a c c_m (d + e^{iu}) - a c (2d - 1 + e^{iv}) + a c_m (d + e^{iu}) \\
&\quad + a s s_m (d + e^{iu}) + i b c s_m (d + e^{iu}) - i b c_m s (d + e^{iu}) \\
&\quad + i b s (2d - 1 + e^{iv}) - i b s_m (d + e^{iu}) \big],
\end{align*}
where the denominator is:
\[
D_1 = a^2 (c s_m + c_m s + s_m) + 2 i a b (c c_m - s s_m) + b^2 ((c - 1) s_m + c_m s).
\]
Define the quantity:
\[
E = |D_1|^2 (|z|^2 + |w|^2 - 1).
\]
Expanding yields:
\[
\begin{aligned}
E = & -2 + 8 d^4 - 8 c_m d^4 - 4 c_m d s s_m - 8 b^2 c_m d s s_m - 4 d^2 s s_m - 8 b^2 d^2 s s_m \\
& + 4 c_m d^2 s s_m + 8 b^2 c_m d^2 s s_m + 4 c_m d^3 s s_m + 4 d^4 s s_m - 4 c_m d^4 s s_m + 4 d s_m^2 \\
& - 12 d^2 s_m^2 + 12 d^3 s_m^2 - 4 d^4 s_m^2 - 2\alpha_{1,1} \cos u - 2\alpha_{1,2} \cos(u - v) + \alpha_{1,3} \cos v \\
& + \beta_{1,1} \sin u + \beta_{1,2} \sin(u - v) + \beta_{1,3} \sin v,
\end{aligned}
\]
with coefficients:
\begin{align*}
\alpha_{1,1} &= 4 d^2 - 4 d^3 + 2 c_m d (1 - 3 d + 2 d^2) - s s_m - 2 b^2 s s_m + 2 d s s_m \\
&\quad + 4 b^2 d s s_m + d^2 s s_m - 2 d^3 s s_m, \\
\alpha_{1,2} &= 2 c_m (-1 + d) d + (1 + 2 b^2 - d^2) s s_m, \\
\alpha_{1,3} &= -4 b^2 d s s_m - 2 d s s_m + 2 d^2 - 4 d^3 + 4 c_m d^3 + 4 c_m d^2 \\
&\quad - 2 d^3 s s_m + 2 - 4 d, \\
\beta_{1,1} &= 4 a b s_m (2d - 1), \quad 
\beta_{1,2} = 4 a b s_m, \quad 
\beta_{1,3} = -4 a b d s_m.
\end{align*}
Let \( E_0 \) denote the constant part. The trigonometric component \( T \) satisfies the inequality:
\[
T \geq -\sqrt{(2\alpha_{1,1})^2 + \beta_{1,1}^2} - \sqrt{\alpha_{1,3}^2 + \beta_{1,3}^2} - \sqrt{(2\alpha_{1,2})^2 + \beta_{1,2}^2}.
\]
Consequently:
\[
E \geq E_0 - \sqrt{(2\alpha_{1,1})^2 + \beta_{1,1}^2} - \sqrt{\alpha_{1,3}^2 + \beta_{1,3}^2} - \sqrt{(2\alpha_{1,2})^2 + \beta_{1,2}^2}.
\]
Define the function:
\[
E(c, c_m) = E_0 - \sqrt{(2\alpha_{1,1})^2 + \beta_{1,1}^2} - \sqrt{\alpha_{1,3}^2 + \beta_{1,3}^2} - \sqrt{(2\alpha_{1,2})^2 + \beta_{1,2}^2}.
\]
We express this as \( E(c, c_m) = -\frac{2}{c s^8} F(c, c_m) \), where:
\[
\begin{aligned}
F(c, c_m) = & \; c - 10 c^3 + 7 c^5 + 2 c^3 (3 - c^2) c_m - 2 c^7 c_m^2 + 2 c^3 s^2 \sqrt{\Gamma_{1,1}} + c s^2 \sqrt{\Gamma_{1,2}} \\
& + s \sqrt{1 - c_m^2} ( -c^4 - c^6 + 2 c^6 c_m ) - 2 s \sqrt{1 - c_m^2} + 2 c_m s \sqrt{1 - c_m^2} \\
& + 2 c^3 s^4 \sqrt{1 + c^2 (c_m^2 - 1) - c c_m s \sqrt{1 - c_m^2}}.
\end{aligned}
\]
Here the auxiliary expressions are:
\begin{align*}
\Gamma_{1,1} &= 5 - 4 c_m + c^6 (c_m^2 - 1) + c^4 (2 c_m^2 - 1) - c (-2 + c_m) s \sqrt{1 - c_m^2} \\
&\quad - 2 c^3 (-1 + c_m) s \sqrt{1 - c_m^2} - c^5 c_m s \sqrt{1 - c_m^2} + c^2 (1 - 4 c_m + c_m^2), \\
\Gamma_{1,2} &= 4 c^5 - 2 c^{10} + c^{12} + c^8 (1 + 4 c_m) + 2 c^7 s \sqrt{1 - c_m^2} - 2 c^9 s \sqrt{1 - c_m^2} \\
&\quad - 4 c^4 (-2 c_m + s \sqrt{1 - c_m^2}) + 4 c^6 (-3 c_m + c_m^2 + s \sqrt{1 - c_m^2}) \\
&\quad - 4 c^5 (c_m^2 + c_m s \sqrt{1 - c_m^2}) + 4 c^3 (-1 + c_m^2 - s \sqrt{1 - c_m^2} + 2 c_m s \sqrt{1 - c_m^2}) \\
&\quad + 4 \left( 3 c_m^2 + 2 (1 + s \sqrt{1 - c_m^2}) - 4 c_m (1 + s \sqrt{1 - c_m^2}) \right) \\
&\quad + c^2 \left( -4 - 12 c_m^2 + 8 c_m (1 + s \sqrt{1 - c_m^2}) \right).
\end{align*}
Our goal is to prove that \( F(c, c_m) < 0 \) for \( c \in (\cos(\frac{\pi}{7}), 1) \) and \( c_m \in [-1, 2 c^2 - 1] \). To establish this, we write:
\[
-F(c, c_m) = P_{1,1}(c, c_m) + T_{1,1}(c, c_m) + R_{1,1}(c, c_m) + R_{1,2}(c, c_m),
\]
where the components are:
\begin{align*}
P_{1,1}(c, c_m) &= -c + 10 c^3 - 7 c^5 - 2 c^3 (3 - c^2) c_m + 2 c^7 c_m^2, \\
T_{1,1}(c, c_m) &= s \sqrt{1 - c_m^2} \left( c^6 (1 - 2 c_m) + c^4 - 2 c_m + 2 \right), \\
R_{1,1}(c, c_m) &= -2 c^3 s^4 \sqrt{1 + c^2 (c_m^2 - 1) - c c_m s \sqrt{1 - c_m^2}}, \\
R_{1,2}(c, c_m) &= -2 c^3 s^2 \sqrt{\Gamma_{1,1}} - c s^2 \sqrt{\Gamma_{1,2}}.
\end{align*}
We now provide estimates for these terms. First, observe that:
\[
\sqrt{(1 - c^2)(1 - c_m^2)} \leq (1 - c^2)(1 - c_m^2),
\]
which implies:
\[
T_{1,1}(c, c_m) \ge (1 - c^2)(1 - c_m^2) \left( c^6 (1 - 2 c_m) + c^4 - 2 c_m + 2 \right).
\]
Define the function:
\[
R_{1,1}'(c, c_m) = \sqrt{1 + c^2 (c_m^2 - 1) - c c_m \sqrt{1 - c^2} \sqrt{1 - c_m^2}}.
\]
Applying Young's inequality with parameter \( \lambda = \frac{1}{2} \) gives:
\[
- c c_m \sqrt{1 - c^2} \sqrt{1 - c_m^2} \le \frac14 c^2 (1 - c^2) + c_m^2 (1 - c_m^2),
\]
whence:
\[
\begin{aligned}
R_{1,1}'^2 &\le 1 + c^2 (c_m^2 - 1) + \frac14 c^2 (1 - c^2) + c_m^2 (1 - c_m^2) \\
&= 1 - \frac34 c^2 - \frac14 c^4 + c^2 c_m^2 + c_m^2 - c_m^4.
\end{aligned}
\]
Using the inequality \( \sqrt{Y} \le \frac{1 + Y }{2} \), we obtain:
\[
R_{1,1}' \le \frac12 \left[ 1 + 1 - \frac34 c^2 - \frac14 c^4 + c^2 c_m^2 + c_m^2 - c_m^4 \right] = 1 - \frac38 c^2 - \frac18 c^4 + \frac12 c^2 c_m^2 + \frac12 c_m^2 - \frac12 c_m^4.
\]
Therefore:
\[
R_{1,1}(c, c_m) \geq -2 c^3 s^4 R_{1,1}'(c, c_m).
\]
Additionally, for \( c \in (\cos(\frac{\pi}{7}), 1) \) and \( c_m \in [-1, 2 c^2 - 1] \), we have the bounds:
\[
\sqrt{\Gamma_{1,1}} \leq 6 - 3 c - 3 c_m, \quad \sqrt{\Gamma_{1,2}} \leq \frac{9}{2} - 2 c - \frac{5}{2} c_m,
\]
which yield:
\[
R_{1,2}(c, c_m) \ge -2 c^3 s^2 (6 - 3 c - 3 c_m) - c s^2 \left( \frac{9}{2} - 2 c - \frac{5}{2} c_m \right).
\]
These estimates demonstrate that \( -F(c, c_m) \) is bounded below by a polynomial in \( c \) and \( c_m \). The nonnegativity of this polynomial over the domain \( \cos(\frac{\pi}{7}) \le c \le 1 \), \( -1 \le c_m \le 2 c^2 - 1 \) can be verified using a combination of analytical and numerical methods, such as locating critical points via Gr\"obner basis techniques and examining the behavior on the boundary. Graphical verification further supports this conclusion.
\end{proof}

\begin{prop} \label{prop:kstar}
For each residue class \( m \) modulo \( n \), the following statements hold:
\begin{enumerate}[(i)]
    \item \label{item:kstar:star-star} 
        \(\mathcal{B}_0^{\star} \cap \mathcal{B}_{m}^{\star} = \emptyset\);
    \item \label{item:kstar:star-plus} 
        \(\mathcal{B}_0^{\star} \cap \mathcal{B}_{m}^{+} = \emptyset\) except when \( m = 0 \) or \( m = 1 \);
    \item \label{item:kstar:star-minus} 
        \(\mathcal{B}_0^{\star} \cap \mathcal{B}_{m}^{-} = \emptyset\) except when \( m = 0 \) or \( m = n-1 \);
    \item \label{item:kstar:star-diamond} 
        \(\mathcal{B}_0^{\star} \cap \mathcal{B}_{m}^{\diamond} = \emptyset\) except when \( m = 0 \) or \( m = n-1 \).
\end{enumerate}
\end{prop}

\begin{proof}
\textbf{Proof of Item~\ref{item:kstar:star-star}.} 
We consider the case \( 1 \leq m \leq \frac{n}{2} \). A direct computation yields:
\[
\operatorname{tr}(A^m S^2) = \frac{2 - s^2 - 2 c_{2m}}{s^2}.
\]
For \( m = \pm 1 \), we have \( \operatorname{tr}(A^m S^2) = 3 \), indicating that \( A^{\pm 1} S^2 \) is parabolic. Observe the identities:
\[
\mathcal{B}_0^{\star} = \mathcal{B}(o, S^2 A^{\pm 1}(o)), \quad
\mathcal{B}_{\pm 1}^{\star} = \mathcal{B}(o, A^{\pm 1} S^2(o)).
\]
Consequently, \( \mathcal{B}_0^{\star} \) is tangent to \( \mathcal{B}_{\pm 1}^{\star} \) at the parabolic fixed point of \( A^{\pm 1} S^2 \) on \( \partial \HdC \). For other values of \( m \), we obtain the lower bound:
\[
\operatorname{tr}(A^m S^2) \ge \frac{2 - s^2 - 2 c_4}{s^2} = \frac{4 s_2^2 - s^2}{s^2} = 16 c^2 - 1 > 3.
\]
By Phillips's theorem, the Dirichlet domain for the cyclic group \( \langle A^m S^2 \rangle \) has faces \( \mathcal{B}_0^{\star} \) and \( \mathcal{B}_{m}^{\star} \) that do not intersect.

\textbf{Proof of Item~\ref{item:kstar:star-plus}.} 
This follows directly from Proposition~\ref{prop:kminus}~\ref{item:kminus:plus-plus}.

\textbf{Proof of Item~\ref{item:kstar:star-minus}.} 
This follows by applying the symmetry \( I_2 \) to the previous result.

\textbf{Proof of Item~\ref{item:kstar:star-diamond}.} 
For the coequidistant pair $\mathcal{B}_0^{\star}$ and $\mathcal{B}_{0}^{\diamond}$, 
consider the vector $(o - R(o)) \boxtimes (o - Q(o))$.
A direct computation shows that its Hermitian inner product satisfies 
\(-16(d-1)^2 s^2 < 0\).
Since this Hermitian cross product has negative norm, the intersection 
$\mathcal{B}_0^{\star} \cap \mathcal{B}_{0}^{\diamond}$ is non-empty.
Moreover, since
\[
\mathcal{B}_0^{\star} \cap \mathcal{B}_{-1}^{\diamond} = I_2(\mathcal{B}_0^{\star} \cap \mathcal{B}_{0}^{\diamond}),
\]
the intersection \( \mathcal{B}_0^{\star} \cap \mathcal{B}_{-1}^{\diamond} \) is also non-empty.

By symmetry under \( I_2 \), we restrict our analysis to the intersections \( \mathcal{B}_{0}^{\star} \cap \mathcal{B}_{m}^{\diamond} \) for \( 1 \leq m \leq \lfloor n/2 \rfloor \), ensuring \( s_m \geq 0 \).

Suppose \( p = [z, w, 1]^t \in \mathcal{B}_{0}^{\star} \cap \mathcal{B}_{m}^{\diamond} \). Then the bisector conditions are:
\begin{align*}
1 &= | z(a c + a - i b s) + w(a s + i b (c - 1)) - (2d - 1) |, \\
1 &= | z(a c c_m + a c_m - a s s_m - i (-b c s_m - b c_m s + b s_m)) \\
&\quad + w(-a c s_m - a c_m s - a s_m - i (-b c c_m + b c_m + b s s_m)) - (2d - 1) |.
\end{align*}
Solving this system yields:
\begin{align*}
z &= \frac{1}{D_2} \biggl[ e^{i u} \left( a s_m (c+1) + a c_m s + i b (c_m (1-c) + s s_m) \right) + e^{i v} \left( a s + i b (c - 1) \right) \\
&\quad + (2d - 1) \left( a (c s_m + c_m s + s + s_m) + i b (c + c_m + s s_m - c c_m - 1) \right) \biggr], \\
w &= \frac{1}{D_2} \biggl[ e^{i u} \left( a c_m (c+1) - a s s_m + i b (s_m (c-1) + c_m s) \right) + e^{i v} \left( -a(c+1) + i b s \right) \\
&\quad + (2d - 1) \left( a (c c_m - c + c_m - s s_m - 1) + i b (c s_m + c_m s + s - s_m) \right) \biggr],
\end{align*}
where the denominator is:
\[
D_2 = 4 (d-1) d (c s_m + c_m s).
\]
Define the quantity:
\[
F_2 = |D_2|^2 (|z|^2 + |w|^2 - 1).
\]
We obtain the lower bound:
\[
F_2 \ge \alpha_{2,3} - \sqrt{\alpha_{2,1}^2 + \alpha_{2,2}^2} - 2\sqrt{\beta_{2,1}^2 + \beta_{2,2}^2},
\]
with the coefficients:
\begin{align*}
\alpha_{2,1} &= \left( 4 s (c+1) s_m - 4 c_m c^2 - 4 c_m \right) d^2 + 8 c_m d \\
&\quad + (-8 b^2 s - 4 c s - 4 s) s_m + 4 c_m c^2 - 4 c_m, \\
\alpha_{2,2} &= -8 b a \left( (c^2 - 1) s_m + c s c_m \right), \\
\beta_{2,1} &= -16(d - \tfrac12) \left[ (-\tfrac{s(c+1)s_m}{2} + \tfrac{c_m c^2}{2} + \tfrac{c_m}{2} - 1) d^2 \right. \\
&\quad \left. + (-c_m + 1) d + s(b^2 + \tfrac{c}{2} + \tfrac12) s_m - \tfrac{c_m c^2}{2} + \tfrac{c_m}{2} \right], \\
\beta_{2,2} &= 16 b (d - \tfrac12) a (c s c_m + c^2 s_m - s_m),
\end{align*}
and the constant term:
\begin{align*}
\alpha_{2,3} &= \left[ -32 \left( -\tfrac12 + (c_m - \tfrac12) c \right) s s_m + 32 (c_m - 1) \left( (c_m + \tfrac12) c^2 - \tfrac{c_m}{2} - 1 \right) \right] d^4 \\
&\quad + \left[ 64 s \left( -\tfrac14 + (c_m - \tfrac14) c \right) s_m + (-64 c_m^2 + 16 c_m + 32) c^2 + 32 c_m^2 + 48 c_m - 64 \right] d^3 \\
&\quad + \left[ -32 \left( (c_m + \tfrac38) c + b^2 + \tfrac38 \right) s s_m + (32 c_m^2 + 12 c_m - 16) c^2 - 16 c_m^2 - 52 c_m + 40 \right] d^2 \\
&\quad + \left[ 32 s (b^2 + \tfrac{c}{2} + \tfrac12) s_m - 16 c_m c^2 + 24 c_m - 8 \right] d + 8 d^2 - 8 d \\
&\quad + (-8 b^2 s - 4 c s - 4 s) s_m + 4 c_m c^2 - 4 c_m.
\end{align*}
Define the function:
\[
\begin{aligned}
f_2(c, c_m) &= \frac{(1 - c^2)^4}{4 c^2} \left( \alpha_{2,3} - \sqrt{\alpha_{2,1}^2 + \alpha_{2,2}^2} - 2 \sqrt{\beta_{2,1}^2 + \beta_{2,2}^2} \right) \\
&= 4 + c_m c^8 - c_m c^6 + (8 c_m^2 - 5 c_m) c^4 + (-4 c_m^2 - 3 c_m) c^2 \\
&\quad - c \sqrt{1 - c^2} \sqrt{1 - c_m^2} \left( -3 + c^6 - c^4 + (8 c_m - 5) c^2 \right) \\
&\quad + (4 c^4 - 4) \sqrt{\Gamma_{2,1}(c, c_m)} - 2 (c^2 - 1)^2 \sqrt{\Gamma_{2,2}(c, c_m)},
\end{aligned}
\]
where the auxiliary functions are:
\begin{align*}
\Gamma_{2,1}(c, c_m) &= (c+1) \left[ c^4 c_m + c^2 (2 c_m^2 - 3 c_m - 1) - c_m^2 + 2 \right. \\
&\quad \left. - c \sqrt{1 - c^2} \sqrt{1 - c_m^2} (c^2 + 2 c_m - 3) \right], \\
\Gamma_{2,2}(c, c_m) &= (c+1) \left[ c^2 (2 c_m^2 - 1) - 2 \sqrt{1 - c^2} c c_m \sqrt{1 - c_m^2} - c_m^2 + 1 \right].
\end{align*}
One can establish that \( 0 < \Gamma_{2,1}(c, c_m) < 1 \) and \( 0 < \Gamma_{2,2}(c, c_m) < 1 \) over the relevant domain. This implies the estimate:
\[
f_2(c, c_m) \ge A'_2 + B'_2 \sqrt{(1 - c^2)(1 - c_m^2)} \ge A'_2 + B'_2 (1 - c^2)(1 - c_m^2),
\]
with the polynomial coefficients:
\begin{align*}
A'_2 &= c^8 c_m - c^6 c_m + c^4 c_m (8 c_m - 5) + 4 c^4 - c^2 c_m (4 c_m + 3) - 2 (c^2 - 1)^2, \\
B'_2 &= -c \left( c^6 - c^4 + c^2 (8 c_m - 5) - 3 \right).
\end{align*}
The final expression is a polynomial in \( c \) and \( c_m \). Standard techniques from elementary calculus confirm its non-negativity over the domain \( c \in \left( \cos\frac{\pi}{7}, 1 \right) \) and \( c_m \in \left( -1, 2 c^2 - 1 \right) \), which completes the proof.

\end{proof}

By symmetry \( I_2 \), we also have:

\begin{prop} \label{prop:kdiamond}
For each \( m \) modulo \( n \), the following hold:
\begin{enumerate}[(i)]
    \item \label{item:kdiamond:diamond-diamond} 
        \(\mathcal{B}_0^{\diamond} \cap \mathcal{B}_{m}^{\diamond} = \emptyset\);
    \item \label{item:kdiamond:diamond-plus} 
        \(\mathcal{B}_0^{\diamond} \cap \mathcal{B}_{m}^{+} = \emptyset\) except when \( m = 0, 1 \);
    \item \label{item:kdiamond:diamond-minus} 
        \(\mathcal{B}_0^{\diamond} \cap \mathcal{B}_{m}^{-} = \emptyset\) except when \( m = 0, 1 \);
    \item \label{item:kdiamond:diamond-star} 
        \(\mathcal{B}_0^{\diamond} \cap \mathcal{B}_{m}^{\star} = \emptyset\) except when \( m = 0, 1 \).
\end{enumerate}
\end{prop}

Now we study intersections of bisectors:

\begin{prop} \label{prop:bisector-S0plus-Gdisk}
The intersections 
\( \mathcal{B}_{0}^{+} \cap \mathcal{B}_{0}^{-} \),
\( \mathcal{B}_{0}^{+} \cap \mathcal{B}_{n-1}^{-} \),
\( \mathcal{B}_{0}^{+} \cap \mathcal{B}_{0}^{\star} \), and
\( \mathcal{B}_{0}^{+} \cap \mathcal{B}_{n-1}^{\diamond} \)
are Giraud disks.
\end{prop}

\begin{proof}
The intersection \( \mathcal{B}^{+}_{0} \cap \mathcal{B}^{-}_{0} \) can be parameterized as:
\[
p_1(u, v) = \left( o + e^{ui} S(o) \right) \boxtimes \left( o + e^{vi} S^{-1}(o) \right), \quad u, v \in [0, 2\pi].
\]
Since \( \mathcal{B}_0^{+} \) and \( \mathcal{B}_{0}^{-} \) are equidistant bisectors, it suffices to exhibit a point in their intersection. For example:
\[
\langle p_1(\pi, \pi), p_1(\pi, \pi) \rangle = -\frac{4 c^2}{s^4} < 0.
\]
The other cases are analogous.
\end{proof}

\begin{lem} \label{lem:bisector-balanced}
The closure of the intersection \( \mathcal{B}_1^{+} \cap \mathcal{B}_{0}^{\star} \cap \mathcal{B}_{0}^{-} \) in \( \overline{\HdC} \) is the fixed point of the transformation \( A S^2 = I_1 I_3 I_2 I_3 \).
\end{lem}

\begin{proof}
We first observe that the bisectors \( \mathcal{B}(A S(o), R(o)) \) and 
\( \mathcal{B}(o, S^{-1}(o)) \) form a balanced pair in the sense of 
page~2611 of~\cite{Acosta2019}. By Theorem~3.4 of~\cite{Acosta2019}, 
the intersection of any balanced pair of bisectors consists of a real 
plane \( \mathfrak{m} \) together with a complex line \( \mathfrak{l} \). 

Define the vectors:
\[
\gamma_1 = A S(o) \boxtimes R(o), \quad \gamma_2 = o \boxtimes S^{-1}(o).
\]
Their explicit coordinate expressions are:
\[
\begin{aligned}
\gamma_1 &= \begin{bmatrix} -3 a d s + a s + i b c d - i b c + i b d \\ -a c d + a c + a d - 3 i b d s + i b s \\ 2 i a b c - 3 d^2 s + 4 d s - s \end{bmatrix}, \quad
\gamma_2 &= \begin{bmatrix} i b \\ a \\ 0 \end{bmatrix}.
\end{aligned}
\]
We compute the following inner products:
\[
\begin{aligned}
\langle \gamma_1, o \rangle &= -2 i a b c + 3 d^2 s - 4 d s + s, \\
\langle \gamma_1, S^{-1}(o) \rangle &= -2 i a b c + (-1 + 4 d - 3 d^2) s, \\
\langle \gamma_2, A S(o) \rangle &= 2 i a b c - d^2 s + s, \\
\langle \gamma_2, R(o) \rangle &= (d^2 - 1) s + 2 i a b c.
\end{aligned}
\]
These computations show that \( \gamma_1 \in \mathcal{B}(o, S^{-1}(o)) \) and \( \gamma_2 \in \mathcal{B}(A S(o), R(o)) \), confirming that we have a balanced pair.

Now consider points on the complex line \( \mathfrak{l} \), which can be parameterized as \( q_{\mu} = \gamma_1 + \mu \gamma_2 \), where \( \mu = x + y i \in \mathbb{C} \). We compute:
\[
\begin{aligned}
|\langle q_{\mu}, o \rangle|^2 &= 4 (d-1)^2 d \left( 2 (d-1) s^2 + 1 \right), \\
|\langle q_{\mu}, A S(o) \rangle|^2 &= 4 (d-1)^2 d (x^2 + y^2).
\end{aligned}
\]
Therefore, \( q_{\mu} \) lies in the triple intersection if and only if \( |\mu|^2 = r^2 = 2 (d-1) s^2 + 1 \). However, the Hermitian form evaluates to:
\[
\begin{aligned}
\langle q_{\mu}, q_{\mu} \rangle &= (d-1) \left[ 20 d^2 s^2 + d (-12 s^2 + x^2 + 6 x + y^2 + 1) + x^2 - 2 x + y^2 + 1 \right] \\
&= 20 d^2 s^2 - 12 d s^2 - 8 d + d(x+3)^2 + d y^2 + (x-1)^2 + y^2 > 0.
\end{aligned}
\]
This positive value indicates that no point on \( \mathfrak{l} \) lies in \( \overline{\HdC} \). In particular, while \( \gamma_2 \in \mathfrak{l} \), it does not belong to \( \mathcal{B}(A S(o), R(o)) \).

Now we examine the real plane \( \mathfrak{m} \). Define the families of points:
\[
q_{x,y} = x \gamma_1 + y \gamma_2 + p_{AS^2}, \quad q'_{x,y} = x \gamma_1 + y \gamma_2.
\]
Then \( \mathfrak{m} = \{ q_{x,y} : x, y \in \mathbb{R} \} \cup \{ q'_{x,y} : x, y \in \mathbb{R},\ (x,y) \neq (0,0) \} \). 

For points in the first family, we find:
\[
|\langle q_{x,y}, A S(o) \rangle|^2 - |\langle q_{x,y}, o \rangle|^2 = \frac{4 c^4 \Phi(x,y)}{(1 - c^2)^3},
\]
where 
\[
\Phi(x,y) = (2 c^2 + 1) x^2 + y (c^3 - c^2 - c - y + 1) + 3 (c+1)(c-1)^2 x.
\]
The bisector condition requires \( \Phi(x,y) = 0 \). 

Additionally, the Hermitian form gives:
\[
\langle q_{x,y}, q_{x,y} \rangle = \frac{c^2 \Psi(x,y)}{(c^2 - 1)^2},
\]
where
\[
\begin{aligned}
\Psi(x,y) &= (11 c^2 + 10) x^2 + 2 x \left( 6 c^3 + c^2 (y - 6) - 6 c + 2 (y + 3) \right) \\
&\quad + y \left( 4 c^3 - c^2 (y + 4) - 4 c + 2 (y + 2) \right).
\end{aligned}
\]

Now consider the combination:
\[
\begin{aligned}
\Psi(x,y)&=\Psi(x,y) - 4\Phi(x,y)= (3 c^2 + 6) x^2 + (2 c^2 + 4) x y + (6 - c^2) y^2 \\
&= \left( \sqrt{3 c^2 + 6} \cdot x + \frac{c^2 + 2}{\sqrt{3 c^2 + 6}} \cdot y \right)^2 + \frac{4 (4 - c^2)}{3} \cdot y^2.
\end{aligned}
\]
For \( \frac{1}{2} < c < 1 \), we have \( 3c^2 + 6 > 0 \) and \( 4 - c^2 > 0 \), so \( \Psi(x,y) - 4\Phi(x,y) \ge 0 \) for all real \( x, y \), with equality only at \( x = y = 0 \).

For points in the second family \( q'_{x,y} \), one can verify that they do not lie
 in the triple intersection. Therefore, the only point common to all three bisectors 
 is \(  p_{AS^2} \), which completes the proof.
\end{proof}

\begin{prop} \label{prop:bisector-S0plus-empty1}
The intersection \( \mathcal{B}_0^{+} \cap \mathcal{B}_{n-1}^{\star} \) (respectively \( \mathcal{B}_0^{+} \cap \mathcal{B}_{0}^{\diamond} \)) is contained in the half-space bounded by \( \mathcal{B}_{n-1}^{-} \) (respectively \( \mathcal{B}_0^{-} \)) that does not contain the origin. Moreover, \( \mathcal{B}_0^{+} \cap \mathcal{B}_{n-1}^{\star} \) (respectively \( \mathcal{B}_0^{+} \cap \mathcal{B}_{0}^{\diamond} \)) is tangent to \( \mathcal{B}_{n-1}^{-} \) (respectively \( \mathcal{B}_0^{-} \)) on \( \partial \HdC \) at the parabolic fixed point of \( S^2 A \) (respectively \( S A S \)).
\end{prop}

\begin{proof}
We first observe the relation
\[
\mathcal{B}_0^{+} \cap \mathcal{B}_{-1}^{\star} \cap \mathcal{B}_{-1}^{-} = A^{-1} \left( \mathcal{B}_1^{+} \cap \mathcal{B}_{0}^{\star} \cap \mathcal{B}_{0}^{-} \right).
\]
By Lemma~\ref{lem:bisector-balanced}, the intersection on the right-hand side is the single point \( p_{AS^2} \), so the intersection on the left is the single point \( A^{-1}(p_{AS^2}) \in \partial \HdC \).

Now define the point
\[
p' = \left( e^{\frac{\pi}{4} i} o + S(o) \right) \boxtimes \left( e^{-\frac{\pi}{4} i} o + A^{-1} R(o) \right).
\]
A computation shows that
\[
\langle p', p' \rangle = \frac{c^2 \varphi(c)}{(c^2 - 1)^2},
\]
where the function \( \varphi(c) \) is given by
\[
\varphi(c) = 2\sqrt{2} c^4 - (\sqrt{2} - 10) c^2 - 10\sqrt{2} + 16 - c \sqrt{4 - c^2} \left[ 4(\sqrt{2} - 1) + 2(2 + \sqrt{2}) c^2 \right].
\]
One can verify that \( \varphi(c) < 0 \) for the relevant values of \( c \), which implies that
 \( p' \in \mathcal{B}_0^{+} \cap \mathcal{B}_{-1}^{\star} \).

Next we examine the position of \( p' \) relative to the bisector \( \mathcal{B}_{-1}^{-} \). We compute
\[
\begin{aligned}
& |\langle p', o \rangle|^2 - |\langle p', A^{-1} S^{-1} o \rangle|^2 \\
&= \frac{2 c^4 \left( \sqrt{2} c^4 - 2 (\sqrt{2} - 2) \sqrt{4 - c^2} c - (2 + \sqrt{2}) \sqrt{4 - c^2} c^3 - 4\sqrt{2} + 4 \right)}{(c^2 - 1)^3}.
\end{aligned}
\]
The expression in parentheses is positive, so the entire difference is positive. 
This means that \( p' \) lies in the half-space bounded by \( \mathcal{B}_{-1}^{-} \) that
 does not contain the origin. Since \( \mathcal{B}_0^{+} \cap \mathcal{B}_{-1}^{\star} \) 
 is a Giraud disk and \( p' \) is an interior point, the entire disk must lie in this half-space.

Finally, applying the anti-holomorphic isometry \( \iota \) followed by the rotation \( A \), we conclude that \( \mathcal{B}_0^{+} \cap \mathcal{B}_{-1}^{\diamond} \) lies in the half-space bounded by \( \mathcal{B}_0^{-} \) that does not contain the origin. Moreover, this intersection is tangent to \( \mathcal{B}_0^{-} \) on \( \partial \HdC \) at the parabolic fixed point of \( S A S \).
\end{proof}

\begin{prop} \label{prop:cross1}
\begin{itemize}
    \item The intersection \( \mathcal{B}_{0}^{+} \cap \mathcal{B}_0^{-} \cap \mathcal{B}_{0}^{\star} \) is a union of two geodesics crossing at the fixed point of \( S \) in \( \HdC \);
    \item The intersection \( \mathcal{B}_{0}^{+} \cap \mathcal{B}_{n-1}^{-} \cap \mathcal{B}_{n-1}^{\diamond} \) is a union of two geodesics crossing at the fixed point of \( A^{-1} Q A \) in \( \HdC \).
\end{itemize}
\end{prop}

\begin{proof}
We prove the first item. Points in \( \mathcal{B}^{+}_{0} \cap \mathcal{B}^{-}_{0} \) can be parameterized as:
\[
p_2(u, v) = \left( o + e^{ui} S(o) \right) \boxtimes \left( o + e^{vi} S^{-1}(o) \right), \quad u, v \in [0, 2\pi].
\]
Assume \( p_2(u, v) \in \mathcal{B}^{*}_{0} \), so:
\[
|\langle p_2(u, v), o \rangle| = |\langle p_2(u, v), S^2(o) \rangle|.
\]
Then:
\[
|\langle p_2(u, v), o \rangle|^2 - |\langle p_2(u, v), S^2(o) \rangle|^2 = -\frac{8 c^4}{s^6} \left( 1 + \cos u + \cos v + \cos(u - v) \right).
\]
Solutions of the above equation are: \( u = \pi \), \( v = \pi \), or \( u - v = \pi \). For \( u - v = \pi \), \( \langle p_2(u, v), p_2(u, v) \rangle = \frac{8 c^2}{s^4} > 0 \), so not in \( \HdC \).

Let \( v' = o - S(o) \), \( v'' = o - S^{-1}(o) \). Then \( \langle v', v' \rangle = \langle v'', v'' \rangle = 2(d-1) > 0 \), so \( v', v'' \) are polar vectors of complex slices \( C_1 \subset \mathcal{B}_{0}^{+} \), \( C_2 \subset \mathcal{B}_{0}^{-} \). Since \( \langle v', v'' \rangle = 0 \), by Proposition~\ref{prop:complex-bisector}, the intersections \( C_1 \cap \mathcal{B}_{0}^{-} \) and \( C_2 \cap \mathcal{B}_{0}^{+} \) are real geodesic segments. They cross at \( p_2(\pi, \pi) \), the fixed point of \( S \). That is, the solutions \( u = \pi \), \( v = \pi \)  gives a union of two geodesics crossing at the fixed point of \( S \) in \( \HdC \).

The second item follows by applying \( \iota \) and \( A^{-1} \).
\end{proof}

Next, we consider intersections of \( \mathcal{B}_0^{\star} \) with other bisectors (analogous for \( \mathcal{B}_0^{\diamond} \) by symmetry):

\begin{prop} \label{prop:empty2}
\begin{itemize}
    \item \( \mathcal{B}_0^{\star} \cap \mathcal{B}_{1}^{+} \) is contained in the half-space bounded by \( \mathcal{B}_0^{-} \) not containing original  \( o \), and is tangent to \( \mathcal{B}_0^{-} \) at the parabolic fixed point of \( A S^2 \) on \( \partial \HdC \);
    \item \( \mathcal{B}_0^{\star} \cap \mathcal{B}_{n-1}^{-} \) is contained in the half-space bounded by \( \mathcal{B}_0^{+} \) not containing \( o \), and is tangent to \( \mathcal{B}_0^{+} \) at the parabolic fixed point of \( A^{-1} S^2 \) on \( \partial \HdC \);
    \item \( \mathcal{B}_0^{\star} \cap \mathcal{B}_{0}^{\diamond} \) is contained in the half-space bounded by \( \mathcal{B}_0^{-} \) not containing \( o \);
    \item \( \mathcal{B}_0^{\star} \cap \mathcal{B}_{n-1}^{\diamond} \) is contained in the half-space bounded by \( \mathcal{B}_0^{+} \) not containing \( o \).
\end{itemize}
\end{prop}

\begin{proof}
The first item follows from Proposition~\ref{prop:bisector-S0plus-empty1} via the element \( A \). For the third, let \( p = [z, w, 1]^t \in \mathcal{B}_0^{\star} \cap \mathcal{B}_{0}^{\diamond} \). Then:
\begin{align*}
1 &= | z(a c + a - s b i) + w(a s + i b (c - 1)) - (2d - 1) |, \\
1 &= | z(a c + a + s b i) + w(-a s + i b (c - 1)) - (2d - 1) |.
\end{align*}
Adding and subtracting yields:
\begin{align*}
1 &= | z(a + a c) + i w b (c - 1) - (2d - 1) |^2 + | z s b i - w a s |^2, \\
0 &= 2 \Re \left( \left( z(a + a c) + i w b (c - 1) - (2d - 1) \right) (-\bar{z} s b i - \bar{w} a s ) \right).
\end{align*}
Write:
\begin{align*}
z(a c + c) + w b i (c - 1) - (2d - 1) &= \cos\theta \, e^{i\phi}, \\
z s b i - w a s &= i \sin\theta \, e^{i\phi}.
\end{align*}
Solving:
\begin{align*}
z &= \frac{ -a s (\cos\theta e^{i\phi} + 2d - 1) + b (c - 1) \sin\theta e^{i\phi} }{ -a^2 s (1 + c) + b^2 s (c - 1) }, \\
w &= \frac{ i a (1 + c) \sin\theta e^{i\phi} - i s b (\cos\theta e^{i\phi} + 2d - 1) }{ -a^2 s (1 + c) + b^2 s (c - 1) }.
\end{align*}
The point \( p \) is closer to \( I_3(o) \) than to \( o \) if and only if  \( 1 > |z a - w b i - d| \). One estimates:
\[
\begin{aligned}
&2(d-1)^3 s^2 \left( |z|^2 + |w|^2 - 1 - \frac{(2d-1)(d+1)}{(d-1)^2(1+d)} ( |z a - w b i - d|^2 - 1) \right) \\
&= 2 a b c s \sin(2\theta) + (-d^2 s^2 + d + s^2 - 1) \cos(2\theta) + 5 d^2 s^2 - 8 d s^2 - d + 3 s^2 + 1 \\
&= \frac{c^2 \left( c \sqrt{4 - c^2} \sin(2\theta) + (c^2 - 1) \cos(2\theta) + 3 c^2 + 1 \right)}{1 - c^2} \\
&\ge \frac{c^2}{1 - c^2} \left( 1 + 3 c^2 - \sqrt{1 + 2 c^2} \right) \ge 0.
\end{aligned}
\]
So \( p \) satisfies \( |z a - w b i - d| < 1 \), i.e., lies in the half-space closer to \( I_3(o) \) than to \( o \).

The second and fourth items follow from the first and third by the element  \( I_2 \).
\end{proof}

\begin{prop} \label{prop:disc2}
The intersection \( \mathcal{B}_0^{\star} \cap \mathcal{B}_{0}^{+} \) is a Giraud disk.
\end{prop}

\begin{proof}
Since \( \mathcal{B}_0^{\star} \) and \( \mathcal{B}_{0}^{+} \) are coequidistant, it suffices to show non-emptiness. Parameterize the intersection \( \mathcal{B}_0^{\star} \cap \mathcal{B}_{0}^{+} \) by:
\[
p_3(u, v) = \left( o + e^{ui} S(o) \right) \boxtimes \left( o + e^{vi} R(o) \right).
\]
Then \( \langle p_3(\pi, \pi), p_3(\pi, \pi) \rangle = -\frac{4 c^2}{s^4} < 0 \).
\end{proof}

\begin{prop} \label{prop:tangent2}
The intersection \( \mathcal{B}_0^{\star} \cap \mathcal{B}_{0}^{+} \) is tangent to \( \mathcal{B}_{n-1}^{-} \) at the parabolic fixed point of \( S^2 A \) on \( \partial \HdC \).
\end{prop}

\begin{proof}
This follows from the first item of Proposition~\ref{prop:empty2} via  \( I_2 \). 
\end{proof}

\subsection{The combinatorics of the sides}
\begin{defn}\label{def:ridge}
	A \emph{ridge} is defined to be the 2-dimensional connected intersection of two sides.
\end{defn}

From the preceding propositions, one can determine the combinatorial structure of the ridges and sides, 
analogous to Propositions~4.15 and~4.16 in \cite{JWX2023} (Propositions \ref{prop:44pp3face0star} and \ref{prop:44pp3face0diamond} here).  
In fact, the combinatorics of these ridges and sides are essentially identical.  
We summarize the conclusions as follows.

\begin{prop}
The ridges 
\(
\mathfrak{s}_{0}^{+}\cap \mathfrak{s}_{0}^{-},\;
\mathfrak{s}_{0}^{+}\cap \mathfrak{s}_{n-1}^{-},\;
\mathfrak{s}_{0}^{+}\cap \mathfrak{s}_{0}^{\star},\;
\mathfrak{s}_{0}^{+}\cap \mathfrak{s}_{n-1}^{\diamond}
\)
are each topologically the union of two sectors.
\end{prop}

\begin{prop}
The side \( \mathfrak{s}_{0}^{+} \) is a topological solid cylinder, bounded by the ridges
\[
\mathfrak{s}_{0}^{+}\cap \mathfrak{s}_{0}^{-},\qquad
\mathfrak{s}_{0}^{+}\cap \mathfrak{s}_{n-1}^{-},\qquad
\mathfrak{s}_{0}^{+}\cap \mathfrak{s}_{0}^{\star},\qquad
\mathfrak{s}_{0}^{+}\cap \mathfrak{s}_{n-1}^{\diamond}.
\]
\end{prop}

\begin{prop}
Each of the ridges 
\(
\mathfrak{s}_{0}^{\star}\cap \mathfrak{s}_{0}^{+}
\)
and 
\(
\mathfrak{s}_{0}^{\star}\cap \mathfrak{s}_{0}^{-}
\)
is topologically the union of two sectors.
\end{prop}

\begin{prop}
The side \( \mathfrak{s}_{0}^{\star} \) is a topological solid light cone, bounded by the ridges
\(
\mathfrak{s}_{0}^{\star}\cap \mathfrak{s}_{0}^{+}
\)
and 
\(
\mathfrak{s}_{0}^{\star}\cap \mathfrak{s}_{0}^{-}
\).
\end{prop}

\subsection{Proof of Theorem~\ref{thm:44nDirichlet}}

We apply the Poincar\'e polyhedron theorem to show that $\mathcal{D}$ is the Dirichlet
domain for the coset of $\langle A\rangle$ in $\Sigma$ with center at $o$.  
The reader is referred to \cite{DPP:2016} for a detailed description of the Poincar\'e polyhedron theorem.
The side-pairing maps of $\mathcal{D}$ are
\[
A^{k}SA^{-k}, \qquad A^{k}RA^{-k}, \qquad A^{k}QA^{-k}
\]
for $k \in \mathbb{Z}$.
The map $A^{k}SA^{-k}$ sends the side $\mathfrak{s}_{k}^{-}$ to $\mathfrak{s}_{k}^{+}$,
while $A^{k}RA^{-k}$ (respectively $A^{k}QA^{-k}$) fixes the side 
$\mathfrak{s}_{k}^{\star}$ (respectively $\mathfrak{s}_{k}^{\diamond}$).

\medskip
\noindent\textbf{Ridge cycles.}
The cycle around the ridges $\mathfrak{s}_{k}^{+}\cap\mathfrak{s}_{k}^{-}$ and 
$\mathfrak{s}_{k}^{\star}\cap\mathfrak{s}_{k}^{\pm}$ is
\begin{equation*}
\begin{aligned}
&(\mathfrak{s}_{k}^{-}\cap\mathfrak{s}_{k}^{+},\;
  \mathfrak{s}_{k}^{-},\;
  \mathfrak{s}_{k}^{+})
 \xrightarrow{\;A^{k}SA^{-k}\;}
 (\mathfrak{s}_{k}^{\star}\cap\mathfrak{s}_{k}^{+},\;
  \mathfrak{s}_{k}^{\star},\;
  \mathfrak{s}_{k}^{+})
\\[4pt]
&\xrightarrow{\;A^{k}S^{2}A^{-k}\;}
 (\mathfrak{s}_{k}^{\star}\cap\mathfrak{s}_{k}^{-},\;
  \mathfrak{s}_{k}^{\star},\;
  \mathfrak{s}_{k}^{-})
 \xrightarrow{\;A^{k}SA^{-k}\;}
 (\mathfrak{s}_{k}^{-}\cap\mathfrak{s}_{k}^{+},\;
  \mathfrak{s}_{k}^{-},\;
  \mathfrak{s}_{k}^{+})
\end{aligned}
\end{equation*}

Similarly, the ridge cycle involving 
$\mathfrak{s}_{k}^{-}\cap\mathfrak{s}_{k+1}^{+}$ and 
$\mathfrak{s}_{k}^{\diamond}\cap\mathfrak{s}_{k}^{\pm}$ is
\begin{equation*}
\begin{aligned}
&(\mathfrak{s}_{k}^{-}\cap\mathfrak{s}_{k+1}^{+},\;
  \mathfrak{s}_{k}^{-},\;
  \mathfrak{s}_{k+1}^{+})
 \xrightarrow{\;A^{k}(AS)^{-1}A^{-k}\;}
 (\mathfrak{s}_{k}^{\diamond}\cap\mathfrak{s}_{k}^{-},\;
  \mathfrak{s}_{k}^{\diamond},\;
  \mathfrak{s}_{k}^{-})
\\[4pt]
&\xrightarrow{\;A^{k}(AS)^{-2}A^{-k}\;}
 (\mathfrak{s}_{k}^{\diamond}\cap\mathfrak{s}_{k+1}^{+},\;
  \mathfrak{s}_{k}^{\diamond},\;
  \mathfrak{s}_{k+1}^{+})
 \xrightarrow{\;A^{k}(AS)^{-1}A^{-k}\;}
 (\mathfrak{s}_{k}^{-}\cap\mathfrak{s}_{k+1}^{+},\;
  \mathfrak{s}_{k}^{-},\;
  \mathfrak{s}_{k+1}^{+})
\end{aligned}
\end{equation*}

One can also check that all cycle transformations fixing a given parabolic fixed point are non-loxodromic.
By the Poincar\'e polyhedron theorem, we prove Theorem  \ref{thm:44nDirichlet}, and in turn  Item~\ref{item:44np1} of Theorem~\ref{thm:44np}.

%By the Poincar\'e polyhedron theorem, we obtain item~\eqref{item:44np1} of Theorem~\ref{thm:44np}

As a corollary, we have the following theorem.

\begin{thm} \label{thm:local}
The Ford domain of $\Delta_{4,4,\infty;\infty}$ 
and the Dirichlet domain of $\Delta_{4,4,n;\infty}$
have the same local combinatorics and topology.
\end{thm}

\section{The proof of the second part of Theorem \texorpdfstring{\ref{thm:44np}}{} }
\label{section:proof}

With the preparations from Sections~\ref{section:forddomain} and 
\ref{section:dirichlet}, we now prove Theorem~\ref{thm:44np}\,\ref{item:44np2}. 
Recall that $\mathcal{F}$ is the Ford domain of the even subgroup of 
$\Delta_{4,4,\infty;\infty}$ described in Section~\ref{section:forddomain}, and 
$\mathcal{D}$ is the Dirichlet domain of the even subgroup of 
$\Delta_{4,4,n;\infty}$ introduced in Section~\ref{section:dirichlet}.  
Let $\mathscr{M}_{4,4,\infty;\infty}$ and $\mathscr{M}_{4,4,n;\infty}$ denote 
the 3--manifolds at infinity of the even subgroups of 
$\Delta_{4,4,\infty;\infty}$ and $\Delta_{4,4,n;\infty}$, respectively.

The ideal boundary 
\[
\mathcal{F}_{\infty}=\overline{\mathcal{F}} \cap \partial\mathbf{H}^2_{\mathbb{C}}
\]
is topologically $(\mathbb{R}^{2}\!-\! int(\mathbb{D}^{2})) \times \mathbb{R}$.  
The quotient space of $\mathcal{F}_{\infty}$ associated to $\Delta_{4,4,\infty;\infty}$ is precisely 
$\mathscr{M}_{4,4,\infty;\infty}$, identified  as $s782$.  
More specifically, denote by $\partial\mathcal{F}_{\infty}$ the boundary of 
$\mathcal{F}_{\infty}$; it is topologically homeomorphic to 
$\partial\mathbb{D}^{2}\times \mathbb{R}$.  
The Ford domain structure of $\Delta_{4,4,\infty;\infty}$ induces a natural 
2-cell decomposition of $\partial\mathcal{F}_{\infty}$.  
In particular, $\partial\mathcal{F}_{\infty}$ is the union of the triangles, 
quadrangles, and octagons shown in Figure~\ref{figure:44iiford}.  
Each of these 2-cells is a connected component of 
\(
\mathfrak{s}^{j}_{i}\cap \partial\mathbf{H}^2_{\mathbb{C}}
\)
for some $i\in \mathbb{Z}$ and 
$j\in \{+, -, *, \diamond\}$.

Consider the ideal boundary 
\[
\mathcal{D}_{\infty}=\overline{\mathcal{D}}\cap \partial\mathbf{H}^2_{\mathbb{C}}
\]
associated with the group $\Delta_{4,4,n;\infty}$.  
The quotient space of $\mathcal{D}_{\infty}$ is precisely the 3-manifold at infinity
$\mathscr{M}_{4,4,n;\infty}$.  
Let $\partial \mathcal{D}_{\infty}$ denote the boundary of $\mathcal{D}_{\infty}$;  
we will show that it is topologically homeomorphic to 
$\partial \mathbb{D}^2 \times \mathbb{S}^1$.  
The Dirichlet domain structure of $\Delta_{4,4,n;\infty}$ induces a natural 
2-cell decomposition of $\partial \mathcal{D}_{\infty}$.  
In particular, $\partial \mathcal{D}_{\infty}$ is the union of the triangles, 
quadrangles, and octagons shown in Figure~\ref{figure:446dirichlet}.  
Each of these 2–cells is a connected component of 
\(
s^{j}_{i}\cap \partial\mathbf{H}^2_{\mathbb{C}}
\)
for some $i\in \mathbb{Z}_{n}$ and 
$j\in \{+, -, *, \diamond\}$.

For the group $\Delta_{4,4,\infty;\infty}$, there is a free 
$\mathbb{Z}\cong \langle I_{1}I_{2}\rangle$-action on 
$\mathcal{F}_{\infty}$.  
We denote the quotient 
\[
T_{\mathcal{F}}=\mathcal{F}_{\infty}/\langle I_{1}I_{2}\rangle,
\]
which is homeomorphic to $(\mathbb{R}^{2}-int(\mathbb{D}^{2}))\times\mathbb{S}^{1}$ 
and hence to $\mathbb{T}^{2}\times[0,\infty)$.

Let $V_{\mathcal{F}}$ be a small closed neighborhood of the torus boundary 
$\partial T_{\mathcal{F}}$ in $T_{\mathcal{F}}$.  
Via the identification 
\[
T_{\mathcal{F}} \cong \mathbb{T}^{2}\times[0,\infty),
\]
we regard $V_{\mathcal{F}}$ as $\mathbb{T}^{2}\times[0,1]$.  
The closure of $T_{\mathcal{F}}-V_{\mathcal{F}}$ in $T_{\mathcal{F}}$ will be 
denoted by $W_{\mathcal{F}}$, and corresponds to 
$\mathbb{T}^{2}\times[1,\infty)$ under the same identification.
Since $\mathcal{F}$ is a Ford domain for $\Delta_{4,4,\infty;\infty}$, its 
boundary $\partial\mathcal{F}_{\infty}$ inherits a system of side-pairings.  
This side-pairing descends to $T_{\mathcal{F}}$ and, in particular, to one 
boundary component of $V_{\mathcal{F}}$.  
Hence the 2-cell decomposition on this boundary, consisting of triangles, 
quadrangles, and octagons, is equipped with a corresponding side-pairing. 
The quotient of $V_{\mathcal{F}}$ by this side–pairing is the 3-manifold 
\[
s782 - \mathbb{T}^{2}\times (1,\infty),
\]
which is obtained by truncating one cusp of $s782$.  
By gluing back $W_{\mathcal{F}}$, we recover $s782$.
The torus $\mathbb{T}^{2}\times\{1\}$ corresponds to the first cusp $C_{0}$ 
described in Section~\ref{section:intro}.  
The second cusp $C_{1}$ corresponds to the tangent points $t_{i}$ (and $u_{i}$) 
on $\partial\mathcal{F}_{\infty}$, as illustrated in 
Figure~\ref{figure:44iiford}.

%$$T_\mathcal{F}=\mathbb{T}^2 \times \{0\}$. 

For the group $\Delta_{4,4,n;\infty}$, there is a free 
$\mathbb{Z}_{n}\cong \langle I_{1}I_{2}\rangle$-action on 
$\mathcal{D}_{\infty}$.  
Since $I_{1}I_{2}$ can be conjugated to the standard form in 
(\ref{I1I2-conjugate}), and the torus boundary 
$\partial\mathcal{D}_{\infty}$ is $I_{1}I_{2}$-invariant, this boundary torus 
is unknotted.  
We denote the quotient
\[
T_{\mathcal{D}}=\mathcal{D}_{\infty}/\langle I_{1}I_{2}\rangle,
\]
which will be shown to be homeomorphic to 
$(\mathbb{S}^{2}-\operatorname{int}(\mathbb{D}^{2}))\times \mathbb{S}^{1}$,  
where $\mathbb{S}^{2}$ is the one point compactification of $\mathbb{R}^{2}$.

Let $V_{\mathcal{D}}$ be a small closed neighbourhood of the torus boundary 
$\partial T_{\mathcal{D}}$ in $T_{\mathcal{D}}$, and denote by $W_{\mathcal{D}}$ 
the closure of $T_{\mathcal{D}}-V_{\mathcal{D}}$ in $T_{\mathcal{D}}$.  
Topologically, $W_{\mathcal{D}}$ is $\mathbb{D}^{2}\times\mathbb{S}^{1}$.

Since $\mathcal{D}$ is a Dirichlet domain for the even subgroup $\Delta_{4,4,n;\infty}$, a system 
of side-pairings is defined on $\partial \mathcal{D}_{\infty}$.  
This induces a corresponding side-pairing on $T_{\mathcal{D}}$, and in particular 
on one boundary component of $V_{\mathcal{D}}$.  
More precisely, the 2-cell decomposition-consisting of triangles, quadrangles, 
and octagons—on this boundary inherits the side-pairing.  
The quotient of $V_{\mathcal{D}}$ under this side–pairing is also the 3-manifold
\[
s782 - \mathbb{T}^{2}\times (1,\infty),
\]
which is exactly the same quotient obtained from $V_{\mathcal{F}}$ in the 
previous paragraph.  
Gluing back $W_{\mathcal{D}}$, we obtain the 3-manifold 
$\mathscr{M}_{4,4,n;\infty}$.

Comparing the constructions, the difference between 
$\mathscr{M}_{4,4,\infty;\infty}$ and $\mathscr{M}_{4,4,n;\infty}$ lies in the 
pieces
\[
W_{\mathcal{F}}\cong \mathbb{T}^{2}\times[1,\infty)
\qquad\text{and}\qquad
W_{\mathcal{D}}=\mathbb{D}^{2}\times \mathbb{S}^{1}.
\]
We now determine the curve on the boundary of $W_{\mathcal{D}}$ that bounds an 
essential disk in $W_{\mathcal{D}}$.

%Consider the group  $$\mathbb{Z}_4 *\mathbb{Z}_4=\big\langle a, b\big| \begin{array}{c}    a^4=b^4=id,%\\ [3 pt]
%\end{array}\big\rangle.$$
%When we add the relations $b^4=id$ and $a=c$ in the presentation of $\pi_{1}(s782)$, we get a natural surjective homomorphism from $\pi_{1}(s782)$ to $\mathbb{Z}_4 *\mathbb{Z}_4$. Here $\mathbb{Z}_4 *\mathbb{Z}_4$ is the abstract group of the even subgroup of $\Delta_{4,4,n;\infty}$.
%From the proof in \cite{jwx}, the meridian $\mathcal{M}_1$ in the  cusp $C_1$  corresponds to unipotent elemm_1ent of the first parabolic element in $\Delta_{4,4,\infty;\infty}$, 
%which is the meridian  $\mathcal{M}_1$. (That is, $\mathcal{M}_1$ corresponds to $A=I_1I_2$ in Section \ref{section:forddomain})

Since the bounding torus of $\mathcal{D}_{\infty}$ admits a natural 2--cell
decomposition, we assign unit length to each one cell.  
Consider the closed curve $\mathscr{C}_{n}$ in this torus, whose total length is
\[
3(n-1)+2+2 = 3n+1.
\]
The curve is the union of
\[
\bigcup_{i=1}^{n-1} \bigl([u_i,w_i] \cup [w_i,x_i] \cup [x_i,u_{i+1}]\bigr)
\]
and
\[
[u_n,v_n] \cup [v_n,t_n] \cup [t_n,y_n] \cup [y_n,u_1].
\]
The position of $\mathscr{C}_{n}$ in the bounding torus of $\mathcal{D}_{\infty}$
is shown in Figure~\ref{figure:446dirichlet}.

\begin{prop} \label{prop:diskindirichlet}
In Figure~\ref{figure:445i}, the embedded curve $\mathscr{C}_{n}$ on the 
bounding torus of $\mathcal{D}_{\infty}$ bounds an essential disk in 
$\mathcal{D}_{\infty}$.
\end{prop}

\begin{proof}
The proof is similar to the first part of Proposition~7.8 in 
\cite{Acosta2019}.
\end{proof}

%We note that $C_n \cap A(C_n)$ is non-empty, 

%Consider the fundamental group of $s782$ with presentation as in Section \ref{section:intro}. 

Recall the isomorphism 
\[
\phi: \pi_1(\mathscr{M}_{4,4,\infty;\infty}) = \langle f_1, f_2, f_5 \rangle 
\longrightarrow \pi_1(s782),
\]
introduced in Section~\ref{section:forddomain}, where
\[
\phi(f_1) = c^{-1}a, \qquad 
\phi(f_2) = bc, \qquad \text{and} \qquad
\phi(f_5) = c.
\]  
Here, $f_2$ corresponds to $A = I_1 I_2$, and $f_5$ corresponds to  
$S^{-1}$, as described in Section~\ref{section:forddomain}.

The 3--manifold $s782$ has two cusps.  
For the second cusp $C_1$, a meridian--longitude system is given by
\[
(\mathcal{m}_1, \mathcal{l}_1) = (bc, ba).
\]  
Moreover, $s782$ admits the dihedral group $D_4$ as its symmetry group, which 
contains an involution exchanging the two cusps.  
Hence, the two cusps of $s782$ are symmetric.  
For the purpose of Theorem~\ref{thm:44np}, the meridian--longitude system is taken 
to be $(\mathcal{m}_1, \mathcal{l}_1)$ for the cusp $C_1$.

% The  2-cell decomposition of   the bounding %torus of $\partial_{\infty} D_{n}$ descends to a  2-cell decomposition of   the bounding torus of $\partial  W_{n}$.  The curve $C_{n}$ in  $\partial_{\infty} D_{n}$ descends to a curve in $\partial  W_{n}$, which bounds a disk in $W_{n}$.  We can identify $\partial  W_{n}$
%and one cusp of  $s782$ together with the 2-cell decompositions. From Proposition \ref{prop:diskinDirichlet},   the curve bounds a disk in    $W_{n}$ is identified with the curve 
%   in the first cusp  of  $s782$ which corresponds to $f_2 f^n_1$ (or $f^{-1}_2f^n_1$) bounds a disk in the solid tours $W_{n}$. 
% So it corresponds to slope $\mathcal{L}_1+(n-1) \mathcal{M}_1$ (or $-\mathcal{L}_1+(n+1) \mathcal{M}_1$).   This ends the proof of %Theorem \ref{thm:44np}. 

The 2--cell decomposition of the bounding torus of $\mathcal{D}_{\infty}$ 
descends naturally to a 2--cell decomposition of the bounding torus of $T_{\mathcal{D}}$.  
Similarly, the curve $\mathscr{C}_{n}$ in $\mathcal{D}_{\infty}$ descends to a curve in $T_{\mathcal{D}}$.  
Consequently, the corresponding curve in $T_{\mathcal{D}}$ bounds an essential disk in $W_{\mathcal{D}}$.  

We may identify $\partial W_{\mathcal{D}}$ with a cusp of $s782$, together with the 
2--cell decomposition.  
By Proposition~\ref{prop:diskindirichlet}, the curve bounding a disk in $W_{\mathcal{D}}$ 
corresponds to the curve in the first cusp of $s782$ represented by $f_2 f_1^n$, 
which bounds a disk in the solid torus $W_{\mathcal{D}}$.  
Therefore, this curve corresponds to the slope 
\[
\mathcal{l}_1 + (n-1)\,\mathcal{m}_1.
\]  
This completes the proof of Theorem~\ref{thm:44np}.

\begin{rem} \label{remark:slope}
From Sections~\ref{section:forddomain} and~\ref{section:dirichlet}, 
readers can easily observe that the 3--manifold at infinity of the even subgroup 
$\Delta_{4,4,n;\infty}$ is obtained from the 3--manifold $s782$ by Dehn filling.  
However, computing the precise slope of this Dehn filling is somewhat subtle, as 
illustrated in Theorem~\ref{thm:44np}.  
For instance, at first glance, both slopes
\[
-\mathcal{m}_1 + (n+1)\mathcal{l}_1 \quad \text{and} \quad 
\mathcal{m}_1 + (n-1)\mathcal{l}_1
\] 
might appear reasonable:

\begin{itemize}
\item The 3--manifold $s782$ is not a link complement in the 3--sphere, 
so there is no canonical meridian--longitude system.  
We therefore adopt the system provided by SnapPy.
\item We need to compare curves in three tori: the boundary of $\mathcal{D}_{\infty}$, 
the boundary of $T_{\mathcal{D}}$, and the boundary of 
$s782 - \mathbb{T}^2 \times (1,\infty)$.
\item In Section~\ref{section:group}, we compute the fundamental group of the 
3--manifold at infinity of the even subgroup $\Delta_{4,4,5;\infty}$ directly from 
Proposition~\ref{prop:diskindirichlet}.
\end{itemize}
\end{rem}

\section{The fundamental group of the 3-manifold at infinity of \texorpdfstring{$\Delta_{4,4,5;\infty}$}{}} 
\label{section:group}

Let $\mathscr{M}_{4,4,5;\infty}$ denote the 3--manifold at infinity of the 
even subgroup of $\Delta_{4,4,5;\infty}$.  
In this section, we identify the manifold $\mathscr{M}_{4,4,5;\infty}$ directly 
using the key Proposition~\ref{prop:diskindirichlet}.  
This section also serves as a double check of a small part of 
Theorem~\ref{thm:44np}~\ref{item:44np2}.  
Moreover, the method may be useful in other contexts, so we prefer to include it here.  
The argument involves lengthy but routine edge-cycle computations. 
We present the complete 2-cell decomposition, the side-pairings, and 
the resulting relations; we omit only the routine bookkeeping of 
vertex identifications, which are recorded on our scratch paper.

\begin{prop} \label{thm:445p}
The 3--manifold $\mathscr{M}_{4,4,5;\infty}$ is a 1--cusped hyperbolic 3--manifold, 
obtained from the 2--cusped 3--manifold $s782$ in SnapPy by Dehn filling on the 
second cusp along the slope
\[
4\mathcal{m}_1 + \mathcal{l}_1,
\] 
where $(\mathcal{m}_1, \mathcal{l}_1)$ denotes the meridian--longitude system 
provided by SnapPy.
\end{prop}

To prove Proposition~\ref{thm:445p}, we first consider the solid torus 
$T_{\mathcal{D}}$ with the 2--cell decomposition of its boundary 
$\partial T_{\mathcal{D}}$ described in Section~\ref{section:proof}.  
The 3--manifold $\mathscr{M}_{4,4,5;\infty}$ is obtained as the quotient of 
$T_{\mathcal{D}}$ by the side--pairings on $\partial T_{\mathcal{D}}$.  
We then cut $T_{\mathcal{D}}$ into a 3--ball along a disk, keeping track of the 
side--pairings, in order to compute the fundamental group of 
$\mathscr{M}_{4,4,5;\infty}$.

The local 2--cell decomposition of the torus $\partial \mathcal{D}_{\infty}$, 
equipped with the $\mathbb{Z}_n = \langle I_1 I_2 \rangle$--action, is identical 
to the 2--cell decomposition of the annulus $\partial \mathcal{F}_{\infty}$ 
with the $\mathbb{Z} = \langle I_1 I_2 \rangle$--action.  
We then take a fundamental domain of the $\langle I_1 I_2 \rangle$--action on 
the torus $\partial \mathcal{D}_{\infty}$, which is a compact annulus 
$\mathcal{A}$ pinched at the point $u_2$.  
See Figure~\ref{figure:445idomain}, which can be viewed as part of both 
Figures~\ref{figure:44iiford} and~\ref{figure:446dirichlet}.

Let $X=AS^{-1}A^{-1}$ and $Y=ASA^{-1}$.  
Since $\mathcal{D}_{\infty}$ covers $T_{\mathcal{D}}$, we use $\mathfrak{s}(S) \cap T_{\mathcal{D}}$ in $T_{\mathcal{D}}$ to denote the image of $\mathfrak{s}(S) \cap \mathcal{D}_{\infty} \cap \mathcal{A}$ to simplify notation.  
Thus, $\mathfrak{s}(S) \cap T_{\mathcal{D}}$  is a quadrangle and $\mathfrak{s}(ASA^{-1}) \cap T_{\mathcal{D}}$ is an octagon.
With these notations, in Figure \ref{figure:445idomain}, $\partial T_{\mathcal{D}}$ is a torus consisting of:
 \begin{itemize}[-]
    \item  two octagons of $\mathfrak{s}(S^{-1}) \cap T_{\mathcal{D}}$ and $\mathfrak{s}(ASA^{-1}) \cap T_{\mathcal{D}}$, denoted by $S^{-}$ and $Y$, respectively;

    %\item   the octagon   $\mathfrak{s}(ASA^{-1}) \cap T_{\mathcal{D}}$, we denote it by $Y$;

    \item two quadrangles of $\mathfrak{s}(AS^{-1}A^{-1}) \cap T_{\mathcal{D}}$ and $\mathfrak{s}(S) \cap T_{\mathcal{D}}$, denoted by $X$ and $S$, respectively;

    %\item   the quadrangle   $\mathfrak{s}(S) \cap T_{\mathcal{D}}$, we denote it by $S$;

    \item the two triangle component of $\mathfrak{s}(Q) \cap T_{\mathcal{D}}$, denoted by $Q_{+}$ with vertices $\{t_1,p_2,q_2\}$ and $Q_{-}$ with vertices $\{r_2,s_2,t_2\}$;

    \item one triangle $\mathfrak{s}(R) \cap T_{\mathcal{D}}$, we denote it by $R$;

\item    one triangle   $\mathfrak{s}(ARA^{-1}) \cap T_{\mathcal{D}}$. We denote it by $ARA^{-1}$.
\end{itemize}
In other words, we  have a 2--cell decomposition of $\partial T_{\mathcal{D}}$, 
and we denote the 2--cells by $X$, $S$, $R$, etc., even though these symbols 
also represent group elements.

Note that the projections of the octagonal components 
\[
\mathfrak{s}(AS^{-1}A^{-1}) \cap \partial \mathcal{D}_{\infty} \quad \text{and} \quad 
\mathfrak{s}(S^{-1}) \cap \partial \mathcal{D}_{\infty}
\] 
coincide in $\partial T_{\mathcal{D}}$.  
Here we use $\mathfrak{s}(AS^{-1}A^{-1}) \cap \partial \mathcal{D}_{\infty}$, 
as it provides a more convenient description of the boundary of the essential disk in 
$T_{\mathcal{D}}$.

By gluing together the left and right zigzag paths in Figure~\ref{figure:445idomain}, 
we obtain an annulus.  
Then, gluing the upper and lower paths (now circles) produces a torus.  
This torus is precisely the boundary of the solid torus $T_{\mathcal{D}}$ described 
in Section~\ref{section:proof}.

The side pairings on $\partial T_{\mathcal{D}}$ are given by
\begin{equation*}\label{sidepair:W5}
    \begin{aligned}
        b_1 = RA^{-1} &: ARA^{-1} \longrightarrow R,\\
        c_1 = Q &: Q^+ \longrightarrow Q^-,\\
        g = S^{-1}A^{-1} &: Y \longrightarrow S^-,\\
        h = AS^{-1} &: S \longrightarrow X.
    \end{aligned}
\end{equation*}

From these side pairings, we obtain the 3--manifold $\mathscr{M}_{4,4,5;\infty}$.  
The correspondences of the vertices of the polygons under these maps are shown in 
Table~\ref{vertices_corresponding445i}.

\begin{table}[htbp]
\centering
\caption{2--cells and their vertices under the side maps on $\partial T_{\mathcal{D}}$.}
\begin{tabular}{|c|c|c|}
\hline 
Maps & Vertices (domain) & Vertices (image) \\
\hline  
$b_1$ & $u_2, v_2, w_2$ & $u_2, x_1, y_1$ \\
\hline
$c_1$ & $t_1, p_2, q_2$ & $t_2, r_2, s_2$ \\
\hline
$g$ & $w_2, x_2, y_2, t_2, r_2, q_2, p_2, u_2$ & $v_1, w_1, x_1, u_2, s_2, r_2, q_2, t_1$ \\
\hline
$h$ & $t_1, v_1, u_1, s_1$ & $u_3, y_2, t_2, p_3$ \\
\hline
\end{tabular}
\label{vertices_corresponding445i}
\end{table}

However, the relationship between the boundary of an arbitrary essential proper disk in 
$T_{\mathcal{D}}$ and the 2--cell decomposition of $\partial T_{\mathcal{D}}$ described above 
is not ideal.  
By this, we mean that there appears to be no essential curve $\delta$ in 
$\partial T_{\mathcal{D}}$ such that:
\begin{itemize}
    \item $\delta$ bounds a disk in $T_{\mathcal{D}}$;
    \item the set of arcs $\delta \cap x$, for all 
    $x \in \{X, Y, S, S^{-}, R, AR^{-1}, Q_{+}, Q_{-}\}$ in the 2--cell decomposition 
    of $\partial T_{\mathcal{D}}$, is invariant under the side--pairing maps 
    $\{b_1, c_1, g, k\}$ on $\partial T_{\mathcal{D}}$.
\end{itemize}

This observation makes the proof of Proposition~\ref{thm:445p} somewhat delicate.  
Therefore, we need to subdivide the quadrangles and octagons in 
Figure~\ref{figure:445idomain} into smaller components in order to accommodate a proper 
essential disk in $T_{\mathcal{D}}$, see Figure~\ref{figure:445i}.

 \begin{figure}[htb]
	\begin{center}
		\begin{tikzpicture}
		\node at (0,0) {\includegraphics[width=14cm,height=10cm]{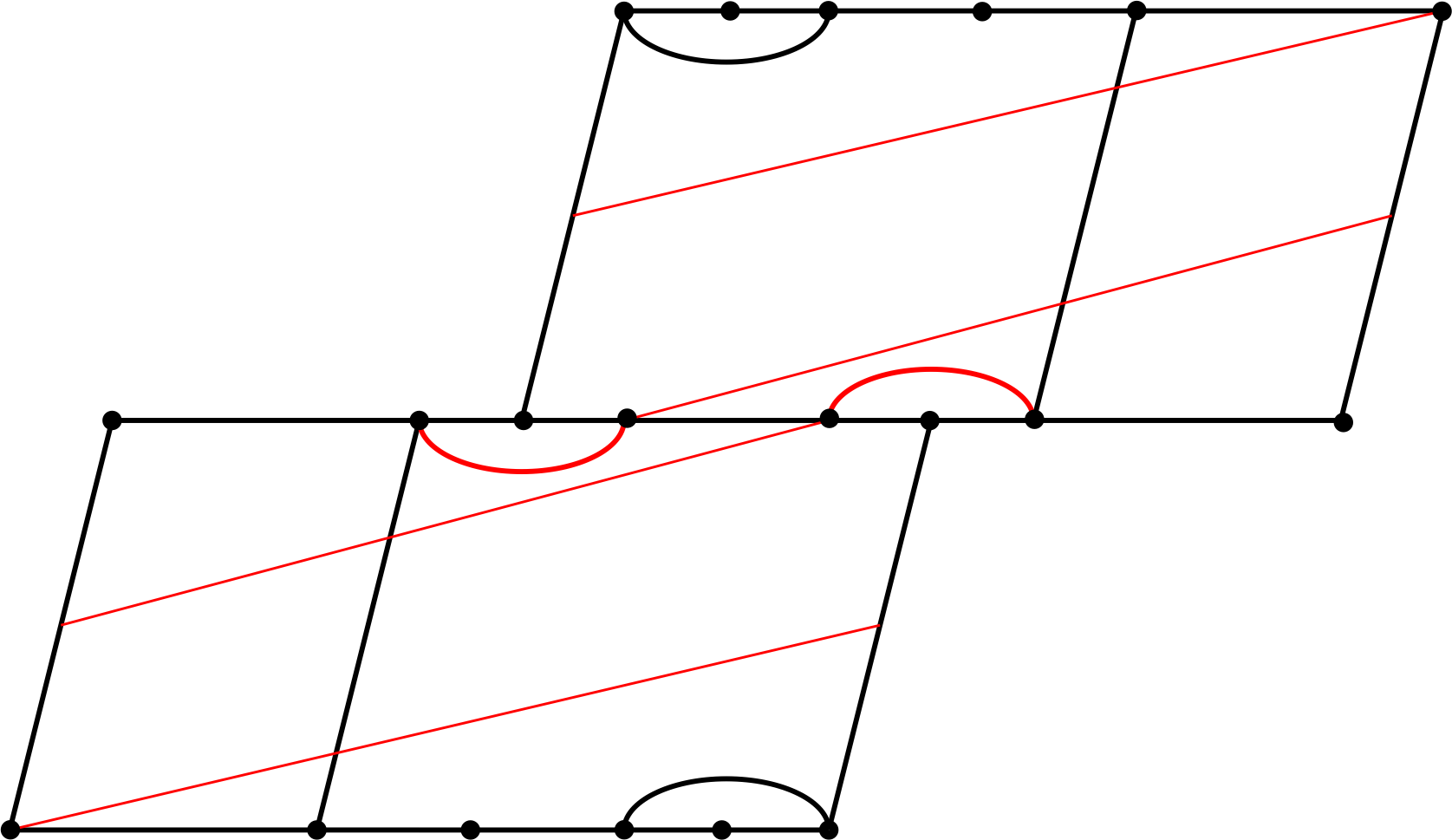}};

    \node at (-1.3,5.25){\large $u_2$};

    \node at (-0,5.25){\large $v_2$};

    \node at (1,5.25){\large $w_2$};

    \node at (2.2,5.25){\large $x_2$};

    \node at (4.2,5.25){\large $y_2$};
    
    \node at (7.0,5.25){\large $u_3$};
    \node at (-7.0,-5.3){\large $u_1$};

    \node at (-4.0,-5.3){\large $v_1$};

    \node at (-2.6,-5.3){\large $w_1$};
    \node at (-1.0,-5.3){\large $x_1$};
    \node at (0.0,-5.3){\large $y_1$};
    
    \node at (1.3,-5.3){\large $u_2$};

    \node at (-5.7,0.4){\large $s_1$};
    \node at (-3.2,0.4){\large $t_1$};
       \node at (-2.2,0.4){\large $p_2$};

 \node at (2.2,-0.4){\large $s_2$};
    \node at (3.2,-0.4){\large $t_2$};
        \node at (5.7,-0.4){\large $p_3$};
    \node at (1.3,2.3){\large $X$};
    \node at (4.8,3.0){\large $Y$};
    \node at (-4.8,-3.0){\large $S$};
    \node at (-1.3,-2.3){\large $S^-$};

    \node at (0.3,-4.6){\small $R$};
    \node at (-0.2,4.65){\small $ARA^{-1}$};
    \node at (-1.8,-0.3){\small $Q_+$};
    \node at (1.5,0.3){\small $Q_-$};
\end{tikzpicture}
\end{center}
\caption{A fundamental domain $\mathcal{A}$ of the $\langle I_1I_2\rangle $-action on the 
torus $\partial \mathcal{D}_{\infty}$, which projects to the torus $\partial T_{\mathcal{D}}$. 
This provides a  2-cell decomposition of $\partial T_{\mathcal{D}}$ together  with the boundary 
curve (in red) of an essential disk in $T_{\mathcal{D}}$. }
\label{figure:445idomain}
\end{figure}
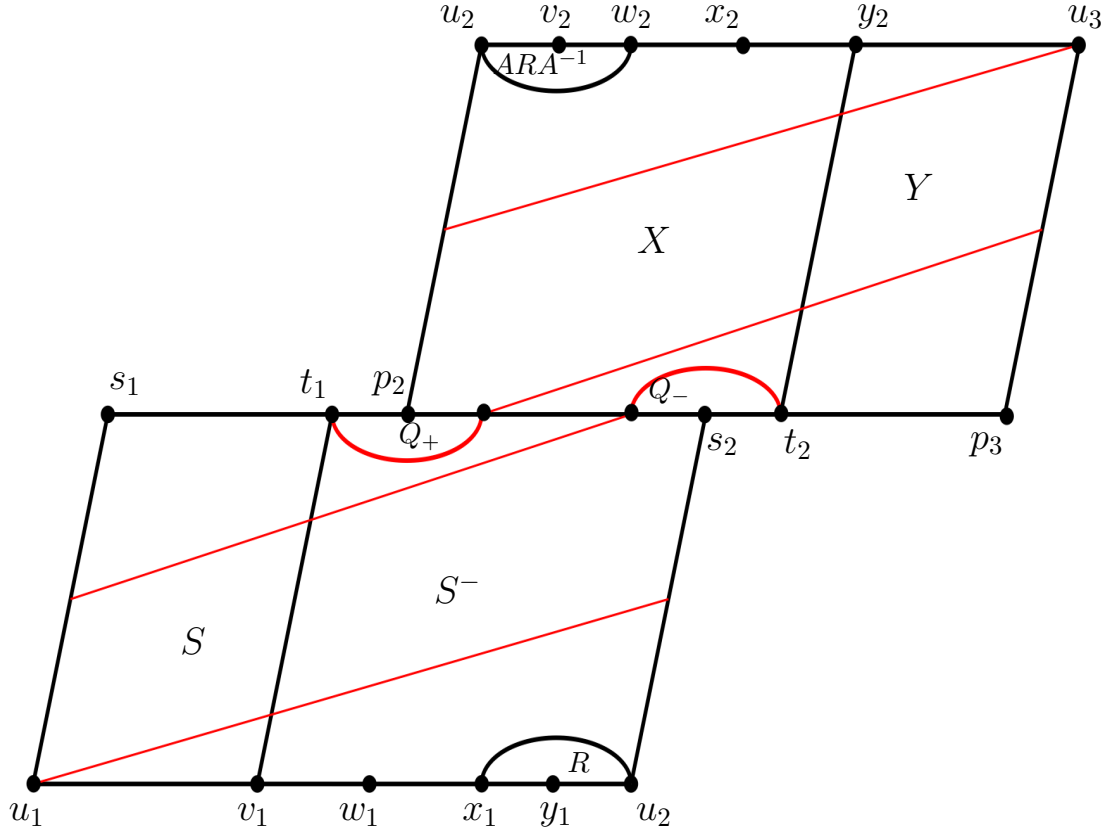

When we glue together the left and right zigzag paths, and then the upper and lower paths 
in Figure~\ref{figure:445i}, we obtain the torus $\partial T_{\mathcal{D}}$.  
According to Proposition~\ref{prop:diskindirichlet}, the red arcs in 
Figure~\ref{figure:445i} are glued together to form a simple closed curve $c_5$ 
in $\partial T_{\mathcal{D}}$.  
This curve $c_5$ bounds a disk $D$ in $T_{\mathcal{D}}$.  
Cutting $T_{\mathcal{D}}$ along $D$ produces a 3--ball $N$.

The 3-manifold  $\mathscr{M}_{4,4,5;\infty}$  is obtained as the quotient of 
the 3--ball $N$.  
Figure~\ref{figure:445i} provides a 2--cell decomposition of $\partial N$.

\begin{itemize}
    \item The red arcs decompose the octagon labeled $S^{-}$ in 
    Figure~\ref{figure:445idomain} into three parts;

    \item The red arcs decompose the octagon labeled $Y$ in 
    Figure~\ref{figure:445idomain} into three parts;

    \item There is a side--pairing $g$ between these octagons.  
    Under this pairing, the red arcs in $S^{-}$ are mapped to the three blue arcs 
    in $Y$, and the red arcs in $Y$ are mapped to the three blue arcs in $S^{-}$;

    \item In total, each octagon is subdivided into eight smaller parts by the red 
    and blue arcs. We denote them by $Y_i$ and $S^{-}_i$ for $i = 1, 2, \dots, 8$;

    \item Similarly, each quadrangle labeled $X$ and $S$ is subdivided into nine 
    smaller parts, denoted by $X_i$ and $S_i$ for $i = 1, 2, \dots, 9$.
\end{itemize} 
Additionally, there are thirty-eight new star-shaped vertices in Figure~\ref{figure:445i} 
compared with Figure~\ref{figure:445idomain}.  
For instance, the three new vertices along the arc $[u_2, v_2]$ 
(an edge of the triangle labeled $ARA^{-1}$) arise from the intersections 
between $[x_1, u_2]$ and the blue arcs in Figure~\ref{figure:445i}.  
We will not label these star-shaped vertices, but several edges in 
Figure~\ref{figure:445i} will be labeled.

We remark that Figure~\ref{figure:445i} is a topological illustration.  
Nevertheless, it is straightforward to verify that there exist geometric realizations 
of the red arcs within the 2--cells $X$, $Y$, $S$, and $S^{-}$ such that their images 
under $c_1$, $g$, and $h$ intersect the red arcs transversely and minimally.  
In other words, one can realize the red and blue arcs in Figure~\ref{figure:445i} 
geometrically without any unnecessary bigon complements.  
Hence, the geometric realization preserves the same intersection pattern as depicted 
in Figure~\ref{figure:445i}.

We can then assemble the small disks in Figure~\ref{figure:445i} together with 
$D_+$ and $D_-$ into a 2--sphere, which provides a 2-cell decomposition of $\partial N$.  
This decomposition induces side-pairings on $\partial N$.  

In particular, the side-pairing 
\[
g = S^{-1}A^{-1}: Y \longrightarrow S^{-}
\] 
in $\partial T_{\mathcal{D}}$ splits into eight side-pairings on $\partial N$.  
Similarly, the side-pairing 
\[
h = AS^{-1}: S \longrightarrow X
\] 
in $\partial T_{\mathcal{D}}$ splits into nine side-pairings on $\partial N$.

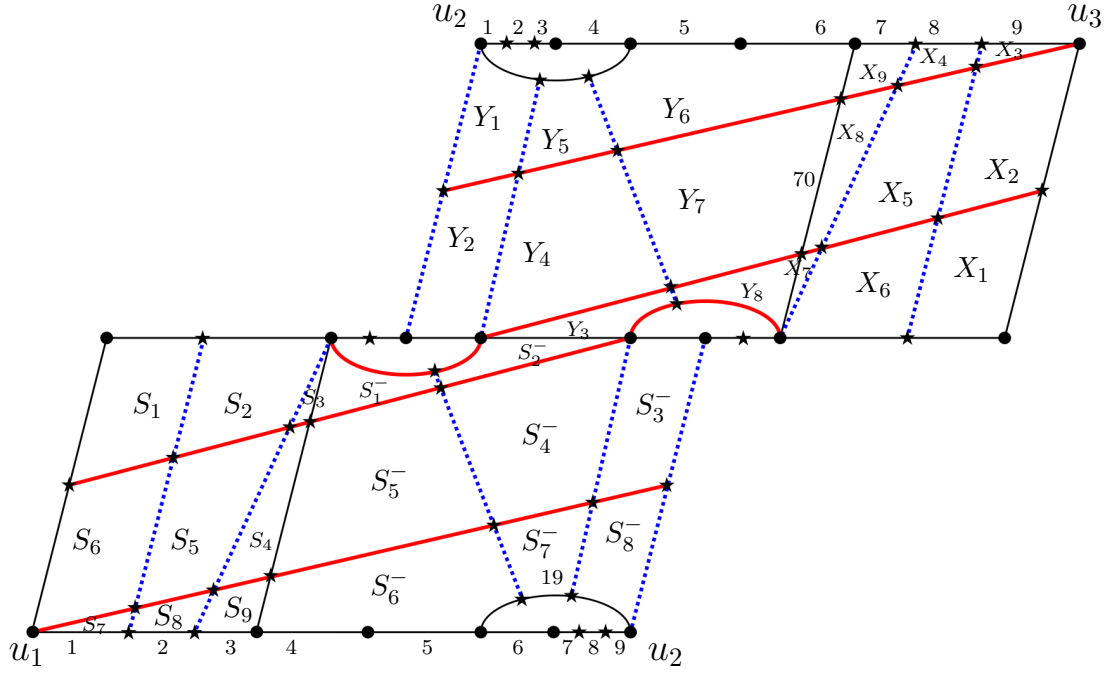
\begin{figure}[htbp]
\begin{center}
\begin{tikzpicture}
\node at (0,0) {\includegraphics[width=14cm,height=8cm]{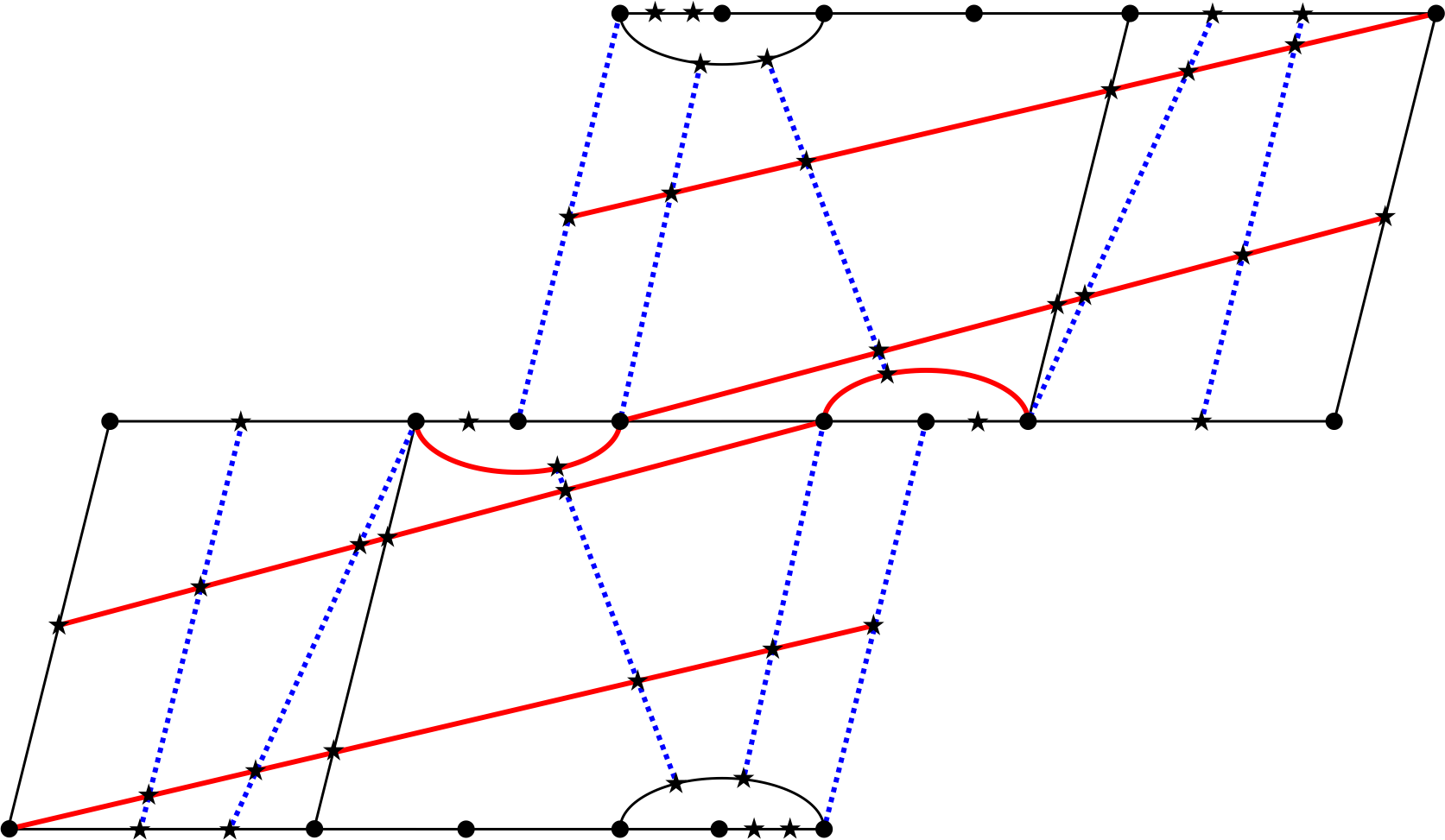}};
    \node at (-1.4,4.25){\large $u_2$};
    \node at (7.0,4.25){\large $u_3$};
    \node at (-7.0,-4.2){\large $u_1$};
    \node at (1.45,-4.2){\large $u_2$};
    
\node at (-0.9,2.9){\small $Y_1$};
\node at (-1.25,1.3){\small $Y_2$};
\node at (0.30,0.1){\tiny $Y_3$};
\node at (-0.25,1.1){\small$Y_4$};
\node at (-0.00,2.6){\small$Y_5$};
\node at (1.6,3.0){\small$Y_6$};
\node at (1.8,1.8){\small$Y_7$};
    \node at (2.6,0.6){\tiny $Y_8$};

    \node at (5.5,0.9){\small $X_1$};
 \node at (5.9,2.2){\small $X_2$};
  \node at (6.0,3.8){\tiny $X_3$};
  \node at (5.0,3.7){\tiny $X_4$};
   \node at (4.5,1.9){\small $X_5$};
   \node at (4.2,0.7){\small $X_6$};
   \node at (3.2,0.9){\tiny $X_7$};
   \node at (3.9,2.7){\tiny $X_8$};
   \node at (4.2,3.5){\tiny $X_9$};

\node at (-5.4,-0.9){\small $S_1$};
\node at (-4.2,-0.9){\small $S_2$};
\node at (-3.2,-0.8){\tiny $S_3$};
\node at (-3.9,-2.7){\tiny $S_4$};
\node at (-4.9,-2.7){\small $S_5$};
\node at (-6.2,-2.7){\small $S_6$};
\node at (-6.1,-3.8){\tiny $S_7$};
\node at (-5.12,-3.7){\small $S_8$};
\node at (-4.2,-3.6){\small $S_9$};

\node at (-2.4,-0.7){\tiny $S^{-}_1$};
\node at (-0.3,-0.2){\tiny $S^{-}_2$};
\node at (1.3,-0.9){\small $S^{-}_3$};
\node at (-0.2,-1.3){\small$S^{-}_4$};

\node at (-2.2,-1.9){\small$S^{-}_5$};
\node at (-2.2,-3.3){\small$S^{-}_6$};

    \node at (-0.2,-2.7){\small$S^{-}_7$};
    \node at (0.9,-2.6){\small$S^{-}_8$};

     \node at (-0.9,4.1){\tiny$1$};
      \node at (-0.5,4.1){\tiny$2$};
       \node at (-0.18,4.1){\tiny$3$};
        \node at (0.5,4.1){\tiny$4$};
         \node at (1.7,4.1){\tiny$5$};
          \node at (3.5,4.1){\tiny$6$};
           \node at (4.3,4.1){\tiny$7$};
            \node at (5.0,4.1){\tiny$8$};
             \node at (6.1,4.1){\tiny$9$};

    \node at (-6.4,-4.1){\tiny$1$};
      \node at (-5.2,-4.1){\tiny$2$};
       \node at (-4.3,-4.1){\tiny$3$};
        \node at (-3.5,-4.1){\tiny$4$};
          \node at (-1.7,-4.1){\tiny$5$};
            \node at (-0.5,-4.1){\tiny$6$};
              \node at (0.15,-4.1){\tiny$7$};
                \node at (0.5,-4.1){\tiny$8$};
                  \node at (0.85,-4.1){\tiny$9$};

                    \node at (3.28,2.1){\tiny$70$};

                      \node at (-0.05,-3.17){\tiny$19$};
     
\end{tikzpicture}
\end{center}
\caption{The 2-cell decomposition of the 2-sphere boundary of $N$ to get the 3-manifold $\mathscr{M}_{4,4,5; \infty}$. When calculating edge circles, the edge labeled by $i$ is referred to as edge $e_i$. }
\label{figure:445i}
\end{figure}

Based on the above analysis and in comparison with the side-pairings on 
$\partial T_{\mathcal{D}}$, the side-pairings on $\partial N$ are as follows:
\begin{equation*}\label{sidepair:M5}
    \begin{aligned}
        b_1 = RA^{-1} &: ARA^{-1} \longrightarrow R;\\
        c_1 = Q &: Q^+ \longrightarrow Q^-;\\
        g_i = S^{-1}A^{-1} &: Y_i \longrightarrow S^{-}_i, \quad i = 1, 2, \dots, 8;\\
        h_j = AS^{-1} &: S_j \longrightarrow X_j, \quad j = 1, 2, \dots, 9;\\
        k &: D^+ \longrightarrow D^-.
    \end{aligned}
\end{equation*}

We then have Table~\ref{edgecircle445i}, which lists all the edge circles 
and their corresponding relations for the 3-manifold $\mathscr{M}_{4,4,5; \infty}$.  
As an example, consider the edge circle
\[
e_2 \xrightarrow{h_8} e_{70} 
      \xrightarrow{g_7} e_{19} 
      \xrightarrow{b^{-1}} e_2.
\]
This edge circle yields the relation $b_1^{-1} g_7 h_8$, 
as shown in the second row of the left subtable in Table~\ref{edgecircle445i}.  
All edge labels are recorded on our scratch paper; readers may ignore labels with 
large numbers, such as $e_{48}$.

\begin{table}[ht]
	\centering
	\caption{The cycle relations for the 3-manifold $\mathscr{M}_{4,4,5; \infty}$.}
	\begin{adjustbox}{minipage=0.31\textwidth, valign=t}
		\centering
		\begin{tabular}{c|c}
			\toprule
			\textbf{edge} & \textbf{cycle relation}\\
			\midrule
			$e_{1}$ & $b_1^{-1}g_8h_7$ \\  [1 ex]
			$e_{2}$ & $b_1^{-1}g_7h_8$ \\  [1 ex]
			$e_{3}$ & $b_1^{-1}g_6h_9$ \\  [1 ex]
			$e_{4}$ & $b_1^{-1}g_6^{-2}$ \\  [1 ex]
			$e_{7}$ & $h_9g_6b_1^{-1}$ \\  [1 ex]
			$e_{8}$ & $h_4g_5b_1^{-1}$ \\  [1 ex]
			$e_{9}$ & $h_3g_1b_1^{-1}$ \\  [1 ex]
			$e_{10}$ & $g_1^{-1}kc_1^{-1}h_2^{-1}$ \\  [1 ex]
			$e_{11}$ & $g_2^{-1}kc_1^{-1}h_1^{-1}$ \\  [1 ex]
			$e_{12}$ & $c_1^{-1}kg_3^{-1}h_1^{-1}$ \\  [1 ex]
			$e_{13}$ & $c_1^{-1}kg_8^{-1}h_6^{-1}$ \\  [1 ex]
			$e_{21}$ & $g_6k^{-1}g_7^{-1}$ \\ 
			\bottomrule
		\end{tabular}
	\end{adjustbox}
	\hfill
	\begin{adjustbox}{minipage=0.31\textwidth, valign=t}
		\centering
		\begin{tabular}{c|c}
			\toprule
			\textbf{edge} & \textbf{cycle relation}\\
			\midrule
			$e_{22}$ & $g_7k^{-1}g_8^{-1}$ \\  [1 ex]
			$e_{23}$ & $g_3g_8^{-1}k$ \\  [1 ex]
			$e_{24}$ & $g_4g_7^{-1}k$ \\  [1 ex]
			$e_{25}$ & $g_5k^{-1}g_4^{-1}$ \\  [1 ex]
			$e_{26}$ & $g_4k^{-1}g_3^{-1}$ \\  [1 ex]
			$e_{27}$ & $g_2g_4^{-1}k$ \\  [1 ex]
			$e_{30}$ & $g_1g_5^{-1}k$ \\  [1 ex]
			$e_{31}$ & $g_1k^{-1}g_2^{-1}$ \\  [1 ex]
			$e_{33}$ & $h_2^{-1}kh_3$ \\  [1 ex]
			$e_{34}$ & $h_1^{-1}kh_2$ \\  [1 ex]
			$e_{35}$ & $h_1^{-1}h_6k$ \\
			\bottomrule
		\end{tabular}
	\end{adjustbox}
	\hfill
	\begin{adjustbox}{minipage=0.31\textwidth, valign=t}
		\centering
		\begin{tabular}{c|c}
			\toprule
			\textbf{edge} & \textbf{cycle relation}\\
			\midrule
			$e_{36}$ & $h_2^{-1}h_5k$ \\  [1 ex]
			$e_{37}$ & $h_3^{-1}h_4k$ \\  [1 ex]
			$e_{39}$ & $h_5^{-1}kh_4$ \\  [1 ex]
			$e_{40}$ & $h_6^{-1}kh_5$ \\  [1 ex]
			$e_{41}$ & $h_6^{-1}h_7k$ \\  [1 ex]
			$e_{42}$ & $h_7^{-1}kh_8$ \\  [1 ex]
			$e_{43}$ & $h_8^{-1}kh_9$ \\  [1 ex]
			$e_{44}$ & $h_5^{-1}h_9k$ \\  [1 ex]
			$e_{45}$ & $h_4^{-1}h_9k$ \\  [1 ex]
			$e_{47}$ & $g_5g_6^{-1}k$ \\  [1 ex]
			$e_{48}$ & $g_2^{-1}g_3^{-1}c_1$ \\
			\bottomrule
		\end{tabular}
	\end{adjustbox}
	\label{edgecircle445i}
\end{table}

Using Magma, we can simplify the presentation to obtain
\[
\pi_1(\mathscr{M}_{4,4,5; \infty}) 
= \left\langle x_1, x_2 \ \bigg| \ x_1^{-1} x_2^4 x_1^{-1} x_2^4 x_1^{-1} x_2^{-1} x_1^3 x_2^{-1} \right\rangle,
\]
where $x_1 = g_6$ and $x_2 = k$.  

Let $L$ be the 3-manifold in the SnapPy census \cite{CullerDunfield:2014} obtained from $s782$ by Dehn filling along the slope $4\mathcal{m}_1 + \mathcal{l}_1$.  
It is hyperbolic with volume $4.7517019655$, and
\[
\pi_1(L) = \left\langle a, b \ \big| \ a^{-1} b^4 a^{-1} b^4 a^{-1} b^{-1} a^3 b^{-1} \right\rangle.
\]  

It is now clear that there is an isomorphism between 
$\pi_1(\mathscr{M}_{4,4,5; \infty})$ and $\pi_1(L)$.  
By the prime decomposition theorem for 3-manifolds \cite{Hempel:2004}, 
$\mathscr{M}_{4,4,5; \infty}$ is the connected sum of $L$ with a closed 3-manifold $L'$ having trivial fundamental group.  
By the solution of the Poincar\'e conjecture\cite{Perelman2002, Perelman2003a, Perelman2003b}, $L'$ is homeomorphic to the 3-sphere.  
Thus, $\mathscr{M}_{4,4,5; \infty}$ is homeomorphic to $L$, completing the proof of Proposition~\ref{thm:445p}.

\end{document}